\documentclass[12pt]{article}
\usepackage{amsfonts, amsmath, amssymb, latexsym, eucal, amscd} 

\usepackage{amsthm}
\usepackage{amscd}

\usepackage{cite}

\newtheorem{definition}{Definition}[section]
\newtheorem{theorem}[definition]{Theorem}
\newtheorem{lemma}[definition]{Lemma}
\newtheorem{corollary}[definition]{Corollary}
\newtheorem{remark}[definition]{Remark}
\newtheorem{example}[definition]{Example}
\newtheorem{conjecture}[definition]{Conjecture}
\newtheorem{problem}[definition]{Problem}
\newtheorem{note}[definition]{Note}
\newtheorem{assumption}[definition]{Assumption}
\newtheorem{proposition}[definition]{Proposition}

\begin{document} 

\title{\bf The Norton-balanced condition for \\
$Q$-polynomial distance-regular graphs
}
\author{Kazumasa Nomura and 
Paul Terwilliger }
\date{}

\maketitle
\begin{abstract} Let $\Gamma$ denote a $Q$-polynomial distance-regular graph, with vertex set $X$ and diameter $D\geq 3$.
 The standard module $V$ has a basis $\lbrace {\hat x} \vert x \in X\rbrace$,
where ${\hat x}$ denotes column $x$ of the identity matrix $I \in {\rm Mat}_X(\mathbb C)$.
Let $E$ denote a $Q$-polynomial primitive idempotent of $\Gamma$. The eigenspace  $EV$ is spanned by the vectors
 $\lbrace E {\hat x} \vert x \in X\rbrace$.
 It was previously known that these vectors
satisfy a condition  called the balanced set condition.
In this paper, we introduce a variation on the balanced set condition
called the Norton-balanced condition. The Norton-balanced  condition involves the Norton algebra product on $EV$.
We define $\Gamma$ to be Norton-balanced whenever 
$\Gamma$ has a $Q$-polynomial primitive idempotent $E$ such that the set $\lbrace E {\hat x} \vert x \in X\rbrace$ is Norton-balanced.
 We show that $\Gamma$ is Norton-balanced in the following cases:
 (i)
 $\Gamma$ is bipartite; (ii) $\Gamma$ is almost bipartite;
(iii) $\Gamma$ is  dual-bipartite; (iv) $\Gamma$ is almost dual-bipartite; (v) $\Gamma$ is tight; (vi) 
 $\Gamma$ is a Hamming graph; (vii) $\Gamma$ is a Johnson graph; (viii) $\Gamma$ is the Grassmann graph $J_q(2D,D)$; 
 (ix) $\Gamma$ is a halved bipartite dual-polar graph; (x) $\Gamma$ is a halved Hemmeter graph; (xi) $\Gamma$ is a halved hypercube; (xii) $\Gamma$ is a folded-half hypercube;
 (xiii) $\Gamma$ has $q$-Racah type and affords a spin model. 
 Some  theoretical results about the Norton-balanced condition are obtained, and some open problems are given.

 \bigskip

 \end{abstract}
 
\section{Introduction}
This paper is about a family  of finite undirected graphs, said to be distance-regular  \cite{bcn}. We will investigate
a type of distance-regular graph, called $Q$-polynomial \cite[Chapter~8]{bcn}.
The $Q$-polynomial property was introduced in 1973 by Delsarte \cite{delsarte} in his work on coding theory and design theory.
Since that beginning, the  property has been linked to many  topics, such as 
orthogonal polynomials
\cite[p.~260]{banIto},
\cite{leonard, 2LT};
spin models
\cite{CW, curtNom, CNhom, nomSpinModel, nomTerSM};
$q$-deformed enveloping algebras
\cite{qRacTet, qtet, altDRG};
partially ordered sets
\cite{twist2, posetDRG, twist, rnp, unif2, unif1, grass3};
tridiagonal pairs
\cite{cerzo, someAlg, augIto};
free fermions
\cite{fermion2, fermion1, fermion3};
and the double affine Hecke algebra 
\cite{jh1,jh2, jh3}.
Comprehensive treatments of the $Q$-polynomial property can be found in \cite{bbit, banIto, bcn, dkt, int}.
\medskip

\noindent
For a $Q$-polynomial distance-regular graph, the adjacency matrix has a distinguished primitive idempotent called a $Q$-polynomial idempotent.
There is a  characterization of the $Q$-polynomial primitive idempotents, called the balanced set characterization  \cite[Theorem~1.1]{balanced}.
Over  the next few paragraphs, we will describe this characterization in order to motivate our main topic.
\medskip

\noindent
 Throughout this section,  $\Gamma=(X, \mathcal R)$ denotes a distance-regular graph 
with vertex set $X$, adjacency relation $\mathcal R$, and diameter $D\geq 3$ (formal definitions start in Section 2).
\medskip

\noindent
 Let $V$ denote the $\mathbb R$-vector space  consisting of the column vectors with coordinates indexed by $X$ and all entries in $\mathbb R$.
 The vector space  $V$ becomes a Euclidean space with bilinear form $\langle u,v \rangle = u^t v $ for $u,v \in V$.
 For $x \in X$ define a vector ${\hat x} \in V$ that has $x$-coordinate $1$ and all other coordinates $0$. The vectors $\lbrace {\hat x} \vert x \in X\rbrace$ form
 an orthonormal basis for $V$.  The adjacency matrix $A$ of $\Gamma$ acts on $V$.
\medskip

\noindent Let $E$ denote a primitive idempotent of $\Gamma$. The matrix $E$ is the orthogonal projection onto
the eigenspace $EV$ of $A$. By construction, the subspace $EV$ is spanned by the vectors $\lbrace E{\hat x} \vert x \in X\rbrace$.
\medskip

\noindent
Let $\partial$ denote the path-length distance function for $\Gamma$.
For $x \in X$ and $0 \leq i \leq D$, define the set
$\Gamma_i(x) = \lbrace y \in X\vert \partial(x,y)=i \rbrace$.
According to the balanced set characterization  \cite[Theorem~1.1]{balanced}, $E$ is $Q$-polynomial if and only if the following (i), (ii) hold:
\begin{enumerate}
\item[\rm (i)] the vectors $\lbrace E{\hat x} \vert x \in X\rbrace$ are mutually distinct;
\item[\rm (ii)]  for $x,y \in X$ and $0 \leq i,j\leq D$,
 \begin{align*}
 \sum_{z \in \Gamma_i(x) \cap \Gamma_j(y)} E {\hat z} -
  \sum_{z \in \Gamma_j(x) \cap \Gamma_i(y)} E {\hat z} 
 \in {\rm Span}\lbrace E{\hat x} - E {\hat y}\rbrace.
 \end{align*}
\end{enumerate}
\noindent For the rest of this section, assume that $E$ is $Q$-polynomial. We mention a special case of the balanced set dependency.
Pick $x, y \in X$ and write $i = \partial (x,y)$. Define
\begin{align*}
 x^-_y  = \sum_{z \in \Gamma(x) \cap \Gamma_{i-1}(y)}  {\hat z},
 \qquad \qquad x^+_y  = 
\sum_{z \in \Gamma(x) \cap \Gamma_{i+1}(y) }  {\hat z},
 \end{align*}
\noindent where  $\Gamma(x)= \Gamma_1(x)$ and $\Gamma_{-1}(x)= \emptyset=
 \Gamma_{D+1}(x)$. Then
\begin{align*}
E x^-_y - E y^-_x \in {\rm Span} \lbrace E{\hat x} - E {\hat y}\rbrace, \qquad \qquad
E x^+_y - E y^+_x \in {\rm Span} \lbrace E{\hat x} - E {\hat y}\rbrace.
\end{align*}
\noindent The vectors $\lbrace E{\hat x} \vert x \in X\rbrace$ satisfy another type of linear dependency, known as  the symmetric balanced set dependency \cite[Theorem~2.6]{kiteFree}.
 Let $x,y \in X$ and $0 \leq i,j\leq D$. According to the symmetric balanced set dependency,
\begin{align*}
 \sum_{z \in \Gamma_i(x) \cap \Gamma_j(y)} E {\hat z} +
  \sum_{z \in \Gamma_j(x) \cap \Gamma_i(y)} E {\hat z} 
 \in {\rm Span}\lbrace  Ex^-_y + E y^-_x,  E x^+_y + E y^+_x, E{\hat x} + E {\hat y}\rbrace.
\end{align*}
Comparing the balanced set dependency with its symmetric version, we find that
for $x,y \in X$ and $0 \leq i,j\leq D$,
\begin{align*}
 \sum_{z \in \Gamma_i(x) \cap \Gamma_j(y)} E {\hat z} 
 \in {\rm Span}\lbrace  Ex^-_y,  E x^+_y, E{\hat x}, E {\hat y} \rbrace.
\end{align*}
\noindent It could happen that for all $x,y \in X$ the vectors $Ex^-_y$,  $E x^+_y$, $E{\hat x}$, $E {\hat y}$ are linearly dependent. We now consider some situations where this occurs.
\medskip

\noindent The set of vectors  $\lbrace E{\hat x} \vert x \in X\rbrace$ is called strongly balanced  \cite[Section~2]{stronglybalanced}
whenever for all $x,y \in X$ and $0 \leq i,j\leq D$,
 \begin{align*}
 \sum_{z \in \Gamma_i(x) \cap \Gamma_j(y)} E {\hat z} 
 \in {\rm Span}\lbrace E{\hat x}, E {\hat y}\rbrace.
 \end{align*}

\noindent According to  \cite[Theorems~1, 3]{stronglybalanced} the following are equivalent:
\begin{enumerate}
\item[\rm (i)] the set $\lbrace E{\hat x} \vert x \in X\rbrace$ is strongly balanced;
\item[\rm (ii)] $E$ is dual-bipartite or almost dual-bipartite (see Section 3 below).
\end{enumerate}
\medskip

\noindent  We now recall the Norton algebra structure on $EV$  \cite[Proposition~5.2]{norton}. 
For $u \in V$ and $x \in X$ let $u_x$ denote the $x$-coordinate of $u$. So $u = \sum_{x \in X} u_x \hat x$.
 For $u, v \in V$ define a vector 
$u \circ v = \sum_{x \in X} u_x v_x \hat x$.
The Norton algebra consists of the $\mathbb R$-vector space $EV$, together with the product
\begin{align*}
u \star v = E (u \circ v) \qquad \qquad (u,v \in EV).   
\end{align*}
The Norton product $\star$ is commutative, and nonassociative in general.
\medskip

\noindent We now introduce the Norton-balanced condition.
The set of vectors
$\lbrace E{\hat x} \vert x \in X\rbrace$ is called Norton-balanced
whenever for all $x,y \in X$ and $0 \leq i,j\leq D$,
 \begin{align*}
 \sum_{z \in \Gamma_i(x) \cap \Gamma_j(y)} E {\hat z} 
 \in {\rm Span}\lbrace E{\hat x}, E {\hat y}, E{\hat x} \star E{\hat y}\rbrace.
 \end{align*}

\noindent Let us clarify the Norton-balanced condition. By our above comments, the following are equivalent:
\begin{enumerate}
\item[\rm (i)] the set $\lbrace E{\hat x} \vert x \in X\rbrace$ is Norton-balanced;
\item[\rm (ii)]  for all $x,y \in X$ we have
$ Ex^-_y, Ex^+_y \in {\rm Span}\lbrace E{\hat x}, E{\hat y}, E{\hat x} \star E{\hat y}\rbrace $.
\end{enumerate}
We say that $\Gamma$ is Norton-balanced whenever $\Gamma$ has a $Q$-polynomial primitive idempotent $E$ such that the set 
$\lbrace E{\hat x} \vert x \in X\rbrace$ is Norton-balanced.
\medskip

\noindent  Next, we describe  our results. We have two kinds of results; some are about examples, and some are more theoretical. 
We first  describe the results about examples. This will be done over the next four paragraphs.
\medskip

\noindent Assume that $\Gamma$ is $Q$-polynomial. Using some elementary arguments, we  show that $\Gamma$ 
is Norton-balanced in the following cases:
 (i) $\Gamma$ is bipartite; (ii) $\Gamma$ is almost bipartite;
(iii) $\Gamma$ is  dual-bipartite; (iv) $\Gamma$ is almost dual-bipartite; (v) $\Gamma$ is tight.
\medskip

\noindent The combinatorial structure of $\Gamma$ is described by some well-known parameters called the intersection numbers. We show
that in general, $\Gamma$ being Norton-balanced is not a condition on the intersection numbers alone. To do this, we consider the Hamming
graph $H(D,4)$ and a Doob graph with diameter $D$. These graphs have the same intersection
numbers. We show that $H(D,4)$ is Norton-balanced and the Doob graph is not.
\medskip

\noindent  The book \cite[Chapter~6.4]{bbit} gives a list of the known infinite families of $Q$-polynomial distance-regular graphs with unbounded diameter.
For each listed graph, every $Q$-polynomial structure is described. We 
examine these $Q$-polynomial structures. For each listed graph $\Gamma=(X,\mathcal R)$ and each $Q$-polynomial primitive idempotent $E$ of $\Gamma$, we determine
if the set  $\lbrace E{\hat x} \vert x \in X\rbrace$ is Norton-balanced or not.
In summary form,  our conclusion is that $\Gamma$ 
is Norton-balanced in the following cases:
 (vi) 
 $\Gamma$ is a Hamming graph; (vii) $\Gamma$ is a Johnson graph; (viii) $\Gamma$ is the Grassmann graph $J_q(2D,D)$; 
 (ix) $\Gamma$ is a halved bipartite dual-polar graph; (x) $\Gamma$ is a halved Hemmeter graph; (xi) $\Gamma$ is a halved hypercube; (xii) $\Gamma$ is a folded-half hypercube.
\medskip

\noindent The Norton-balanced condition was inspired by our recent work with Nomura on spin models \cite{nomTerSM}. We show that  $\Gamma$ is Norton-balanced in the following case: (xiii) $\Gamma$ has $q$-Racah type and affords a spin model.
\medskip

\noindent  We will  describe our theoretical results after a definition and some comments.
\medskip

\noindent 
We define $\Gamma$ to be reinforced whenever the following (i), (ii) hold for $2 \leq i \leq D$:
\begin{enumerate}
\item[\rm (i)] for $x,y \in X$ at distance $\partial(x,y)=i$, the average valency of the induced subgraph $\Gamma(x) \cap \Gamma_{i-1}(y)$ is independent of $x$ and $y$;
\item[\rm (ii)]  for $x,y \in X$ at distance $\partial(x,y)=i-1$, the average valency of the induced subgraph $\Gamma(x) \cap \Gamma_{i}(y)$ is independent of $x$ and $y$.
\end{enumerate}
If $\Gamma$ is distance-transitive then $\Gamma$ is reinforced. Assume for the moment that $\Gamma$ is reinforced.
For $2 \leq i \leq D$ let $z_i$ denote the average valency mentioned in (i), and note that $a_1-z_i$ is the average valency mentioned in (ii). 
In Lemma \ref{lem:kite} we give a formula $z_i = z_2 \alpha_i  + a_1 \beta_i$, where $\alpha_i, \beta_i$ are determined by the intersection numbers.
\medskip

\noindent We now describe our theoretical results. This will be done over the next three paragraphs.
Let $E$ denote a $Q$-polynomial primitive idempotent of $\Gamma$.
\medskip

\noindent Consider the following two conditions on $E$: 
\begin{enumerate}
\item[\rm (i)] the set $\lbrace E{\hat x} \vert x \in X\rbrace$ is Norton balanced;
\item[\rm (ii)] for $x,y \in X$ the vectors $Ex^-_y$, $Ex^+_y$, $E{\hat x}$, $E{\hat y}$ are linearly dependent.
\end{enumerate}
By our earlier comments, (i) implies (ii). 
We display an example for which (ii) holds but not (i).
We show that (i) is implied by (ii) together with
a certain restriction on the coefficients in the linear dependence.
\medskip

\noindent
Let $\lambda$ denote an indeterminate. For $2 \leq i \leq D-1$ we define a quadratic polynomial $\Phi_i(\lambda)$ whose coefficients
are determined by the intersection numbers of $\Gamma$. Pick $x,y \in X$ at distance $\partial(x,y)=i$.
Assuming that $\Gamma$ is reinforced, we compute the inner products between 
 $Ex^-_y$, $Ex^+_y$, $E{\hat x}$, $E{\hat y}$ in terms of the intersection numbers and $z_i, z_{i+1}$.
 Using these inner products and a Cauchy-Schwarz inequality, we show that $\Phi_i(z_2) \geq 0$, with
 equality if and only if  $Ex^-_y$, $Ex^+_y$, $E{\hat x}$, $E{\hat y}$ are linearly dependent.
 We show that  if  $\Gamma$ is reinforced and  the set $\lbrace E{\hat x} \vert x \in X\rbrace$ is Norton-balanced, then $\Phi_i(z_2)=0$ for $2 \leq i \leq D-1$.
 \medskip
 
 \noindent
 We say that $E$ is a dependency candidate (or DC) whenever there exists $\xi \in \mathbb C$ such that
 $\Phi_i(\xi) =0$ for $2\leq i \leq D-1$. Note that $E$ being DC is a condition on the intersection numbers of $\Gamma$.
If $\Gamma$ is reinforced and the set $\lbrace E{\hat x} \vert x \in X\rbrace$ is Norton-balanced, then $E$ is DC.
 In our main theoretical result Theorem  \ref{thm:main},  we display a necessary and sufficient condition on the intersection numbers of $\Gamma$,
 for $E$ to be DC. Using Theorem  \ref{thm:main} we show that for certain examples $\Gamma$ is not Norton-balanced.
\medskip

\noindent In the previous paragraphs, we often assumed that $\Gamma$ is reinforced; this was done for clarity and simplicity. In the main body of the paper,
we sometimes use a more general argument that involves weaker hypotheses.
\medskip

\noindent
This paper is organized as follows. Section 2 contains some preliminaries.
Sections 3, 4 contain  basic information about a distance-regular graph $\Gamma$ and its $Q$-polynomial primitive idempotents $E$.
In Section 5 we recall the Norton algebra.
In Section 6 we introduce the Norton-balanced condition.
In Section 7 we give some examples that satisfy  the Norton-balanced condition.
In Section 8 we give some linear algebraic consequences of the Norton-balanced condition.
In Section 9 we recall some parameters related to the $Q$-polynomial property.
In Section 10 we discuss a 4-vertex configuration called a kite, and we introduce the reinforced condition.
In Sections 11--14 we consider a pair of vertices $x,y$ of $\Gamma$, and investigate the potential linear dependence of  $Ex^-_y$, $Ex^+_y$, $E{\hat x}$, $E{\hat y}$.
In Section 15 we introduce the polynomials $\Phi_i(\lambda)$.
In Section 16 we discuss the DC condition.
Sections 17--29 are about examples.
Section 30 is about the case in which $\Gamma$ affords a spin model.
Section  31 contains some directions for future research.

 \section{Preliminaries} 
 We now begin our formal argument. The following concepts and notation will be used throughout the paper.
 Let $\mathbb R$ denote the field of real numbers.
  Let $X$ denote a nonempty finite set. The elements of $X$ are called {\it vertices}.
 Let ${\rm Mat}_X(\mathbb R)$ denote the $\mathbb R$-algebra consisting of the matrices with rows and columns indexed by $X$ and all entries in $\mathbb R$.
 Let $I \in {\rm Mat}_X(\mathbb R)$ denote the identity matrix.
 Let $V=\mathbb R^X$ denote the $\mathbb R$-vector space  consisting of the column vectors with coordinates indexed by $X$ and all entries in $\mathbb R$.
 The algebra ${\rm Mat}_X(\mathbb R)$ acts on $V$ by left multiplication. We endow $V$ with a bilinear form $\langle \,,\,\rangle$ such that
 $\langle u,v \rangle = u^t v$ for all $u,v \in V$, where $t$ denotes  transpose.
 Note that $\langle u,v\rangle = \langle v,u \rangle $ for $u,v \in V$.
 For $u \in V$ we abbreviate $\Vert u \Vert^2 =\langle u,u \rangle $. We have
 $\Vert u \Vert^2 \geq 0$,
  with equality if and only if $u=0$. The bilinear form turns $V$ into a Euclidean space.
  For $B \in {\rm Mat}_X(\mathbb R)$ we have $\langle Bu,v\rangle = \langle u, B^t v \rangle$ for all $u,v \in V$.
 For $x \in X$ define a vector ${\hat x} \in V$ that has $x$-coordinate $1$ and all other coordinates $0$. The vectors $\lbrace {\hat x} \vert x \in X\rbrace$ form
 an orthonormal basis for $V$. The vector  ${\bf 1} = \sum_{x \in X} {\hat x}$ has all coordinates $1$.  Let  $J \in {\rm Mat}_X(\mathbb R)$ have all entries $1$.
 Note that $J {\hat x} = {\bf 1}$ for all $x \in X$. For $B, C \in {\rm Mat}_X(\mathbb R)$ their entrywise product $B \circ C \in {\rm Mat}_X(\mathbb R)$ has $(x,y)$-entry $B_{x,y} C_{x,y}$ for all $x,y \in X$.
 For a positive $q \in \mathbb R$ let $q^{\frac{1}{2}}$ denote the positive square root of $q$.
\medskip

\section{Distance-Regular Graphs}
 \noindent In this section, we review some definitions and  basic concepts concerning
 distance-regular graphs. See \cite{bbit, banIto, bcn, dkt, int} for more information.
Let $\Gamma = (X, \mathcal R)$ denote a finite, undirected, connected graph,
without loops or multiple edges, with vertex set $X$ and 
adjacency relation
$\mathcal R$.  
For an integer $n\geq 0$, a {\it path of length $n$} in $\Gamma$ is a sequence of vertices
$\lbrace x_i \rbrace_{i=0}^n$ such that $x_{i-1}, x_i$ are adjacent for $1 \leq i \leq n$.
This path is said to {\it connect $x_0, x_n$}. For $x,y \in X$ let $\partial(x,y)$ denote
the length of a shortest path that connects $x,y$. We call $\partial(x,y)$ the
{\it distance between $x$ and $y$}.
The integer $D=\mbox{max}\lbrace \partial(x,y) \vert x,y \in X\rbrace $ is called the {\it diameter} of $\Gamma$. For an integer $i\geq 0$ and $x \in X$ define the set
$\Gamma_i(x) = \lbrace y \in X\vert \partial(x,y)=i \rbrace$. We abbreviate $\Gamma(x)= \Gamma_1(x)$.
For $x \in X$ we call $\vert \Gamma(x)\vert$ the {\it valency of $x$}.
For an integer $k\geq 0$, we say that $\Gamma$ is {\it regular with valency $k$} whenever each vertex in $X$ has valency $k$.
We say that $\Gamma$ is {\it distance-regular}
whenever for all integers $h,i,j$ $(0 \leq h,i,j \leq D)$ 
and  all
vertices $x,y \in X$ at distance $\partial(x,y)=h,$ the cardinality
$p^h_{i,j} = \vert \Gamma_i(x) \cap \Gamma_j(y) \vert$
is independent of $x$ and $y$. The integers $p^h_{i,j}$  are called
the {\it intersection numbers} of $\Gamma$. For the rest of this paper, we assume that $\Gamma$ is distance-regular with $D\geq 3$.
Note that $\Gamma$ is regular with valency $k=p^0_{1,1}$.
By construction $p^h_{i,j} = p^h_{j,i}$ for $0 \leq h,i,j\leq D$.
By the triangle inequality the following holds for $0 \leq h,i,j\leq D$:
\begin{enumerate}
\item[\rm (i)] $p^h_{i,j}= 0$ if one of $h,i,j$ is greater than the sum of the other two;
\item[\rm (ii)] $p^h_{i,j}\not=0$ if one of $h,i,j$ is equal to the sum of the other two.
\end{enumerate}
We abbreviate
\begin{align*}
c_i = p^i_{1,i-1} \quad (1 \leq i \leq D), \qquad a_i = p^i_{1,i} \quad (0 \leq i \leq D), \qquad  b_i = p^i_{1,i+1} \quad (0 \leq i \leq D-1).
\end{align*}
We have  $b_0=k$. Note that $a_0=0$ and $c_1=1$. By \cite[Lemma~4.1.6]{bcn} we have
\begin{align*}
c_{i-1} \leq c_i \quad (2 \leq i \leq D), \qquad b_{i-1}\geq b_i \quad (1 \leq i\leq D-1), \qquad 
b_i \geq c_{D-i} \quad (0 \leq i \leq D-1).
\end{align*}
Observe that
$ k=c_i + a_i + b_i$  $(0 \leq i \leq D)$,
where $c_0=0$ and $b_D=0$.
For $0 \leq i \leq D$ define $k_i = p^0_{i,i}$ and note that $k_i = \vert \Gamma_i(x) \vert$ for all $x \in X$.
 We have $k_0=1$ and $k_1=k$.
By \cite[p.~195]{banIto} we have
\begin{align*} 
k_i = \frac{ b_0 b_1 \cdots b_{i-1}}{c_1 c_2 \cdots c_i} \qquad \qquad (0 \leq i \leq D).
\end{align*}
The graph $\Gamma$ is called {\it bipartite} whenever  $a_i=0$ for $0 \leq i \leq D$.
 The graph $\Gamma$ is called {\it almost bipartite} whenever  $a_i=0$ for $0 \leq i \leq D-1$ and $a_D \not=0$.
  The graph $\Gamma$ is called an {\it antipodal 2-cover} whenever $k_D=1$.
  This occurs if and only if $k_i = k_{D-i}$ $(0 \leq i \leq D)$ if and only if  $b_i = c_{D-i}$  $(0 \leq i \leq D)$; see \cite[Proposition~4.2.2]{bcn}.
\medskip

\noindent We recall the Bose-Mesner algebra of $\Gamma.$ 
For 
$0 \leq i \leq D$ define $A_i \in {\rm Mat}_X(\mathbb R)$ that has
$(x,y)$-entry
\begin{align*}
(A_i)_{x,y} = \begin{cases}  
1, & {\mbox{\rm if $\partial(x,y)=i$}};\\
0, & {\mbox{\rm if $\partial(x,y) \ne i$}}
\end{cases}
 \qquad \qquad (x,y \in X).
\end{align*}
 \noindent For $x \in X$ we have
 \begin{align} \label{eq:AiMeaning}
 A_i {\hat x} = \sum_{y \in \Gamma_i(x)} {\hat y}.
  \end{align}
  \noindent
We call $A_i$ the $i$th {\it distance matrix} of $\Gamma$. 
We abbreviate $A=A_1$ and call this the {\it adjacency
matrix} of $\Gamma$. Observe that
(i) $A_0 = I$;
 (ii)
$J=\sum_{i=0}^D A_i $;
(iii)  $A_i^t = A_i  \;(0 \leq i \leq D)$;
(iv) $A_iA_j = \sum_{h=0}^D p^h_{i,j} A_h \;( 0 \leq i,j \leq D) $.
Therefore the matrices
 $\lbrace A_i\rbrace_{i=0}^D$
form a basis for a commutative subalgebra $M$ of ${\rm Mat}_X(\mathbb R)$, called the 
{\it Bose-Mesner algebra} of $\Gamma$.
The matrix $A$ generates $M$ \cite[Corollary~3.4]{int}. 
The matrices $\lbrace A_i \rbrace_{i=0}^D$ are symmetric and mutually commute, so they can be simultaneously diagonalized over $\mathbb R$. Consequently
$M$ has a second basis 
$\lbrace E_i\rbrace_{i=0}^D$ such that
(i) $E_0 = |X|^{-1}J$;
(ii) $I=\sum_{i=0}^D E_i$;
(iii) $E_i^t =E_i  \;(0 \leq i \leq D)$;
(iv) $E_iE_j =\delta_{i,j}E_i  \;(0 \leq i,j \leq D)$.
We call $\lbrace E_i\rbrace_{i=0}^D$  the {\it primitive idempotents}
of $\Gamma$. The primitive idempotent $E_0$ is called {\it trivial}.
\medskip

 \noindent For $0 \leq i \leq D$ let $\theta_i$ denote the eigenvalue of $A$ for $E_i$. We have $AE_i = \theta_i E_i = E_i A$. We have $A=\sum_{i=0}^D \theta_i E_i$. 
The scalars $\lbrace \theta_i \rbrace_{i=0}^D$ are mutually distinct because $A$ generates $M$. We have
\begin{align*}
V = \sum_{i=0}^D E_iV \qquad \qquad {\mbox{\rm (orthogonal direct sum)}}.
\end{align*}
For $0 \leq i \leq D$ the subspace $E_iV$ is the eigenspace of $A$ for the eigenvalue $\theta_i$. By \cite[p.~128]{bcn} we have $\theta_0=k$.
 \medskip
 
 \noindent We recall the Krein parameters of $\Gamma$. 
For $0 \leq i,j\leq D$ we have  $A_i \circ A_j = \delta_{i,j} A_i$.
 Therefore  $M$  is closed under $\circ$.
Consequently, there exist scalars $q^{h}_{i,j} \in \mathbb R $ $(0 \leq h,i,j\leq D)$ such that
\begin{align*} 
E_i \circ E_j = \vert X \vert^{-1} \sum_{h=0}^D q^h_{i,j} E_h \qquad \qquad (0 \leq i,j\leq D).
\end{align*}  
The scalars $q^{h}_{i,j}$ are called the {\it Krein parameters} of $\Gamma$.
By construction $q^h_{i,j} = q^h_{j,i}$ for  $0 \leq h,i,j\leq D$.
By \cite[p.~69]{banIto} we have $q^h_{i,j}\geq 0$  for $0 \leq h,i,j\leq D$.
By \cite[Lemma~5.15]{int} we have $q^0_{i,i} = {\rm dim}(E_iV)$ for $0 \leq i \leq D$.
\medskip

\noindent Next, we describe a feature of the Krein parameters that will play a role in our main results. In this description, we will use the following notation.
For $u \in V$ and $x \in X$ let $u_x$ denote the $x$-coordinate of $u$. So $u = \sum_{x \in X} u_x \hat x$.
 For $u, v \in V$ define a vector $u\circ v \in V$  that has $x$-coordinate $u_x v_x$ for all
$x \in X$. So
$u \circ v = \sum_{x \in X} u_x v_x \hat x$.
We have
\begin{align}
\label{eq:xycirc}
\hat x \circ \hat y = \delta_{x,y} {\hat x} \qquad \qquad (x,y \in X).
\end{align} 
For $v \in V$ we have ${\bf 1} \circ v = v$. 
\begin{lemma} \label{lem:reform} {\rm (See \cite[Proposition~5.1]{norton}.)} The following hold for $0 \leq h,i,j\leq D$.
\begin{enumerate}
\item[\rm (i)] Assume that $q^h_{i,j} \not=0$. Then
$E_h  \bigl(E_i V \circ E_j V \bigr)$ spans $E_hV$.
\item[\rm (ii)] Assume that $q^h_{i,j} =0$. Then $E_h  \bigl(E_i V \circ E_j V \bigr)=0$.
\end{enumerate}
\end{lemma}


 \noindent We recall the $Q$-polynomial property. The ordering $\lbrace E_i \rbrace_{i=0}^D$ of the primitive idempotents is called {\it $Q$-polynomial} whenever the following hold
for $0 \leq h,i,j\leq D$:
\begin{enumerate}
\item[\rm (i)] $q^h_{i,j}=0$ if one of $h,i,j$ is greater than the sum of the other two;
\item[\rm (ii)]  $q^h_{i,j}\not=0$ if one of $h,i,j$ is equal to the sum of the other two.
\end{enumerate}
Assume that the ordering $\lbrace E_i \rbrace_{i=0}^D$ is $Q$-polynomial. We abbreviate
\begin{align*}
c^*_i = q^i_{1,i-1} \;\; (1 \leq i \leq D), \qquad a^*_i = q^i_{1,i} \;\; (0 \leq i \leq D), \qquad  b^*_i = q^i_{1,i+1} \;\; (0 \leq i \leq D-1).
\end{align*}
We emphasize that $c^*_i \not= 0 $  $(1 \leq i \leq D)$ and $b^*_i \not= 0 $ $(0 \leq i \leq D-1)$. By \cite[Lemma~5.15]{int} we have $a^*_0=0$ and $c^*_1=1$.
The $Q$-polynomial ordering $\lbrace E_i \rbrace_{i=0}^D$ is called {\it dual-bipartite} (resp. {\it almost dual-bipartite}) whenever  $a^*_i = 0$ for $0 \leq i \leq D$
(resp. $a^*_i = 0$ for $0 \leq i \leq D-1$ and $a^*_D\not=0$).
\medskip

\noindent 
A  primitive idempotent $E$ of $\Gamma$ is called {\it  $Q$-polynomial}
whenever there exists a $Q$-polynomial ordering $\lbrace E_i\rbrace_{i=0}^D$ of the primitive idempotents of $\Gamma$
such that $E=E_1$. Assume that $E$ is $Q$-polynomial. 
 We say that $E$ is {\it dual-bipartite} (resp. {\it almost dual-bipartite}) whenever the corresponding $Q$-polynomial ordering $\lbrace E_i \rbrace_{i=0}^D$ is
  dual-bipartite (resp. almost dual-bipartite).
By \cite[p.~241]{bcn} and \cite[Theorems~1.1,~1.2]{dualBip}, $E$  is dual-bipartite if and only if $\Gamma$ is an antipodal 2-cover. 
\medskip

\noindent We say that $\Gamma$ is {\it $Q$-polynomial} whenever there exists a $Q$-polynomial ordering $\lbrace E_i \rbrace_{i=0}^D$ of the primitive idempotents of $\Gamma$.
We say that $\Gamma$ is {\it dual-bipartite} (resp. {\it almost dual-bipartite}) whenever there exists a $Q$-polynomial ordering
 $\lbrace E_i \rbrace_{i=0}^D$ of the primitive idempotents of $\Gamma$
 that is dual-bipartite (resp. almost dual-bipartite).
 \medskip
 
 \noindent 
 For the rest of this paper, we assume that $\Gamma$ is $Q$-polynomial. To avoid trivialities, we always assume that the valency $k\geq 3$.
 
\section{Some eigenspace geometry}
We continue to discuss the $Q$-polynomial distance-regular graph $\Gamma=(X, \mathcal R)$ with diameter $D\geq 3$.
Let $E$ denote a $Q$-polynomial primitive idempotent of $\Gamma$. By construction, the eigenspace $EV$ is spanned by
the vectors  $\lbrace E{\hat x} \vert x \in X\rbrace$. In this section, we discuss the geometry of these vectors. We will describe the inner product
$\langle E{\hat x}, E{\hat y}\rangle$ for $x,y\in X$. We will also display some linear dependencies among  $\lbrace E{\hat x} \vert x \in X\rbrace$.
\medskip

\noindent The matrix $E$ is contained in the Bose-Mesner algebra $M$. Therefore,  there exist $\theta^*_i \in \mathbb R$ $(0 \leq i \leq D)$  such that
\begin{align}
\label{eq:EEsum}
E = \vert X \vert^{-1} \sum_{i=0}^D \theta^*_i A_i.
\end{align}
By \cite[p.~260]{banIto} the scalars $\lbrace \theta^*_i \rbrace_{i=0}^D$ are mutually distinct. By \cite[Lemma~3.9]{int} we have $\theta^*_0 = {\rm dim} (EV)$.
The scalars $\lbrace \theta^*_i \rbrace_{i=0}^D$ 
are called the {\it dual eigenvalues} of $\Gamma$ associated with $E$. For notational convenience, let $\theta^*_{-1}$ and $\theta^*_{D+1}$
denote indeterminates.  By \cite[p.~128]{bcn},
\begin{align}
c_i \theta^*_{i-1} +
a_i \theta^*_i +
b_i \theta^*_{i+1} 
= \theta_1 \theta^*_i
\qquad \qquad (0 \leq i \leq D).
\label{eq:3TR}
\end{align}

\noindent The following result is well known; see for example \cite[Proposition~4.4.1]{bcn}.
 \begin{lemma} \label{lem:geometry} {\rm (See \cite[Proposition~4.4.1]{bcn}.)} Pick $x,y \in X$ and write $i= \partial(x,y)$.
 Then the following {\rm (i)--(iii)} hold:
 \begin{enumerate}
 \item[\rm (i)] 
$ \langle E{\hat x}, E{\hat y} \rangle= \vert X \vert^{-1} \theta^*_i$;
\item[\rm (ii)] $\Vert E{\hat x} \Vert^2=\Vert E{\hat y} \Vert^2 = \vert X \vert^{-1} \theta^*_0$;
 \item[\rm (iii)] $\theta^*_i /\theta^*_0$ is the cosine of the angle between $E{\hat x}$ and $E{\hat y}$.
 \end{enumerate}
 \end{lemma}
 
 \begin{corollary} \label{cor:distinct} For distinct $x,y \in X$ we have $E{\hat x} \not=E{\hat y}$.
 \end{corollary}
 \begin{proof} Write $i=\partial(x,y)$ and note that $i \not=0$. The dual eigenvalues $\lbrace \theta^*_j \rbrace_{j=0}^D$ are mutually distinct, so $\theta^*_i \not=\theta^*_0$. The result follows in view of Lemma  \ref{lem:geometry}(iii).
 \end{proof}
 
 \noindent As we consider additional consequences of Lemma \ref{lem:geometry}, we will treat separately the case in which
 $\Gamma$ is an antipodal 2-cover.
 
 \begin{lemma} \label{lem:NotAntip} Assume that $\Gamma$ is not an antipodal 2-cover. Then the following hold:
 \begin{enumerate}
 \item[\rm (i)] $\theta^*_0 > \theta^*_i > - \theta^*_0 \quad (1 \leq i \leq D)$;
 \item[\rm (ii)]for distinct $x,y \in X$ the vectors $E{\hat x}, E{\hat y}$ are linearly independent.
 \end{enumerate}
 \end{lemma}
 \begin{proof} (i) Pick $x,y \in X$ at distance $\partial(x,y)=i$.
 Using Lemma \ref{lem:geometry} and trigonometry we obtain  $\theta^*_0 > \theta^*_i \geq - \theta^*_0$, with equality on the right if and only if
 $E{\hat x} + E{\hat y}=0$. Suppose this equality occurs. The vertices $x,y$ uniquely determine each other by  $E{\hat x} + E{\hat y}=0$ and Corollary \ref{cor:distinct}, so
 $k_i=1$. Now $k_D=1$ in view of \cite[Proposition~5.1.1(i)]{bcn}. Consequently $\Gamma$ is an antipodal 2-cover, for a contradiction. We have shown that $\theta^*_0 > \theta^*_i >- \theta^*_0$.
 \\
 \noindent (ii) By (i) and Lemma \ref{lem:geometry}.
  \end{proof}
 
  \begin{lemma} \label{lem:Antip} Assume that $\Gamma$ is an antipodal 2-cover. Then the following hold.
 \begin{enumerate}
 \item[\rm (i)] $\theta^*_0 > \theta^*_i > - \theta^*_0 \quad (1 \leq i \leq D-1)$ and $\theta^*_D=-\theta^*_0$;
 \item[\rm (ii)]for distinct $x,y \in X$ the vectors $E{\hat x}, E{\hat y}$ are linearly independent if $\partial(x,y) \not=D$, and  $E{\hat x}+E{\hat y}=0$ if $\partial(x,y) =D$.
 \end{enumerate}
 \end{lemma}
  \begin{proof} Similar to the proof of Lemma \ref{lem:NotAntip}, except that $\theta^*_D = - \theta^*_0$ by
   \cite[Theorem~8.2.4]{bcn}.
 \end{proof}
 
\noindent Next, we display some linear dependencies among the vectors $\lbrace E{\hat x} \vert x \in X\rbrace$.

 \begin{lemma} \label{lem:twoSum} Let $x \in X$. Then
  \begin{align} \label{eq:twoSum}
              \sum_{y \in \Gamma(x)} E{\hat y} = \theta_1 E {\hat x}.
 \end{align}
 Moreover for $0 \leq i \leq D$,
  \begin{align} \label{eq:twoSum2}
  \sum_{y \in \Gamma_i(x)} E{\hat y} \in {\rm Span} \lbrace E {\hat x}\rbrace .
 \end{align}
 \end{lemma}
 \begin{proof} The equation  \eqref{eq:twoSum} holds, because each side is equal to $EA{\hat x}$.
 To verify  \eqref{eq:twoSum2}, note that $A_i = f_i(A)$, where $f_i$ is a polynomial with real coefficients and degree $i$. Using
 \eqref{eq:AiMeaning} we obtain
 \begin{align*}
  \sum_{y \in \Gamma_i(x)} E{\hat y} = E A_i {\hat x} = E f_i(A) {\hat x} =  f_i(\theta_1) E{\hat x}    \in {\rm Span} \lbrace E {\hat x}\rbrace.
 \end{align*}
 \end{proof}
 
 
 \section{The Norton algebra}
 We continue to discuss the $Q$-polynomial distance-regular graph $\Gamma=(X, \mathcal R)$ with diameter $D\geq 3$.
Let $E$ denote a $Q$-polynomial primitive idempotent of $\Gamma$.
In this section, we recall the Norton algebra product $\star$ on the vector space $EV$. For $x,y \in X$ we compute $E{\hat x} \star E{\hat y}$.

\begin{definition} \label{def:Norton} {\rm (See \cite[Proposition~5.2]{norton}.)} \rm
The {\it Norton algebra} $EV$ consists of the $\mathbb R$-vector space $EV$, together with the product
\begin{align}
u \star v = E (u \circ v) \qquad \qquad (u,v \in EV).     \label{eq:NortonAlg}
\end{align}
\noindent We call $\star $ the {\it Norton product}.
\end{definition}
\noindent The Norton product $\star$ is commutative, and nonassociative in general. 
\medskip

\noindent As we investigate $\star$ it is natural to consider $E\hat x \star E\hat y$ for all $x,y \in X$.
In the next two lemmas we discuss some extremal cases.

\begin{lemma} \label{lem:warm} {\rm (See \cite[Lemma 3.2]{nortonPT}.)}
For $x \in X$,
\begin{align*}
E \hat x \star E \hat x = |X|^{-1} a^*_1E \hat x.
\end{align*}
\end{lemma}

\begin{lemma} \label{lem:warmup1}
The following {\rm (i)--\rm (iv)} are equivalent:
\begin{enumerate}
\item[\rm (i)] $u \star v = 0$ for all $u,v \in EV$;
\item[\rm (ii)] $E\hat x \star E \hat y =0$ for all $x,y\in X$;
\item[\rm (iii)] $E\hat x \star E \hat x =0$ for all $x\in X$;
\item[\rm (iv)]  $a^*_1 =0$.
\end{enumerate}
\end{lemma}
\begin{proof} The implications ${\rm (i)} \Rightarrow {\rm (ii)} \Rightarrow {\rm (iii)}$ are clear.
The implication ${\rm (iii)} \Rightarrow {\rm (iv)}$ is from Lemma \ref{lem:warm}, and the implication  ${\rm (iv)} \Rightarrow {\rm (i)}$ is from 
Lemma \ref{lem:reform}.
\end{proof}
\begin{corollary} Assume that $E$ is dual-bipartite or almost dual-bipartite. Then the equivalent conditions {\rm (i)--(iv)} 
in Lemma \ref{lem:warmup1} all hold.
\end{corollary}
\begin{proof} We assume that $a^*_i = 0 $ for $0 \leq i \leq D-1$, so $a^*_1=0$.
\end{proof}

\noindent  Let $x,y \in X$. Shortly, we will review a formula for $E{\hat x} \star E{\hat y}$ that appeared in \cite[Theorem~3.7]{nortonPT}.

\begin{lemma} \label{lem:AcA} For $x,y \in X$ and $0 \leq i,j\leq D$,
\begin{align} \label{eq:Phij} 
A_i \hat x \circ A_j \hat y = 
\sum_{z \in \Gamma_i(x) \cap \Gamma_j(y)}  {\hat z}.
 \end{align}
 \end{lemma}
 \begin{proof} By \eqref{eq:AiMeaning} and \eqref{eq:xycirc}.
 \end{proof}

 \begin{lemma} \label{lem:Phij} Pick $0 \leq h,i,j\leq D$ and $x,y \in X$ at distance $\partial(x,y)=h$. Then:
 \begin{enumerate}
 \item[\rm (i)] $\Vert  A_i \hat x \circ A_j \hat y \Vert^2 = p^h_{i,j}$;
 \item[\rm (ii)] 
 $A_i \hat x \circ A_j \hat y =0$ if and only if $p^h_{i,j} =0$.
 \end{enumerate}
 \end{lemma}
 \begin{proof} (i) By \eqref{eq:Phij} and since $\vert \Gamma_i(x) \cap \Gamma_j(y) \vert = p^h_{i,j}$.\\
 \noindent (ii) By (i) above.
 \end{proof} 
 
\begin{definition}\label{def:xyxy} {\rm (See  \cite[Definition~3.5]{nortonPT}.)}
\rm Pick $x, y \in X$ and write $i = \partial (x,y)$. Define
\begin{align*}
 x^-_y &= A \hat x \circ A_{i-1} \hat y = \sum_{z \in \Gamma(x) \cap \Gamma_{i-1}(y)}  {\hat z},
\\
x^0_y &= A \hat x \circ A_{i} \hat y = 
\sum_{z\in \Gamma(x) \cap \Gamma_i(y)} {\hat z},
\\x^+_y &= A \hat x \circ A_{i+1} \hat y = 
\sum_{z \in \Gamma(x) \cap \Gamma_{i+1}(y) }  {\hat z},
 \end{align*}
\noindent where we understand
\begin{align*}
A_{-1}=0, \qquad \quad \Gamma_{-1}(x)= \emptyset, \qquad \quad
A_{D+1}=0, \qquad \quad \Gamma_{D+1}(x)=\emptyset.
\end{align*}
\end{definition}
\noindent We clarify the meaning of Definition \ref{def:xyxy}.
 Pick $x, y \in X$. If $\partial(x,y)=D$ then $x^+_y = 0$. If $\partial(x,y)=1$ then $x^-_y = \hat y$.  If $x=y$ then $x^0_y=0$ and $x^-_y=0$.

\begin{lemma}
\label{lem:ssum} For $x,y \in X$ we have
\begin{align}
x^-_y + x^0_y + x^+_y &= A \hat x,
\label{eq:ssum1}
\\
Ex^-_y + Ex^0_y + Ex^+_y &= \theta_1 E \hat x.
\label{eq:ssum2}
\end{align}
\end{lemma}
\begin{proof}  Assertion \eqref{eq:ssum1} follows from \eqref{eq:AiMeaning} (with $i=1$). 
Assertion \eqref{eq:ssum2}
follows from \eqref{eq:twoSum}.
\end{proof}
 

\begin{proposition}\label{thm:NA} {\rm (See \cite[Theorem~3.7]{nortonPT}.)}
For $x,y\in X$ we have
\begin{align}
 E{\hat x}\star E{\hat y} &=
 \frac{ (\theta^*_{i-1}-\theta^*_i) Ex^-_y+(\theta^*_{i+1}-\theta^*_{i})
Ex_y^+ +(\theta_1-\theta_2)\theta^*_i E{\hat x}+(\theta_2-\theta_0)E{\hat y}}
  {\vert X \vert (\theta_1-\theta_2)}
  \label{eq:main}
\end{align}
where $i = \partial (x,y)$. We recall that  $\theta^*_{-1} $ and $\theta^*_{D+1} $ are  indeterminates.
\end{proposition}

\noindent We mention some special cases of Proposition \ref{thm:NA}.

\begin{corollary} \label{cor:01D} {\rm (See  \cite[Corollary~3.8]{nortonPT}.)}
The following {\rm (i)--(iii)} hold.
\begin{enumerate}
\item[\rm (i)] For $x \in X$,
\begin{align*}
E \hat x \star E \hat x = \frac{\theta_1 \theta^*_1 - \theta_2 \theta^*_0 + \theta_2 - \theta_0}{\vert X \vert (\theta_1 - \theta_2)} E\hat x.
\end{align*}
\item[\rm (ii)] For $x,y \in X$ at distance $\partial(x,y)=1$,
\begin{align*}
E\hat x \star E \hat y = 
\frac{ (\theta^*_2-\theta^*_1) E x^+_y +(\theta_1-\theta_2)\theta^*_1 E \hat x + (\theta_2-\theta_0+\theta^*_0 - \theta^*_1)E \hat y}{\vert X \vert (\theta_1-\theta_2)}.
\end{align*}
\item[\rm (iii)] For $x,y \in X$ at distance $\partial(x,y)=D$,
\begin{align*}
 E{\hat x}\star E{\hat y} &=
 \frac{ (\theta^*_{D-1}-\theta^*_D) Ex^-_y +(\theta_1-\theta_2)\theta^*_D E{\hat x}+(\theta_2-\theta_0)E{\hat y}}
  {\vert X \vert (\theta_1-\theta_2)}.
\end{align*}
\end{enumerate}
\end{corollary}
\noindent Comparing Lemma \ref{lem:warm}  and Corollary  \ref{cor:01D}(i), we obtain
\begin{align*}
 a^*_1 = \frac{\theta_1 \theta^*_1 - \theta_2 \theta^*_0 + \theta_2 - \theta_0}{\theta_1 - \theta_2}.
 \end{align*}

 \section{The Norton-balanced condition}
 
 \noindent We continue to discuss the $Q$-polynomial distance-regular graph $\Gamma=(X, \mathcal R)$ with diameter $D\geq 3$.
Let $E$ denote a $Q$-polynomial primitive idempotent of $\Gamma$. In Section 4, we considered some linear dependencies among
the vectors  $\lbrace E{\hat x} \vert x \in X\rbrace$.
 In the present section, we return to this topic. We will review the balanced set condition \cite{balanced} and some variations \cite{stronglybalanced, kiteFree}.
Then we will  introduce the Norton-balanced condition.
 
 \begin{lemma} \label{lem:bs} {\rm (Balanced set condition \cite[Theorem~1.1]{balanced}.)} For $x,y \in X$ and $0 \leq i,j\leq D$,
 \begin{align*}
 \sum_{z \in \Gamma_i(x) \cap \Gamma_j(y)} E {\hat z} -
  \sum_{z \in \Gamma_j(x) \cap \Gamma_i(y)} E {\hat z} 
 \in {\rm Span}\lbrace E{\hat x} - E {\hat y}\rbrace.
 \end{align*}
\end{lemma}
\noindent We emphasize some special cases of Lemma \ref{lem:bs}. 
For $x,y \in X$,
\begin{align*}
E x^+_y - E y^+_x \in {\rm Span} \lbrace E{\hat x} - E {\hat y}\rbrace, \qquad \qquad
E x^-_y - E y^-_x \in {\rm Span} \lbrace E{\hat x} - E {\hat y}\rbrace.
\end{align*}

\noindent Next, we describe the symmetric balanced set condition.

\begin{lemma} \label{lem:symBS} {\rm (Symmetric balanced set condition \cite[Theorem~2.6]{kiteFree}.)} For $x,y \in X$ and $0 \leq i,j\leq D$,
\begin{align*}
 \sum_{z \in \Gamma_i(x) \cap \Gamma_j(y)} E {\hat z} +
  \sum_{z \in \Gamma_j(x) \cap \Gamma_i(y)} E {\hat z} 
 \in {\rm Span}\lbrace  Ex^-_y + E y^-_x,  E x^+_y + E y^+_x, E{\hat x} + E {\hat y}\rbrace.
\end{align*}
\end{lemma}

\noindent Combining Lemmas \ref{lem:bs} and \ref{lem:symBS}, we obtain the following result.

\begin{lemma}  \label{lem:Combine}  For $x,y \in X$ and $0 \leq i,j\leq D$,
\begin{align*}
 \sum_{z \in \Gamma_i(x) \cap \Gamma_j(y)} E {\hat z} 
 \in {\rm Span}\lbrace  Ex^-_y,  E x^+_y, E{\hat x}, E {\hat y} \rbrace.
\end{align*}
\end{lemma}

\noindent It could happen that for all $x,y \in X$ the vectors $Ex^-_y$,  $E x^+_y$, $E{\hat x}$, $E {\hat y}$ are linearly dependent. We now consider some situations where this occurs.

\begin{definition} \label{def:STR}  \rm {(See \cite[Section~2]{stronglybalanced}.)} The set of vectors $\lbrace E{\hat x} \vert x \in X\rbrace$ is called {\it strongly balanced}
whenever for all $x,y \in X$ and $0 \leq i,j\leq D$,
 \begin{align*}
 \sum_{z \in \Gamma_i(x) \cap \Gamma_j(y)} E {\hat z} 
 \in {\rm Span}\lbrace E{\hat x}, E {\hat y}\rbrace.
 \end{align*}
\end{definition}

\begin{lemma} \label{lem:DBip} {\rm (See \cite[Theorems~1, 3]{stronglybalanced}.)} The following are equivalent:
\begin{enumerate}
\item[\rm (i)] the set $\lbrace E{\hat x} \vert x \in X\rbrace$ is strongly balanced;
\item[\rm (ii)] $E$ is dual-bipartite or almost dual-bipartite.
\end{enumerate}
\end{lemma}

\noindent We now introduce the Norton-balanced condition. 

\begin{definition}\label{def:NB} \rm The set of vectors
$\lbrace E{\hat x} \vert x \in X\rbrace$ is called {\it Norton-balanced}
whenever for all $x,y \in X$ and $0 \leq i,j\leq D$,
 \begin{align*}
 \sum_{z \in \Gamma_i(x) \cap \Gamma_j(y)} E {\hat z} 
 \in {\rm Span}\lbrace E{\hat x}, E {\hat y}, E{\hat x} \star E{\hat y}\rbrace.
 \end{align*}
\end{definition}

\noindent Let us clarify the above definition.
\begin{lemma} \label{lem:clarify} The following are equivalent:
\begin{enumerate}
\item[\rm (i)] the set $\lbrace E{\hat x} \vert x \in X\rbrace$ is Norton-balanced;
\item[\rm (ii)]  for all $x,y \in X$ we have
$ Ex^-_y, Ex^+_y \in {\rm Span}\lbrace E{\hat x}, E{\hat y}, E{\hat x} \star E{\hat y}\rbrace $.
\end{enumerate}
\end{lemma}
\begin{proof} By Lemma   \ref{lem:Combine}.
\end{proof}

\begin{lemma} \label{lem:clarify2} Let $x,y \in X$ and write $i=\partial(x,y)$.
\begin{enumerate}
\item[\rm (i)] Assume that $i \in \lbrace 0,1,D\rbrace$. Then
 $Ex^-_y, Ex^+_y \in {\rm Span}\lbrace E{\hat x}, E{\hat y}, E{\hat x} \star E{\hat y}\rbrace $.
\item[\rm (ii)] Assume that $2 \leq i \leq D-1$. Then $Ex^-_y \in  {\rm Span}\lbrace E{\hat x}, E{\hat y}, E{\hat x} \star E{\hat y}\rbrace $ if and only if
 $Ex^+_y \in  {\rm Span}\lbrace E{\hat x}, E{\hat y}, E{\hat x} \star E{\hat y}\rbrace $.
\end{enumerate}
\end{lemma}
\begin{proof} (i) By  \eqref{eq:twoSum} and Corollary  \ref{cor:01D}. \\
\noindent (ii) By Proposition \ref{thm:NA}.
\end{proof}

\begin{remark}\rm The Norton-balanced condition is not a condition on the intersection numbers alone.
We show this with an example. The example involves the Hamming graph $H(D,4)$ \cite[p.~355]{bbit} and
a Doob graph of diameter $D$ \cite[p.~387]{bbit}. These graphs have the same intersection numbers, but are not isomorphic. They both have a $Q$-polynomial
structure with eigenvalue sequence $\theta_i = 3D-4i$ $(0 \leq i \leq D)$. For either graph, let
$E=E_1$ denote the primitive idempotent associated with $\theta_1$. As we will see, the set
$\lbrace E{\hat x} \vert x \in X\rbrace$ is Norton-balanced for $H(D,4)$ but not for the Doob graph.
\end{remark}


\section{The Norton-balanced condition; first examples}
\noindent We continue to discuss the $Q$-polynomial distance-regular graph $\Gamma=(X, \mathcal R)$ with diameter $D\geq 3$.
Let $E$ denote a $Q$-polynomial primitive idempotent of $\Gamma$.
 In this section,
  we show that the set $\lbrace E{\hat x} \vert x \in X\rbrace$ is Norton-balanced in the following cases: 
 $\Gamma$ is bipartite;  $\Gamma$ is almost bipartite; $E$ is  dual-bipartite; $E$ is almost dual-bipartite; $\Gamma$ is tight.
 
 \begin{lemma} \label{lem:BAB} Assume that $\Gamma$ is bipartite or almost bipartite. Let $1 \leq i \leq D-1$
 and $x,y \in X$ at distance $\partial (x,y)=i$. Then 
 \begin{align} \label{eq:bip1}
 Ex^+_y = \theta_1 E{\hat x} - Ex^-_y.
 \end{align}
 \noindent Moreover,
 \begin{align} \label{eq:bip2}
 E{\hat x}\star E{\hat y} &=
 \frac{ (\theta^*_{i-1}-\theta^*_{i+1}) Ex^-_y +(\theta_1\theta^*_{i+1}-\theta_2 \theta^*_i) E{\hat x}+(\theta_2-\theta_0)E{\hat y}}
  {\vert X \vert (\theta_1-\theta_2)}.
 \end{align}
 \end{lemma}
 \begin{proof} By Lemma \ref{lem:Phij} and $a_i=0$ we have $x^0_y=0$. This and Lemma \ref{lem:ssum} imply \eqref{eq:bip1}. 
  To get \eqref{eq:bip2}, evaluate Proposition \ref{thm:NA}
 using \eqref{eq:bip1}.
 \end{proof}

 \begin{proposition} \label{prop:BAB} Assume that $\Gamma$ is bipartite or almost bipartite. Then the set $\lbrace E{\hat x} \vert x \in X\rbrace$
 is Norton-balanced.
 \end{proposition}
 \begin{proof}  We invoke Lemma \ref{lem:clarify}. 
 For $x, y \in X$ we show that
 \begin{align} \label{eq:need}
 Ex^-_y, Ex^+_y \in {\rm Span}\lbrace E{\hat x}, E{\hat y}, E{\hat x} \star E{\hat y}\rbrace.
 \end{align}
 First assume that $\partial(x,y)\in \lbrace 0,1,D\rbrace$. Then \eqref{eq:need} holds by Lemma \ref{lem:clarify2}(i).
 Next assume that $2 \leq \partial (x,y)\leq D-1$. Then \eqref{eq:need} holds by Lemma \ref{lem:clarify2}(ii) and 
 \eqref{eq:bip2}. The result follows.
 \end{proof}
 
  \begin{proposition} Assume that $E$ is dual-bipartite or almost dual-bipartite. Then the set $\lbrace E{\hat x} \vert x \in X\rbrace$
 is Norton-balanced.
 \end{proposition}
 \begin{proof}  By Lemma \ref{lem:DBip} and Definitions \ref{def:STR}, \ref{def:NB}.
 \end{proof}
 
\noindent Next, we recall what it means for $\Gamma$ to be tight. The tight concept was introduced in \cite{jkt}, and discussed further in \cite{pascasio}.
Assume for the moment that $\Gamma$ is not bipartite. By  \cite[Theorem~1.3]{pascasio},  $a_D = 0$ if and only if $a^*_D=0$.

\begin{definition} \rm (See \cite[Theorem~1.3]{pascasio}.)
We say that $\Gamma$ is {\it tight} whenever $\Gamma$ is not bipartite and $a_D=a^*_D=0$.
\end{definition}

\noindent   We bring in some notation. Write
\begin{align*}
E_D = \vert X \vert^{-1} \sum_{i=0}^D \varrho_i A_i, \qquad \qquad \varrho_i \in \mathbb R.
\end{align*}
For notational convenience, let $\varrho_{-1}$ and $\varrho_{D+1}$ denote indeterminates.

\begin{lemma} \label{lem:rhoDep} Assume that $\Gamma $ is tight. Pick distinct $x,y \in X$ and write $i=\partial(x,y)$. Then
\begin{align} \label{eq:rhoDep}
 (\varrho_{i-1}-\varrho_i) Ex^-_y +  (\varrho_{i+1}-\varrho_i)Ex^+_y  = \bigl(\theta_{D-1}- \theta_1 \bigr) \varrho_i E{\hat x}.
\end{align}
\end{lemma}
\begin{proof}  We first show that
\begin{align} \label{eq:Zer}
E \Bigl( E_D {\hat y} \circ (A- \theta_{D-1}I) {\hat x} \Bigr) = 0.
\end{align}
We have $E_D {\hat y} \in E_DV$. We have
\begin{align*}
 A-\theta_{D-1} I  =\sum_{j=0}^D (\theta_j - \theta_{D-1}) E_j = \sum_{\stackrel{ \scriptstyle 0 \leq j \leq D }{ \scriptstyle j \not=D-1}} (\theta_j - \theta_{D-1}) E_j.
\end{align*}
Therefore,
\begin{align*}
( A-\theta_{D-1} I){\hat x} \in \sum_{\stackrel{ \scriptstyle 0 \leq j \leq D }{ \scriptstyle j \not=D-1}} E_jV.
\end{align*}
For $0 \leq j \leq D$ such that $j \not=D-1$, we have $q^1_{D,j} = 0$ and therefore $E \bigl( E_DV \circ E_j V\bigr) = 0$ in view of Lemma \ref{lem:reform}.
By these comments we get \eqref{eq:Zer}. By \eqref{eq:Zer} and the construction,
\begin{align*}
0 &= \vert X \vert E \Bigl( E_D {\hat y} \circ (A- \theta_{D-1}I) {\hat x} \Bigr) \\
&=\vert X \vert E \Bigl( E_D {\hat y} \circ \bigl( x^-_y + x^0_y + x^+_y - \theta_{D-1} {\hat x}\bigr) \Bigr) \\
&=  E \Bigl( \varrho_{i-1}  x^-_y + \varrho_i x^0_y +\varrho_{i+1}x^+_y -\theta_{D-1} \varrho_i  {\hat x}\Bigr) \\
&=   \varrho_{i-1}  Ex^-_y + \varrho_i E x^0_y +\varrho_{i+1} Ex^+_y - \theta_{D-1} \varrho_i  E{\hat x}.
\end{align*}
In the previous line, we eliminate $Ex^0_y$ using \eqref{eq:ssum2}, and routinely obtain \eqref{eq:rhoDep}.
\end{proof} 

\begin{lemma} Assume that $\Gamma$ is tight. Then
 the set $\lbrace E{\hat x} \vert x \in X\rbrace$ is Norton-balanced.
\end{lemma}
\begin{proof}  We invoke Lemma \ref{lem:clarify}. 
 For $x, y \in X$ we show that
 \begin{align} \label{eq:need4}
 Ex^-_y, Ex^+_y \in {\rm Span}\lbrace E{\hat x}, E{\hat y}, E{\hat x} \star E{\hat y}\rbrace.
 \end{align}
 First assume that $\partial(x,y)\in \lbrace 0,1,D\rbrace$. Then \eqref{eq:need4} holds by Lemma \ref{lem:clarify2}(i).
 Next assume that $2 \leq \partial (x,y)\leq D-1$. 
The equations   \eqref{eq:main}, \eqref{eq:rhoDep} give a linear system in the unknowns $Ex^-_y$, $Ex^+_y$. For this linear system the coefficient matrix is invertible; indeed
\begin{align*}
{\rm det} \begin{pmatrix} \theta^*_{i-1} - \theta^*_i & \theta^*_{i+1}-\theta^*_i \\
                         \varrho_{i-1} - \varrho_i & \varrho_{i+1} - \varrho_i
\end{pmatrix} \not=0
\end{align*}
by \cite[p.~183]{jkt}. By linear algebra, the linear system has a unique solution for  $Ex^-_y$, $Ex^+_y$. Examining the solution, we routinely obtain \eqref{eq:need4}.
\end{proof}


\section{The vectors $Ex^-_y$, $Ex^+_y$, $E{\hat x}$, $E{\hat y}$ }

 We continue to discuss the $Q$-polynomial distance-regular graph $\Gamma=(X, \mathcal R)$ with diameter $D\geq 3$.
Let $E$ denote a $Q$-polynomial primitive idempotent of $\Gamma$. In the previous section, we displayed some examples
for which the set $\lbrace E{\hat x} \vert x \in X\rbrace$ is Norton-balanced. Later in the paper we will discuss some more examples. In this section, we will develop some methods that will
facilitate the discussion.
\medskip

\begin{lemma} \label{lem:NBLD} Assume that the set $\lbrace E{\hat x} \vert x \in X\rbrace$ is Norton-balanced. Then
for  $x,y \in X$ 
the vectors $Ex^-_y$, $Ex^+_y$, $E{\hat x}$, $E{\hat y}$ are linearly dependent.
\end{lemma}
\begin{proof} By Lemma \ref{lem:clarify} and linear algebra.
\end{proof}

\noindent Consider the converse to Lemma  \ref{lem:NBLD}. For the moment, assume that for all $x, y \in X$ the vectors
 $Ex^-_y$, $Ex^+_y$, $E{\hat x}$, $E{\hat y}$ are linearly dependent. It is not necessarily the case that the set $\lbrace E{\hat x} \vert x \in X\rbrace$ is Norton-balanced;
the next result gives a counterexample.

\begin{lemma} \label{lem:q111} Assume that $a^*_1 =0$ and $a^*_2\not=0$. Then:
\begin{enumerate}
\item[\rm (i)]
the set $\lbrace E{\hat x} \vert x \in X\rbrace$ is not Norton-balanced;
\item[\rm (ii)]  for $x,y \in X$ we have
 \begin{align*}
 0 = (\theta^*_{i-1}-\theta^*_i) Ex^-_y+(\theta^*_{i+1}-\theta^*_{i})
Ex_y^+ +(\theta_1-\theta_2)\theta^*_i E{\hat x}+(\theta_2-\theta_0)E{\hat y}
\end{align*}
where $i = \partial (x,y)$. 
\end{enumerate}
\end{lemma}
\begin{proof} (i) We assume that $\lbrace E{\hat x} \vert x \in X\rbrace$ is Norton-balanced, and get a contradiction.
By Lemma \ref{lem:warmup1} and Definitions  \ref{def:STR}, \ref{def:NB} the set $\lbrace E{\hat x}\vert x \in X\rbrace$
is strongly balanced. By this and Lemma  \ref{lem:DBip},  $E$ is dual-bipartite or almost dual-bipartite. This contradicts
$a^*_2\not=0$. \\
\noindent  (ii) By Lemma \ref{lem:warmup1}, the left-hand side of   \eqref{eq:main} is equal to zero.
\end{proof}

\noindent Pick distinct $x,y \in X$. Our next general goal is to investigate the potential linear dependence among the vectors  $Ex^-_y$, $Ex^+_y$, $E{\hat x}$, $E{\hat y}$.
We will consider the following situations:
\begin{enumerate}
\item[\rm (i)] $Ex^-_y$, $E{\hat x}$, $E{\hat y}$ are linearly dependent;
\item[\rm (ii)] $Ex^+_y$, $E{\hat x}$, $E{\hat y}$ are linearly dependent;
\item[\rm (iii)] $Ex^-_y$,  $Ex^+_y$, $E{\hat x}$, $E{\hat y}$ are linearly dependent, but not (i), (ii).
\end{enumerate}
\noindent As we discuss these situations, we will need some parameters $\beta, \gamma, \gamma^*$ associated with $E$. We will
also need some facts about a certain 4-vertex configuration called a kite.
We review these topics in the next two sections.

\section{The parameters $\beta$, $\gamma$, $\gamma^*$}
\noindent We continue to discuss the $Q$-polynomial distance-regular graph $\Gamma=(X, \mathcal R)$ with diameter $D\geq 3$.
Let $E$ denote a $Q$-polynomial primitive idempotent of $\Gamma$. 
In this section,  we discuss some parameters $\beta$, $\gamma$, $\gamma^*$ associated with  $E$. 
\medskip

\noindent By \cite[p.~283]{bbit} the scalars
\begin{align}
\frac{\theta_{i-2} - \theta_{i+1}}{\theta_{i-1}-\theta_i}, \qquad \quad
\frac{\theta^*_{i-2} - \theta^*_{i+1}}{\theta^*_{i-1}-\theta^*_i}
\label{eq:betapone}
\end{align}
\noindent are equal and independent of $i$ for $2 \leq i\leq D-1$.
We denote this common value by $\beta+1$. By \cite[p.~283]{bbit}
 there exist real numbers $\gamma, \gamma^*$ such that both
\begin{align}\label{eq:gam}
\gamma = \theta_{i-1} - \beta \theta_i + \theta_{i+1}, \qquad \qquad
\gamma^* = \theta^*_{i-1} - \beta \theta^*_i + \theta^*_{i+1} 
\end{align}
\noindent for $1 \leq i \leq D-1$. The recurrences \eqref{eq:gam} can be solved in closed form.
We will focus on the sequence $\lbrace \theta^*_i \rbrace_{i=0}^D$; the sequence $\lbrace \theta_i\rbrace_{i=0}^D$ is similar. 
Let $\mathbb C$ denote the field of complex numbers.
There exists  $0 \not=q \in \mathbb C$  such that  $\beta=q+q^{-1}$. Note that $q=1$ iff  $\beta=2$, and $q=-1$ iff $\beta=-2$.
 By \cite[p.~286]{bbit} we have
\begin{align*} 
\begin{tabular}[t]{c|cc}
{\rm case }& {\rm $\theta^*_i$ closed form} &  $\gamma^*$
 \\
 \hline
 $\beta \not=\pm 2$ & $\theta^*_i = a + b q^i + c q^{-i}$ & $(2-\beta) a $ \\
 $\beta = 2$ & $ \theta^*_i = a + b i + c i^2$  & $2 c $\\
 $\beta=-2$ & $ \theta^*_i = a  + b (-1)^i + c i (-1)^i$ & $4a $
   \end{tabular}
\end{align*}
\noindent  In the above table, the $a,b,c$ are appropriate complex numbers.
The case $\gamma^*=0$ becomes important later in the paper. We now examine this case.

\begin{lemma} \label{lem:WZ} We refer to the above table.
\begin{enumerate}
\item[\rm (i)] Assume that $\beta \not=2$. Then  $\gamma^*=0$ if and only if $a=0$.
\item[\rm (ii)]  Assume that $\beta =2$. Then  $\gamma^*=0$ if and only if $c=0$.
\end{enumerate}
\end{lemma}
\begin{proof} Immediate from the above table.
\end{proof}

\noindent  A theorem of Leonard  \cite[p.~260]{banIto} gives detailed formulas for
the intersection numbers and Krein parameters of $\Gamma$.  In \cite[Section~20]{LSnotes} these formulas are derived using
 the theory of Leonard pairs. Later in the paper we will invoke the formulas,
using the notation of \cite[Section~20]{LSnotes}. The details of the formulas depend on the case of $\beta$ shown in the table above Lemma \ref{lem:WZ}. For each case,
there are some subcases as shown in the table below.
\begin{align*}
\begin{tabular}[t]{c|c}
{\rm case} & {\rm subcases}
 \\
 \hline
 $\beta \not=\pm 2$ & $q$-Racah, $q$-Hahn, dual $q$-Hahn,  $q$-Krawtchouk, \\
 &
 affine $q$-Krawtchouk, dual $q$-Krawtchouk
 \\
 $\beta = 2$ & Racah, Hahn, dual Hahn, Krawtchouk \\
 $\beta=-2$ &  Bannai/Ito
   \end{tabular}
\end{align*}

\begin{remark} \label{rem:LP} \rm In the theory of Leonard pairs, for the case $\beta \not=\pm 2$ there is a subcase called type IA in \cite[p.~260]{banIto} and
quantum $q$-Krawtchouk in \cite[Example~20.4]{LSnotes}. We did not include this subcase in the above table, because the subcase  does not occur for $Q$-polynomial distance-regular graphs \cite[Proposition~5.8]{dkt}.
\end{remark}


\section{Kites }
\noindent We continue to discuss the $Q$-polynomial distance-regular graph $\Gamma=(X, \mathcal R)$ with diameter $D\geq 3$.
Let $E$ denote a $Q$-polynomial primitive idempotent of $\Gamma$. In this section we discuss a certain 4-vertex configuration in $\Gamma$,  called a kite.
We also explain what it means for $\Gamma$ to be reinforced.

\begin{definition}\label{def:kite} \rm {(See \cite[Section~1]{kiteFree}.)} For $2 \leq i \leq D$, an {\it $i$-kite} in $\Gamma$ is a 4-tuple of vertices $(x,y,z,w)$ such that
\begin{align*}
&\partial(x,y)=i, \qquad \quad \partial(x,z) = 1, \qquad \quad \partial(y,z) = i-1, \\
& \partial(x,w)=1, \qquad \quad \partial(y,w) = i-1, \qquad \quad \partial(z,w) = 1.
\end{align*}
\end{definition}

\begin{definition} \label{def:zi} \rm For $2 \leq i \leq D$ define
\begin{align*}
z_i = \frac{{\hbox{\rm number of $i$-kites in $\Gamma$}}}{\vert X \vert k_i c_i}.
\end{align*}
We call $z_i$ the {\it $i$th kite number} of $\Gamma$.
\end{definition}

\noindent Shortly we will give a combinatorial interpretation of $z_i$.

\begin{definition}\label{def:ziFunction} \rm Pick $2 \leq i \leq D$ and $x,y,z \in X$ such that 
\begin{align*}
&\partial(x,y)=i, \qquad \quad \partial(x,z) = 1, \qquad \quad \partial(y,z) = i-1.
\end{align*}
 Define
\begin{align*}
\zeta_i(x,y,z) =  \vert \Gamma(x)  \cap \Gamma_{i-1}(y) \cap \Gamma(z) \vert.
\end{align*}
\noindent Note that $\zeta_i(x,y,z)$ is the number of vertices $w\in X$ such that $(x,y,z,w)$ is an $i$-kite. We call $\zeta_i$
the {\it $i$th kite function}.
\end{definition}

\begin{lemma} \label{lem:zzinterp} Referring to Definition \ref{def:ziFunction},  the scalar $\zeta_i(x,y,z)$ is an integer and $0 \leq  \zeta_i(x,y,z) \leq a_1$.
\end{lemma}
\begin{proof} By construction and since $a_1 = \vert \Gamma(x) \cap \Gamma(z)\vert$.
\end{proof}

\begin{note} \label{note:ave} \rm  For $2 \leq i \leq D$ the scalar $z_i$ has the following combinatorial interpretation.
Let $\Omega_i$ denote the set of 3-tuples of vertices $(x,y,z)$ such that
\begin{align*}
&\partial(x,y)=i, \qquad \quad \partial(x,z) = 1, \qquad \quad \partial(y,z) = i-1.
\end{align*}
Note that $\vert \Omega_i \vert = \vert X \vert k_i c_i$. We have
\begin{align*}
z_i = \frac{ \sum_{(x,y,z) \in \Omega_i} \zeta_i(x,y,z) }{ \vert \Omega_i \vert}.
\end{align*}
\noindent In other words, $z_i$ is the average value of $\zeta_i(x,y,z)$ over all $(x,y,z) \in \Omega_i$. We have $0 \leq z_i \leq a_1$ in view of
Lemma  \ref{lem:zzinterp}.
\end{note}

  \begin{lemma} \label{lem:kite} {\rm (See \cite[Theorem~2.11]{kiteFree}.)} For $2 \leq i \leq D$ we have $z_i=z_2 \alpha_i + a_1 \beta_i$,  where
   \begin{align} \label{eq:alphaForm}
   \alpha_i  &= \frac{(\theta^*_1-\theta^*_2)(\theta^*_0 + \theta^*_1 - \theta^*_{i-1}-\theta^*_i)}{(\theta^*_0-\theta^*_2)(\theta^*_{i-1}-\theta^*_i)},
   \\
   \label{eq:betaForm}
   \beta_i &= \frac{(\theta^*_0 - \theta^*_1)(\theta^*_2-\theta^*_i) - (\theta^*_1-\theta^*_2)(\theta^*_1-\theta^*_{i-1})}{(\theta^*_0-\theta^*_2)(\theta^*_{i-1}-\theta^*_i)}.
   \end{align}
   \end{lemma}
\noindent We mention some handy facts about the scalars $\lbrace \alpha_i\rbrace_{i=2}^D$, $\lbrace \beta_i\rbrace_{i=2}^D$.
\begin{lemma} We have
\begin{align*}
\alpha_i + \beta_i = \frac{\theta^*_1 - \theta^*_i}{\theta^*_{i-1}-\theta^*_i} \qquad \qquad (2 \leq i \leq D).
\end{align*}
\end{lemma}
\begin{proof}
 Use  \eqref{eq:alphaForm}, \eqref{eq:betaForm} and the table above Lemma \ref{lem:WZ}.
\end{proof}

\begin{lemma} \label{lem:aaform} For $2 \leq i \leq D-1$ we have
\begin{align*}
\alpha_i \alpha_{i+1} = \frac{ (\beta+2) (\theta^*_1 - \theta^*_2)^2 (\theta^*_0 - \theta^*_i)(\theta^*_1 - \theta^*_i)}{(\theta^*_0-\theta^*_2)^2 (\theta^*_{i-1} - \theta^*_i) (\theta^*_i-\theta^*_{i+1})}.
\end{align*}
\end{lemma}
\begin{proof} Use \eqref{eq:alphaForm}  and the table above Lemma \ref{lem:WZ}.
\end{proof}

\noindent  In Note \ref{note:ave} we discussed some averages. Next, we refine these averages.

\begin{definition} \label{def:ziBar} \rm Let $2 \leq i \leq D$ and $x,y \in X$ at distance $\partial(x,y)=i$. 
Note that $c_i = \vert \Gamma(x) \cap \Gamma_{i-1}(y) \vert $.
Define
\begin{align*}
\zeta_i(x,y, \ast)  = c^{-1}_i  \sum_{z \in \Gamma(x) \cap \Gamma_{i-1}(y)}  \zeta_i(x,y,z).
\end{align*}
In other words,  $\zeta_i(x,y, \ast)$ 
is  the average value of $\zeta_i(x,y,z)$ over all $z  \in \Gamma(x) \cap \Gamma_{i-1}(y)$. 
\end{definition}

\noindent In the next result, we clarify the meaning of $\zeta_i(x,y,\ast)$.
\begin{lemma} \label{lem:subraphVal}  Referring to Definition \ref{def:ziBar}, the scalar  $\zeta_i(x,y, \ast)$ 
is  the average valency of the induced subgraph
$ \Gamma(x) \cap \Gamma_{i-1}(y)$. 
\end{lemma}
\begin{proof} By the last sentence in Definition \ref{def:ziBar}.
\end{proof}

\noindent Let us emphasize a few points.
\begin{lemma}\label{lem:ziz} 
 Let $2 \leq i \leq D$ and $x,y \in X$ at distance $\partial(x,y)=i$. 
The following are equivalent:
\begin{enumerate}
\item[\rm (i)] for  $z \in \Gamma(x) \cap \Gamma_{i-1}(y)$ the integer $\zeta_i(x,y,z)$ is independent of $z$;
\item[\rm (ii)] $\zeta_i(x,y,z) = \zeta_i(x,y,\ast)$ for all $z \in \Gamma(x) \cap \Gamma_{i-1}(y)$;
\item[\rm (iii)] the induced subgraph  $\Gamma(x) \cap \Gamma_{i-1}(y)$ is regular.
\end{enumerate}
\end{lemma}
\begin{proof} By Definitions \ref{def:ziFunction}, \ref{def:ziBar} and Lemma  \ref{lem:subraphVal}.
\end{proof}

\begin{lemma} \label{lem:adjust} For $2 \leq i \leq D$,
\begin{align*}
z_i = \frac{1}{ \vert X \vert  k_i }   \sum_{\stackrel{ \scriptstyle x,y \in X }{ \scriptstyle \partial(x,y)=i}}  \zeta_i(x,y,\ast).
\end{align*}
\end{lemma}
\begin{proof} By Note \ref{note:ave} and Definition  \ref{def:ziBar}.
\end{proof} 

\noindent Pick an integer $i$ $(2 \leq i \leq D)$. It could happen that 
$ \zeta_i(x,y,\ast)$ is independent of $x,y$  $\bigl(x,y \in X, \partial(x,y)=i \bigr)$.

\begin{lemma}\label{lem:drgReinforce}
For $2 \leq i \leq D$ the following are equivalent:
\begin{enumerate}
\item[\rm (i)]  $ \zeta_i(x,y,\ast)$ is independent of $x,y$  $\bigl( x,y \in X, \partial(x,y)=i\bigr)$;
\item[\rm (ii)] $ \zeta_i(x,y,\ast)=z_i$ for all $x,y \in X$ at distance $\partial(x,y)=i$.
\end{enumerate}
\end{lemma}
\begin{proof} By Lemma \ref{lem:adjust}.
\end{proof}

\noindent Recall the notion of distance-transitivity from \cite[p.~136]{bcn}.
\begin{lemma} \label{lem:dtReinforce}
Assume that $\Gamma$ is distance-transitive. Then for $2 \leq i \leq D$
the equivalent conditions {\rm (i), (ii)} hold in Lemma \ref{lem:drgReinforce}.
\end{lemma}
\begin{proof} Routine.
\end{proof}

\noindent We have a comment.

\begin{lemma}\label{def:PziBar} \rm Pick $2 \leq i \leq D$ and $x, y, z \in X$ such that
\begin{align*}
&\partial(x,y)=i, \qquad \quad \partial(x,z) = 1, \qquad \quad \partial(y,z) = i-1.
\end{align*}
Then
\begin{align*}
a_1- \zeta_i (x,y,z)  =  \vert \Gamma(x)  \cap \Gamma_i(y) \cap \Gamma(z) \vert.
\end{align*}
\end{lemma}
\begin{proof} Routine using Definition \ref{def:ziFunction}.
\end{proof}

\begin{definition} \label{def:Pave} \rm
Let $2 \leq i \leq D$ and $y, z \in X$ at distance $\partial(y,z)=i-1$. Note that $b_{i-1} = \vert \Gamma_i(y) \cap \Gamma(z) \vert$.
Define
\begin{align*}
\zeta_i(\ast, y, z) = \frac{1}{b_{i-1}} \sum_{x \in \Gamma_i(y) \cap \Gamma(z)}  \zeta_i (x,y,z).
\end{align*}
In other words,  $\zeta_i(\ast, y, z)$ is the average value of $\zeta_i(x,y,z)$ over all $x \in \Gamma_i(y) \cap \Gamma(z)$.
\end{definition}
\noindent In the next result, we clarify the meaning of $\zeta_i(\ast,y,z)$.

\begin{lemma} \label{lem:interp2} Referring to Definition \ref{def:Pave}, the scalar $a_1 - \zeta_i(\ast, y,z)$ is equal to
the average valency of the induced subgraph $\Gamma_i(y) \cap \Gamma(z)$.
\end{lemma}
\begin{proof} By Lemma \ref{def:PziBar} and the last sentence in Definition \ref{def:Pave}.
\end{proof}

\noindent Let us emphasize a few points.
\begin{lemma}\label{lem:Pziz} 
Let $2\leq i \leq D$ and $y, z \in X$ at distance $\partial(y,z)=i-1$.
The following are equivalent:
\begin{enumerate}
\item[\rm (i)] for  $x \in \Gamma_i(y) \cap \Gamma(z)$ the integer $\zeta_i (x,y,z)$ is independent of $x$;
\item[\rm (ii)] $\zeta_i(x,y,z) = \zeta_i(\ast,y,z)$  for all $x \in \Gamma_i(y) \cap \Gamma(z)$;
\item[\rm (iii)] the  induced subgraph  $\Gamma_i(y) \cap \Gamma(z)$ is regular.
\end{enumerate}
\end{lemma}
\begin{proof} By Definition \ref{def:Pave} and Lemma \ref{lem:interp2}.
\end{proof}

\begin{lemma} \label{lem:2adjust} For $2 \leq i \leq D$,
\begin{align*}
z_i = \frac{1}{ \vert X \vert  k_{i-1}}   \sum_{\stackrel{ \scriptstyle y,z \in X }{ \scriptstyle \partial(y,z)=i-1}}  \zeta_i(\ast,y,z).
\end{align*}
\end{lemma}
\begin{proof} By Note \ref{note:ave} and Definition  \ref{def:Pave}, along with the fact that $k_i c_i = k_{i-1} b_{i-1}$.
\end{proof} 

\noindent Pick an integer $i$ $(2 \leq i \leq D)$. It could happen that 
$ \zeta_i(\ast,y,z)$ is independent of $y,z$  $\bigl(y,z \in X, \partial(y,z)=i-1 \bigr)$.

\begin{lemma}\label{lem:2drgReinforce}
For $2 \leq i \leq D$ the following are equivalent:
\begin{enumerate}
\item[\rm (i)]  $ \zeta_i(\ast, y,z)$ is independent of $y,z$  $\bigl( y,z \in X, \partial(y,z)=i-1\bigr)$;
\item[\rm (ii)] $ \zeta_i(\ast, y,z)=z_i$ for all $y,z \in X$ at distance $\partial(y,z)=i-1$.
\end{enumerate}
\end{lemma}
\begin{proof} By Lemma \ref{lem:2adjust}.
\end{proof}

\begin{lemma} \label{lem:2dtReinforce}
Assume that $\Gamma$ is distance-transitive. Then for $2 \leq i \leq D$
the equivalent conditions {\rm (i), (ii)} hold in Lemma \ref{lem:2drgReinforce}.
\end{lemma}
\begin{proof} Routine.
\end{proof}

\begin{definition}\label{def:Reinforce} \rm The graph $\Gamma$ is said to be {\it reinforced} whenever
the following {\rm (i), (ii)} holds for $2 \leq i \leq D$:
\begin{enumerate}
\item[\rm (i)]  $ \zeta_i(x,y,\ast)$ is independent of $x,y$  $\bigl( x,y \in X, \partial(x,y)=i\bigr)$;
\item[\rm (ii)]  $ \zeta_i(\ast, y,z)$ is independent of $y,z$  $\bigl( y,z \in X, \partial(y,z)=i-1\bigr)$.
\end{enumerate}
\end{definition}


\begin{lemma}\label{lem:WhenReinforce}
Assume that $\Gamma$ is distance-transitive. Then $\Gamma$ is reinforced.
\end{lemma}
\begin{proof} By Lemmas \ref{lem:dtReinforce} and \ref{lem:2dtReinforce}.
\end{proof} 

\begin{lemma}\label{lem:2WhenReinforce}
Assume that the kite function $\zeta_i$ is constant for $2 \leq i \leq D$.  Then $\Gamma$ is reinforced.
\end{lemma}
\begin{proof} For $2 \leq i \leq D$ we have $\zeta_i(x,y,z)=z_i$ for all $x,y,z\in X$ such that
\begin{align*}
\partial(x,y)=i, \qquad \quad \partial(x,z)=1, \qquad \quad \partial(y,z)=i-1.
\end{align*}
\end{proof}

\section{When are $Ex^-_y, E{\hat x}, E{\hat y}$  linearly dependent}
\noindent We continue to discuss the $Q$-polynomial distance-regular graph $\Gamma=(X, \mathcal R)$ with diameter $D\geq 3$.
Let $E$ denote a $Q$-polynomial primitive idempotent of $\Gamma$. Pick $2 \leq i \leq D$ and $x, y \in X$ at distance $\partial(x,y)=i$.
In this section, our goal is to obtain a necessary and sufficient condition for the vectors $Ex^-_y$, $E\hat x$, $E\hat y$ to be linearly dependent. 
 In view of Lemma \ref{lem:Antip},
we assume that $i \not=D$ if $\Gamma $ is an antipodal 2-cover. By Lemmas \ref{lem:NotAntip}, \ref{lem:Antip} the vectors  $E{\hat x}$, $E{\hat y}$ are linearly independent.

\begin{lemma}\label{lem:innerPr1}
We have
\begin{align*}
\vert X \vert \langle Ex^-_y, E{\hat x} \rangle = c_i \theta^*_1, \qquad \qquad \vert X \vert  \langle Ex^-_y, E{\hat y} \rangle = c_i \theta^*_{i-1}.
\end{align*}
\end{lemma}
\begin{proof} By Lemma \ref{lem:geometry}(i) 
and the construction.
\end{proof}

\begin{lemma}\label{lem:innerPrz}
For $z \in \Gamma(x) \cap \Gamma_{i-1}(y)$ we have
\begin{align} \label{eq:zChoice}
\vert X \vert \langle  E x^-_y, E{\hat z}\rangle  = \theta^*_0 + \zeta_i (x,y,z) \theta^*_1 + \bigl(c_i - \zeta_i(x,y,z) - 1\bigr) \theta^*_2.
\end{align}
\end{lemma}
\begin{proof} By construction, $\vert \Gamma(x) \cap \Gamma_{i-1}(y)\vert = c_i$. Also by construction, any two distinct vertices in  $ \Gamma(x) \cap \Gamma_{i-1}(y)$ are at distance 1 or 2 in $\Gamma$.
By Definition \ref{def:ziFunction}, the vertex $z$ is adjacent to exactly $\zeta_i (x,y,z) $ vertices in  $ \Gamma(x) \cap \Gamma_{i-1}(y)$. The result follows in view of Lemma  \ref{lem:geometry}(i).
\end{proof}

\begin{lemma} \label{lem:innerPr} We have 
\begin{align*}
\vert X \vert c^{-1}_i \Vert E x^-_y \Vert^2 = \theta^*_0 + \zeta_i(x,y,\ast)  \theta^*_1 + \bigl(c_i - \zeta_i(x,y,\ast) - 1 \bigr) \theta^*_2.
\end{align*}
\end{lemma}
\begin{proof} Compute the average of  \eqref{eq:zChoice} over all  $z \in \Gamma(x) \cap \Gamma_{i-1}(y)$. Evaluate the
result   using Definitions \ref{def:xyxy}, \ref{def:ziBar}.
\end{proof}

\noindent For notational convenience, define
\begin{align}
\label{eq:risi}
r_i = c_i  \frac{ \theta^*_0 \theta^*_1 - \theta^*_{i-1} \theta^*_i }{\theta^{*2}_0 - \theta^{*2}_i}, \qquad \qquad
s_i = c_i  \frac{ \theta^*_0 \theta^*_{i-1} - \theta^*_{1} \theta^*_i }{\theta^{*2}_0 - \theta^{*2}_i}.
\end{align}

\begin{lemma} \label{lem:rs} For $r,s \in \mathbb R$ the following are equivalent:
\begin{enumerate}
\item[\rm (i)] $Ex^-_y-r E{\hat x} - s E{\hat y}$ is orthogonal to each of $E{\hat x}$, $E{\hat y}$;
\item[\rm (ii)] both
\begin{align*}
c_i \theta^*_1 = r \theta^*_0 + s \theta^*_i, \qquad \qquad  c_i \theta^*_{i-1} = r \theta^*_i + s \theta^*_0.
\end{align*}
\item[\rm (iii)] $r=r_i $ and $s=s_i$.
\end{enumerate}
\end{lemma}
\begin{proof} ${\rm (i)} \Leftrightarrow {\rm (ii)}$ Use Lemma \ref{lem:geometry}(i),(ii) and Lemma  \ref{lem:innerPr1}.  \\
\noindent ${\rm (ii)} \Leftrightarrow {\rm (iii)}$ By linear algebra and $\theta^{*2}_0 \not=\theta^{*2}_i$.
\end{proof}

\begin{definition} \label{def:ziMinus} \rm Define the real number
\begin{align*}
z^-_i = \frac{ r_i \theta^*_1 + s_i \theta^*_{i-1} - \theta^*_0 - (c_i-1)\theta^*_2}{\theta^*_1 - \theta^*_2}.
\end{align*}
\end{definition}

\begin{lemma} \label{lem:SNorm} We have
\begin{align}\label{eq:mine}
 \Vert E x^-_y  - r_i E{\hat x} -s_i E{\hat y} \Vert^2 =\vert X \vert^{-1}  c_i (\theta^*_1 - \theta^*_2) \bigl(\zeta_i(x,y,\ast) - z^-_i \bigr).
\end{align}
\end{lemma}
\begin{proof} By Lemmas \ref{lem:innerPr1}, \ref{lem:innerPr},  \ref{lem:rs} each side of \eqref{eq:mine} is equal to
\begin{align*}
 \langle  E x^-_y - r_i E{\hat x} -s_i E{\hat y}, Ex^-_y\rangle.
\end{align*}
\end{proof}


\begin{lemma}\label{lem:factor} We have
\begin{align} \label{eq:factor}
\frac{ \zeta_i(x,y,\ast) - z^-_i }{\theta^*_1 - \theta^*_2} \geq 0.
\end{align}
\end{lemma}
\begin{proof} By Lemma \ref{lem:SNorm}.
\end{proof}

\begin{lemma} \label{lem:2cases} The following hold.
\begin{enumerate}
\item[\rm (i)] Assume that $\theta^*_1 > \theta^*_2$.  Then $\zeta_i(x,y,\ast) \geq z^-_i$. 
\item[\rm (ii)] Assume that  $\theta^*_1 < \theta^*_2$.  Then $\zeta_i(x,y,\ast) \leq z^-_i$. 
\end{enumerate}
\end{lemma}
\begin{proof} By Lemma \ref{lem:factor}.
\end{proof}

\begin{proposition}\label{prop:AT} The following are equivalent:
\begin{enumerate}
\item[\rm (i)] equality holds in \eqref{eq:factor};
\item[\rm (ii)]  $Ex^-_y = r_i E{\hat x} + s_i E{\hat y}$;
\item[\rm (iii)] the vectors $Ex^-_y$, $E{\hat x}$, $E{\hat y}$ are linearly dependent;
\item[\rm (iv)] $\zeta_i(x,y,z)=z^-_i$  for all $z \in \Gamma(x) \cap \Gamma_{i-1}(y)$;
\item[\rm (v)] $\zeta_i(x,y,\ast) = z^-_i$.
\end{enumerate}
\noindent Assume that {\rm (i)--(v)} hold. Then $z^-_i$ is an integer and $0 \leq z^-_i \leq a_1$.
\end{proposition}
\begin{proof} 
  ${\rm (i)} \Leftrightarrow {\rm (ii)}$  By Lemma \ref{lem:SNorm}. \\
\noindent  ${\rm (ii)} \Rightarrow {\rm (iii)}$ Clear. \\
\noindent  ${\rm (iii)} \Rightarrow {\rm (ii)}$ The vectors  $E{\hat x}$, $E{\hat y}$ are linearly independent, so there exist $r,s \in \mathbb R$ such that
 $Ex^-_y = r E{\hat x} + s E{\hat y}$. The vector  $Ex^-_y - r E{\hat x} - s E{\hat y}$ is equal to zero, so it is orthogonal to $E{\hat x}$ and $E{\hat y}$. We have $r=r_i$
 and $s=s_i$ by Lemma \ref{lem:rs}. \\
 \noindent  ${\rm (ii)} \Rightarrow {\rm (iv)}$ 
  Take the inner product of $E{\hat z}$ with each side of   $Ex^-_y = r_i E{\hat x} + s_i E{\hat y}$; this yields
 \begin{align*}
 \theta^*_0 + \zeta_i(x,y,z) \theta^*_1 + \bigl(c_i - \zeta_i(x,y,z) -1\bigr) \theta^*_2 = r_i \theta^*_1 + s_i \theta^*_{i-1}.
 \end{align*}
 Solve this equation for $\zeta_i(x,y,z)$ to find $\zeta_i(x,y,z)=z^-_i$.
\\
 \noindent  ${\rm (iv)} \Rightarrow {\rm (v)}$ By Definition \ref{def:ziBar}. \\
 \noindent  ${\rm (v)} \Rightarrow {\rm (i)}$ Clear. \\
 \noindent We have shown that (i)--(v) are equivalent. We now assume that (i)--(v) hold. By (iv) and Lemma \ref{lem:zzinterp}, the scalar
 $z^-_i$ is an integer  and $0 \leq z^-_i \leq a_1$.
 \end{proof}

\begin{lemma}\label{lem:more} Referring to Proposition \ref{prop:AT}, assume that the equivalent conditions {\rm (i)--(v)} hold. Then
for all integers $j$ $(i \leq j \leq D)$ we have
\begin{align}
\theta^*_{j-1} &= \theta^*_j  \frac{ \theta^*_0 \theta^*_1 - \theta^*_{i-1} \theta^*_i }{\theta^{*2}_0 - \theta^{*2}_i} +
\theta^*_{j-i}  \frac{ \theta^*_0 \theta^*_{i-1} - \theta^*_{1} \theta^*_i }{\theta^{*2}_0 - \theta^{*2}_i}; \label{eq:more1} \\
\theta^*_{j-i+1} &= \theta^*_{j-i}  \frac{ \theta^*_0 \theta^*_1 - \theta^*_{i-1} \theta^*_i }{\theta^{*2}_0 - \theta^{*2}_i} +
\theta^*_{j}  \frac{ \theta^*_0 \theta^*_{i-1} - \theta^*_{1} \theta^*_i }{\theta^{*2}_0 - \theta^{*2}_i}. \label{eq:more2}
\end{align}
\end{lemma}
\begin{proof}  To obtain \eqref{eq:more1}, pick $w \in X$ such that $\partial(x,w)=j$ and $\partial(y,w)=j-i$. Take the inner product of
$E{\hat w}$ with each term in $Ex^-_y = r_i E{\hat x} + s_i E{\hat y}$. Evaluate the resulting equation using Lemma \ref{lem:geometry}(i) and \eqref{eq:risi}.
To obtain \eqref{eq:more2}, repeat the calculation using $\partial(x,w)=j-i$ and $\partial(y,w)=j$. 
\end{proof}

\begin{lemma} \label{lem:more2} Assume that $2 \leq i \leq D-1$. Then the following are equivalent:
\begin{enumerate}
\item[\rm (i)] the equations  \eqref{eq:more1}, \eqref{eq:more2} hold for all integers $j$ $(i \leq j \leq D)$;
\item[\rm (ii)] $\gamma^*=0$.
\end{enumerate}
\end{lemma}
\begin{proof} Use the forms  in the table above Lemma \ref{lem:WZ}.
\end{proof}

\begin{corollary}\label{cor:gamNZ} Assume that  $2 \leq i \leq D-1$ and $\gamma^* \not=0$. Then  $Ex^-_y$, $E{\hat x}$, $E{\hat y}$ are linearly independent.
Moreover,
\begin{align*} 
\frac{ \zeta_i(x,y,\ast) - z^-_i }{\theta^*_1 - \theta^*_2} > 0.
\end{align*}
\end{corollary}
\begin{proof} By Proposition  \ref{prop:AT} and  Lemmas \ref{lem:more}, \ref{lem:more2}.
\end{proof}

\section{When are $Ex^+_y, E{\hat x}, E{\hat y}$  linearly dependent}

\noindent We continue to discuss the $Q$-polynomial distance-regular graph  $\Gamma=(X,\mathcal R)$ with diameter $D\geq 3$.
Let $E$ denote a $Q$-polynomial primitive idempotent of $\Gamma$. Pick $1 \leq i \leq D-1$ and $x,y \in X$ at distance 
$\partial(x,y)=i$. In this section, our goal is to obtain  a necessary and sufficient condition for the vectors $Ex^+_y$, $E\hat x$, $E\hat y$ to be linearly dependent.
\medskip

\noindent By Lemmas \ref{lem:NotAntip}, \ref{lem:Antip} the vectors $E{\hat x}, E{\hat y}$ are linearly independent.

\begin{lemma}\label{lem:PinnerPr1}
We have
\begin{align*}
\vert X \vert \langle Ex^+_y, E{\hat x} \rangle = b_i \theta^*_1, \qquad \qquad \vert X \vert  \langle Ex^+_y, E{\hat y} \rangle = b_i \theta^*_{i+1}.
\end{align*}
\end{lemma}
\begin{proof} By Lemma \ref{lem:geometry}(i) 
and the construction.
\end{proof}

\begin{lemma}\label{lem:PinnerPr}
For $z \in \Gamma(x) \cap \Gamma_{i+1}(y)$ we have
\begin{align} \label{eq:zInner}
\vert X \vert \langle  E x^+_y, E{\hat z}\rangle  =              \theta^*_0 + \bigl(a_1-\zeta_{i+1} (z,y,x)\bigr) \theta^*_1 + \bigl(b_i - 1-a_1+ \zeta_{i+1}(z,y,x) \bigr) \theta^*_2.
\end{align}
\end{lemma}
\begin{proof} By construction, $\vert \Gamma(x) \cap \Gamma_{i+1}(y)\vert = b_i$. Also by construction, any two distinct vertices in  $ \Gamma(x) \cap \Gamma_{i+1}(y)$ are at distance 1 or 2 in $\Gamma$.
By Lemma \ref{def:PziBar},  the vertex $z$ is adjacent to exactly $a_1-\zeta_{i+1} (z,y,x) $ vertices in  $ \Gamma(x) \cap \Gamma_{i+1}(y)$. The result follows in view of Lemma  \ref{lem:geometry}(i).
\end{proof}

\begin{lemma} \label{lem:SN2} We have
\begin{align*}
\vert X \vert b^{-1}_i \Vert E x^+_y \Vert^2 =  \theta^*_0 + \bigl(a_1-\zeta_{i+1}(\ast, y,x) \bigr) \theta^*_1 + \bigl(b_i - 1-a_1+ \zeta_{i+1}(\ast, y,x)\bigr) \theta^*_2.
\end{align*}
\end{lemma}
\begin{proof} Compute the average of \eqref{eq:zInner} over all $z \in \Gamma(x) \cap \Gamma_{i+1}(y)$. Evaluate the resulting equation using  Definitions \ref{def:xyxy},  \ref{def:Pave}.
\end{proof}

\noindent For notational convenience, define
\begin{align}
\label{eq:Prisi}
R_i = b_i  \frac{ \theta^*_0 \theta^*_1 - \theta^*_{i+1} \theta^*_i }{\theta^{*2}_0 - \theta^{*2}_i}, \qquad \qquad
S_i = b_i  \frac{ \theta^*_0 \theta^*_{i+1} - \theta^*_{1} \theta^*_i }{\theta^{*2}_0 - \theta^{*2}_i}.
\end{align}

\begin{lemma} \label{lem:Prs} For $R, S \in \mathbb R$ the following are equivalent:
\begin{enumerate}
\item[\rm (i)] $Ex^+_y-R E{\hat x} - S E{\hat y}$ is orthogonal to each of $E{\hat x}$, $E{\hat y}$;
\item[\rm (ii)] both
\begin{align*}
b_i \theta^*_1 = R \theta^*_0 + S \theta^*_i, \qquad \qquad  b_i \theta^*_{i+1} = R \theta^*_i + S \theta^*_0.
\end{align*}
\item[\rm (iii)] $R=R_i $ and $S=S_i$.
\end{enumerate}
\end{lemma}
\begin{proof} Similar to the proof of Lemma \ref{lem:rs}.
\end{proof}

\begin{definition} \label{def:ziPlus} \rm Define the real number
\begin{align*}
z^+_{i+1} = \frac{ \theta^*_0 + a_1 \theta^*_1 + (b_i-1-a_1)\theta^*_2 - R_i \theta^*_1 - S_i \theta^*_{i+1}}{\theta^*_1 - \theta^*_2}.
\end{align*}
\end{definition}

\begin{lemma} \label{lem:PSNorm} We have
\begin{align} \label{eq:PSNorm2}
\Vert E x^+_y - R_i E{\hat x} -S_i E{\hat y} \Vert^2 = \vert X \vert^{-1} b_i (\theta^*_1 - \theta^*_2) \bigl( z^+_{i+1} - \zeta_{i+1}(\ast,y,x)   \bigr).
\end{align}
\end{lemma}
\begin{proof} By Lemmas \ref{lem:PinnerPr1},  \ref{lem:SN2}, \ref{lem:Prs} each side of  \eqref{eq:PSNorm2} is equal to
\begin{align*}
\langle E x^+_y - R_i E{\hat x} -S_i E{\hat y}, Ex^+_y \rangle.
\end{align*}
\end{proof}

\begin{lemma} \label{lem:factor2} We have
\begin{align} \label{eq:PminIneq}
\frac{  z^+_{i+1} - \zeta_{i+1}(\ast,y,x) }{ \theta^*_1 - \theta^*_2} \geq 0.
\end{align}
\end{lemma}
\begin{proof} By Lemma \ref{lem:PSNorm}.
\end{proof}

\begin{lemma} \label{lem:2casesP} The following hold.
\begin{enumerate}
\item[\rm (i)] Assume that $\theta^*_1 > \theta^*_2$.  Then $\zeta_{i+1} (\ast, y,x) \leq z^+_{i+1}$. 
\item[\rm (ii)] Assume that  $\theta^*_1 < \theta^*_2$.  Then $\zeta_{i+1}(\ast,y,x) \geq z^+_{i+1}$. 
\end{enumerate}
\end{lemma}
\begin{proof} By Lemma \ref{lem:factor2}.
\end{proof}

\begin{proposition}\label{prop:PAT} The following are equivalent:
\begin{enumerate}
\item[\rm (i)] equality holds in \eqref{eq:PminIneq};
\item[\rm (ii)]  $Ex^+_y = R_i E{\hat x} + S_i E{\hat y}$;
\item[\rm (iii)] the vectors $Ex^+_y$, $E{\hat x}$, $E{\hat y}$ are linearly dependent;
\item[\rm (iv)] $\zeta_{i+1}(z,y,x)=z^+_{i+1}$  for all $z \in \Gamma(x) \cap \Gamma_{i+1}(y)$;
\item[\rm (v)] $\zeta_{i+1}(\ast,y,x) = z^+_{i+1}$.
\end{enumerate}
\noindent Assume that {\rm (i)--(v)} hold. Then $z^+_{i+1}$ is an integer and $0 \leq z^+_{i+1} \leq a_1$.
\end{proposition}
\begin{proof}  Similar to the proof of Proposition \ref{prop:AT}.
\end{proof}

\begin{lemma}\label{lem:Pmore} Referring to Proposition \ref{prop:PAT}, assume that the equivalent conditions {\rm (i)--(v)} hold. Then
for all integers $j$ $(0 \leq j \leq i)$ we have
\begin{align}
 \theta^*_{j+1} = \theta^*_j  \frac{ \theta^*_0 \theta^*_1 - \theta^*_{i+1} \theta^*_i }{\theta^{*2}_0 - \theta^{*2}_i} +  \theta^*_{i-j}
  \frac{ \theta^*_0 \theta^*_{i+1} - \theta^*_{1} \theta^*_i }{\theta^{*2}_0 - \theta^{*2}_i}.
   \label{eq:Pmore1}
\end{align}
\end{lemma}
\begin{proof}  Pick $w \in X$ such that $\partial(x,w)=j$ and $\partial(y,w)=i-j$. Take the inner product of
$E{\hat w}$ with each term in $Ex^+_y = R_i E{\hat x} + S_i E{\hat y}$. Evaluate the resulting equation using Lemma  \ref{lem:geometry}(i) and \eqref{eq:Prisi}.
\end{proof}

\noindent Recall the parameter $\gamma^*$ from \eqref{eq:gam}.

\begin{lemma} \label{lem:Pmore2} Assume that $2 \leq i \leq D-1$. Then the following are equivalent:
\begin{enumerate}
\item[\rm (i)] the equation  \eqref{eq:Pmore1}  holds for all integers $j$ $(0 \leq j \leq i)$;
\item[\rm (ii)] $\gamma^*=0$.
\end{enumerate}
\end{lemma}
\begin{proof} Use the forms  in the table above Lemma \ref{lem:WZ}.
\end{proof}

\begin{corollary}\label{cor:PgamNZ} Assume that $2 \leq i \leq D-1$ and $\gamma^* \not=0$. Then  $Ex^+_y$, $E{\hat x}$, $E{\hat y}$ are linearly independent.
Moreover,
\begin{align*}
\frac{  z^+_{i+1} - \zeta_{i+1}(\ast,y,x) }{ \theta^*_1 - \theta^*_2} > 0.
\end{align*}
\end{corollary}
\begin{proof} By Proposition  \ref{prop:PAT} and Lemmas \ref{lem:Pmore},  \ref{lem:Pmore2}.
\end{proof}

\section{When are $Ex^-_y, Ex^+_y, E{\hat x}, E{\hat y}$  linearly dependent}

 We continue to discuss the $Q$-polynomial distance-regular graph  $\Gamma=(X,\mathcal R)$ with diameter $D\geq 3$.
 Let $E$ denote a $Q$-polynomial primitive idempotent of $\Gamma$. Pick $2 \leq i \leq D-1$ and
 $x,y \in X$ at distance $\partial(x,y)=i$. In this section, our goal is to obtain a necessary and sufficient condition for 
 the vectors $Ex^-_y, Ex^+_y, E{\hat x}, E{\hat y}$ to be  linearly dependent.
 \medskip

\noindent By Lemmas \ref{lem:NotAntip}, \ref{lem:Antip} the vectors $E{\hat x}, E{\hat y}$ are linearly independent.
 Recall the scalars $r_i, s_i$ from \eqref{eq:risi}. By Lemma \ref{lem:rs} the vector $Ex^-_y-r_i E{\hat x} - s_i E{\hat y}$ is orthogonal to each of $E{\hat x}, E{\hat y}$. 
 Recall the scalars $R_i, S_i$ from \eqref{eq:Prisi}. By Lemma \ref{lem:Prs} the vector $Ex^+_y-R_i E{\hat x} - S_i E{\hat y}$ is orthogonal to each of $E{\hat x}, E{\hat y}$. 
 \begin{lemma} \label{lem:4vs2}
 The following are equivalent:
 \begin{enumerate}
 \item[\rm (i)]  the vectors $Ex^-_y, Ex^+_y, E{\hat x}, E{\hat y}$ are  linearly dependent;
 \item[\rm (ii)]  the vectors $Ex^-_y-r_i E{\hat x} - s_i E{\hat y}$ and  $Ex^+_y-R_i E{\hat x} - S_i E{\hat y}$ are linearly dependent.
 \end{enumerate}
 \end{lemma}
 \begin{proof} By the comments above the lemma statement.
 \end{proof}
 
\noindent Next we compute the matrix of inner products for  $Ex^-_y-r_i E{\hat x} - s_i E{\hat y}$ and  $Ex^+_y-R_i E{\hat x} - S_i E{\hat y}$.
We computed $\Vert Ex^-_y-r_i E{\hat x} - s_i E{\hat y} \Vert^2$  in Lemma \ref{lem:SNorm}.
 We computed $\Vert Ex^+_y-R_i E{\hat x} - S_i E{\hat y} \Vert^2$  in Lemma \ref{lem:PSNorm}.


\begin{lemma} \label{lem:pm} We have
\begin{align*}
&\vert X \vert \langle Ex^-_y-r_i E{\hat x} - s_i E{\hat y}, Ex^+_y-R_i E{\hat x} - S_i E{\hat y} \rangle \\
&\qquad = c_i b_i \theta^*_2 - r_i b_i \theta^*_1 - s_i b_i \theta^*_{i+1} = c_i b_i \theta^*_2 - R_i c_i \theta^*_1 - S_i c_i \theta^*_{i-1}
\\
&\qquad = c_i b_i \biggl(\theta^*_2 - \frac{\theta^*_0 \theta^{*2}_1 - \theta^*_1 \theta^*_{i-1} \theta^*_i + \theta^*_0 \theta^*_{i-1} \theta^*_{i+1} - \theta^*_1 \theta^*_i \theta^*_{i+1} }{\theta^{*2}_0-\theta^{*2}_i}    \biggr).
\end{align*}
\end{lemma}
\begin{proof} Use  \eqref{eq:risi} and  Lemma \ref{lem:rs}, or else  \eqref{eq:Prisi} and Lemma \ref{lem:Prs}.
\end{proof}

\begin{lemma} \label{lem:gams2} We have
\begin{align} \label{eq:big}
\theta^*_2 - \frac{\theta^*_0 \theta^{*2}_1 - \theta^*_1 \theta^*_{i-1} \theta^*_i + \theta^*_0 \theta^*_{i-1} \theta^*_{i+1} - \theta^*_1 \theta^*_i \theta^*_{i+1} }{\theta^{*2}_0-\theta^{*2}_i} = 
\gamma^* \frac{\theta^*_i - \theta^*_1}{\theta^*_i+\theta^*_0}.
\end{align}
\end{lemma}
\begin{proof} Use the forms  in the table above Lemma \ref{lem:WZ}.
\end{proof}

\noindent By the Cauchy-Schwarz inequality,
 \begin{align}
 \label{eq:mainInq}
 \begin{split}
& \Vert Ex^-_y-r_i E{\hat x} - s_i E{\hat y} \Vert^2          \Vert Ex^+_y-R_i E{\hat x} - S_i E{\hat y} \Vert^2 \\
&\qquad \geq   \langle Ex^-_y-r_i E{\hat x} - s_i E{\hat y}, Ex^+_y-R_i E{\hat x} - S_i E{\hat y} \rangle^2.
\end{split}
  \end{align}
\noindent Equality holds in \eqref{eq:mainInq} if and only if the following vectors are linearly dependent:
\begin{align*}
 Ex^-_y-r_i E{\hat x} - s_i E{\hat y}, \qquad \qquad  Ex^+_y-R_i E{\hat x} - S_i E{\hat y}.
\end{align*}

 \begin{lemma} \label{lem:mainINEQ} We have
 \begin{align}
 \label{eq:mainINEQ}
&\bigl(\zeta_i(x,y,\ast) - z^-_i \bigr)\bigl(z^+_{i+1} -  \zeta_{i+1}(\ast,y,x )\bigr) \geq  c_i b_i   \biggl(\frac{\gamma^*}{ \theta^*_1-\theta^*_2} \, \frac{\theta^*_i - \theta^*_1}{\theta^*_i+\theta^*_0}  \biggr)^2.
 \end{align}
 \end{lemma}
 \begin{proof} Evaluate the terms in \eqref{eq:mainInq} using Lemmas  \ref{lem:SNorm},  \ref{lem:PSNorm},  \ref{lem:pm},  \ref{lem:gams2}.
  \end{proof}
 
\begin{proposition} \label{lem:WhenZ} The following {\rm (i)--(iii)} are equivalent:
\begin{enumerate}
\item[\rm (i)] equality holds in  \eqref{eq:mainINEQ};
\item[\rm (ii)] the vectors  $Ex^-_y, Ex^+_y,  E{\hat x}, E{\hat y} $ are linearly dependent;
\item[\rm (iii)] the vectors  $Ex^-_y-r_i E{\hat x} - s_i E{\hat y} $ and   $Ex^+_y-R_i E{\hat x} - S_i E{\hat y} $ are linearly dependent.
\end{enumerate}
\end{proposition} 
\begin{proof} By Lemma \ref{lem:4vs2} and the comments above Lemma \ref{lem:mainINEQ}. 
\end{proof}

\noindent Assume for the moment that the equivalent conditions (i)--(iii) hold in Proposition  \ref{lem:WhenZ}. Our next goal is to find the dependency between the vectors in part (iii).
\medskip

\noindent Until further notice, assume that $\gamma^*=0$.  By Lemmas \ref{lem:pm}, \ref{lem:gams2}  the following vectors are orthogonal:
\begin{align} \label{eq:twov}
 Ex^-_y-r_i E{\hat x} - s_i E{\hat y}, \qquad \qquad                 Ex^+_y-R_i E{\hat x} - S_i E{\hat y}.
\end{align}
\noindent The inequality  \eqref{eq:mainINEQ} becomes
\begin{align*}
&\bigl(\zeta_i(x,y,\ast) - z^-_i \bigr)\bigl(z^+_{i+1} -  \zeta_{i+1}(\ast,y,x )\bigr) \geq 0.
 \end{align*}

\begin{lemma} \label{cor:situation}
Assume that $\gamma^*=0$, and the equivalent conditions  {\rm (i)--(iii)} hold in Proposition \ref{lem:WhenZ}.  Then 
 at least one of the vectors \eqref{eq:twov} is equal to zero.
\end{lemma}
\begin{proof} The vectors \eqref{eq:twov} are orthogonal and linearly dependent.
\end{proof}

\begin{lemma} \label{lem:situation2} Assume that $\gamma^* =0$, and the equivalent conditions  {\rm (i)--(iii)} hold in Proposition \ref{lem:WhenZ}. 
Then
\begin{align*}
Ex^-_y, Ex^+_y \in {\rm Span}\lbrace E{\hat x}, E{\hat y}, E{\hat x} \star E{\hat y}\rbrace.
\end{align*}
\end{lemma}
\begin{proof} By  \eqref{eq:main} and Lemma \ref{cor:situation}.
\end{proof}

\noindent 
For the rest of this section, assume that $\gamma^*\not=0$. By Corollary \ref{cor:gamNZ} the vectors  $Ex^-_y,  E{\hat x}, E{\hat y}$ are  linearly independent.
By Corollary \ref{cor:PgamNZ} the vectors $Ex^+_y, E{\hat x}, E{\hat y}$ are  linearly independent.
\medskip

\begin{lemma} \label{lem:WhenZ3} Assume that $\gamma^* \not=0$, and the equivalent conditions  {\rm (i)--(iii)} hold in Proposition \ref{lem:WhenZ}. 
Then
\begin{align} \label{eq:lambdaEQ}
 Ex^-_y-r_i E{\hat x} - s_i E{\hat y} = \lambda_i \bigl( Ex^+_y-R_i E{\hat x} - S_i E{\hat y} \bigr),
\end{align}
\noindent where
\begin{align*}
\frac{\lambda_i}{\theta^*_1-\theta^*_2}  &= \frac{\zeta_i(x,y,\ast) - z^-_i }{\theta^*_i - \theta^*_1} \, \frac{\theta^*_i + \theta^*_0}{ \gamma^* b_i} , \qquad \quad
\frac{ \lambda^{-1}_i }{\theta^*_1 - \theta^*_2} = \frac{ z^+_{i+1}-\zeta_{i+1}(\ast,y,x) } {\theta^*_i - \theta^*_1} \,\frac{\theta^*_i + \theta^*_0}{\gamma^* c_i}.
\end{align*}
\end{lemma}
\begin{proof} By Proposition \ref{lem:WhenZ}(ii) there exists $\lambda_i \in \mathbb R$ such that  \eqref{eq:lambdaEQ} holds. To obtain $\lambda_i$, take the
inner product of each side of  \eqref{eq:lambdaEQ} with  $Ex^-_y-r_i E{\hat x} - s_i E{\hat y}$ or $Ex^+_y-R_i E{\hat x} - S_i E{\hat y}$.
\end{proof}

\begin{lemma} \label{lem:WhenZ5} Assume that $\gamma^* \not=0$, and the equivalent conditions  {\rm (i)--(iii)} hold in Proposition \ref{lem:WhenZ}. 
Assume that the scalar $\lambda_i$ from Lemma \ref{lem:WhenZ3}  satisfies
\begin{align*}
\lambda_i \not=\frac{\theta^*_{i}-\theta^*_{i+1}}{\theta^*_{i-1}-\theta^*_i}.
\end{align*}
Then
\begin{align*}
Ex^-_y, Ex^+_y \in {\rm Span}\lbrace E{\hat x}, E{\hat y}, E{\hat x} \star E{\hat y}\rbrace.
\end{align*}
\end{lemma}
\begin{proof} By  \eqref{eq:main} and \eqref{eq:lambdaEQ}.
\end{proof}

\begin{lemma}  \label{lem:WhenZ4} Assume that $\gamma^*\not=0$, and the equivalent conditions  {\rm (i)--(iii)} hold in Proposition \ref{lem:WhenZ}. 
\begin{enumerate}
\item[\rm (i)] For all $z \in \Gamma(x) \cap \Gamma_{i-1}(y)$,
\begin{align*}
\zeta_i(x,y,z) = \zeta_i(x,y,\ast).
\end{align*}
\item[\rm (ii)] For all $z \in \Gamma(x) \cap \Gamma_{i+1}(y)$,
\begin{align*}
\zeta_{i+1}(z,y,x) = \zeta_{i+1}(\ast,y,x).
\end{align*}
\end{enumerate}
\end{lemma} 
\begin{proof} (i) Take the inner product of $E{\hat z}$ with each side of  \eqref{eq:lambdaEQ}. Evaluate the result using
Lemmas  \ref{lem:geometry},  \ref{lem:innerPrz} and Definition \ref{def:ziMinus}.
 \\
\noindent (ii) Take the inner product of $E{\hat z}$ with each side of  \eqref{eq:lambdaEQ}. Evaluate the result using
Lemmas  \ref{lem:geometry}, \ref{lem:PinnerPr} and Definition \ref{def:ziPlus}.
\end{proof}

\noindent By a {\it $\mu$-graph} for $\Gamma$ we mean the subgraph induced on $\Gamma(u) \cap \Gamma(v)$, where $u,v\in X$
are at distance $\partial(u,v)=2$.
\begin{corollary} \label{cor:WhenZ5}
Assume that $\gamma^*\not=0$ and the set $\lbrace E{\hat x} \vert x \in X\rbrace$ is Norton-balanced. Then every $\mu$-graph of $\Gamma$ is regular.
\end{corollary}
\begin{proof} Set $i=2$ in Lemma \ref{lem:WhenZ4}(i), and interpret the result using Lemma \ref{lem:ziz}.
\end{proof}

\section{When $\Gamma$ is reinforced}

 We continue to discuss the $Q$-polynomial distance-regular graph  $\Gamma=(X,\mathcal R)$ with diameter $D\geq 3$.
 Let $E$ denote a $Q$-polynomial primitive idempotent of $\Gamma$. 
 Throughout this section, we assume that $\Gamma$ is reinforced in the sense of Definition \ref{def:Reinforce}.
  Under this assumption, we will describe the main results of the
  previous three sections. We will treat separately the cases $\gamma^*=0$ and $\gamma^*\not=0$.
 \medskip
 
 \noindent Throughout this section,  fix $2 \leq i \leq D-1$ and $x,y \in X$ at distance $\partial(x,y)=i$.
 Recall the scalars $r_i, s_i$ from \eqref{eq:risi}. By Lemma \ref{lem:rs} the vector $Ex^-_y-r_i E{\hat x} - s_i E{\hat y}$ is orthogonal to each of $E{\hat x}, E{\hat y}$. 
 Recall the scalars $R_i, S_i$ from \eqref{eq:Prisi}. By Lemma \ref{lem:Prs} the vector $Ex^+_y-R_i E{\hat x} - S_i E{\hat y}$ is orthogonal to each of $E{\hat x}, E{\hat y}$. 
 \medskip
 
 \noindent First assume that $\gamma^*=0$. The following vectors are orthogonal:
 \begin{align*}
 Ex^-_y - r_i E{\hat x} - s_i E{\hat y}, \qquad \qquad 
 Ex^+_y - R_i E{\hat x} - S_i E{\hat y}.
 \end{align*}
 We have
 \begin{align*}
 \frac{z_i - z^-_i}{\theta^*_1- \theta^*_2} \geq 0,
 \end{align*}
 with equality iff $Ex^-_y = r_i E{\hat x} + s_i E{\hat y}$ iff $Ex^-_y$, $E{\hat x}$, $E{\hat y}$ are linearly dependent.
 We also have
 \begin{align*} 
 \frac{z^+_{i+1} - z_{i+1}}{\theta^*_1- \theta^*_2} \geq 0,
 \end{align*}
 with equality iff $Ex^+_y = R_i E{\hat x} + S_i E{\hat y}$ iff $Ex^+_y$, $E{\hat x}$, $E{\hat y}$ are linearly dependent.
 We have 
 \begin{align}  \label{eq:GZ}
 (z_i - z^-_i)(z^+_{i+1} - z_{i+1}) \geq 0,
 \end{align}
 with equality iff  $Ex^-_y$, $Ex^+_y$, $E{\hat x}$, $E{\hat y}$ are linearly dependent iff the following are linearly dependent:
 \begin{align} \label{eq:onceAgain}
 Ex^-_y-r_i E{\hat x} - s_i E{\hat y}, \qquad \qquad Ex^+_y-R_i E{\hat x} - S_i E{\hat y}.
 \end{align}
 In this case, at least one of \eqref{eq:onceAgain} is zero and
 \begin{align*}
 Ex^-_y, Ex^+_y \in {\rm Span} \lbrace E{\hat x}, E{\hat y}, E{\hat x} \star E{\hat y} \rbrace.
 \end{align*}
 
 \noindent For the rest of this section, assume that $\gamma^*\not=0$. We have
 \begin{align*}
 \frac{z_i - z^-_i}{\theta^*_1- \theta^*_2} >0, \qquad \qquad
 \frac{z^+_{i+1} - z_{i+1}}{\theta^*_1- \theta^*_2} >0.
 \end{align*}
\noindent The vectors $Ex^-_y$, $E{\hat x}$, $E{\hat y}$ are linearly independent and the vectors $Ex^+_y$, $E{\hat x}$, $E{\hat y}$ are linearly independent.
We have
 \begin{align} \label{eq:split}
 &\bigl(z_i - z^-_i \bigr)\bigl(z^+_{i+1} -  z_{i+1}\bigr) \geq  c_i b_i   \biggl(\frac{\gamma^*}{ \theta^*_1-\theta^*_2} \, \frac{\theta^*_i - \theta^*_1}{\theta^*_i+\theta^*_0}  \biggr)^2,
 \end{align}
with equality iff $Ex^-_y$, $Ex^+_y$, $E{\hat x}$, $E{\hat y}$ are linearly dependent iff the following are linearly dependent:
\begin{align*}
 Ex^-_y - r_i E{\hat x} - s_i E{\hat y}, \qquad \qquad Ex^+_y - R_i E{\hat x} - S_i E{\hat y}.
\end{align*}
In this case 
\begin{align}  \label{eq:lamiR}
 Ex^-_y-r_i E{\hat x} - s_i E{\hat y} = \lambda_i \bigl( Ex^+_y-R_i E{\hat x} - S_i E{\hat y} \bigr),
\end{align}
\noindent where
\begin{align} \label{lem:lambdaCALC}
\frac{\lambda_i}{\theta^*_1-\theta^*_2}  &= \frac{z_i - z^-_i }{\theta^*_i - \theta^*_1} \, \frac{\theta^*_i + \theta^*_0}{ \gamma^* b_i} , \qquad \quad
\frac{ \lambda^{-1}_i }{\theta^*_1 - \theta^*_2} = \frac{ z^+_{i+1}-z_{i+1} } {\theta^*_i - \theta^*_1} \,\frac{\theta^*_i + \theta^*_0}{\gamma^* c_i}.
 \end{align}
 Also in this case, we have
 \begin{align*}
 Ex^-_y, Ex^+_y \in {\rm Span} \lbrace E{\hat x}, E{\hat y}, E{\hat x} \star E{\hat y} \rbrace,
 \end{align*}
 provided that
 \begin{align*}
 \lambda_i \not=\frac{\theta^*_i-\theta^*_{i+1}}{\theta^*_{i-1}-\theta^*_i}.
 \end{align*}



\section{The polynomials $\Phi_i(\lambda)$} 

 We continue to discuss the $Q$-polynomial distance-regular graph  $\Gamma=(X,\mathcal R)$ with diameter $D\geq 3$.
 Let $E$ denote a $Q$-polynomial primitive idempotent of $\Gamma$. In this section, we introduce some polynomials $\Phi_i(\lambda)$ and use
 them to determine when the set $\lbrace E{\hat x} \vert x \in X\rbrace$ is Norton-balanced.
 Recall the scalars $\alpha_i$, $\beta_i$ from Lemma \ref{lem:kite}.

\begin{definition} \label{def:Phi} \rm Let $\lambda$ denote an indeterminate. For $2 \leq i \leq D-1$ define a polynomial
\begin{align*}
\Phi_i (\lambda) = u_i \lambda^2 + v_i \lambda  + w_i,
\end{align*}
where
\begin{align*}
u_i &= - \alpha_i \alpha_{i+1}, \\
v_i &= \alpha_i \bigl(z^+_{i+1} - a_1 \beta_{i+1}\bigr) -\alpha_{i+1} \bigl(a_1 \beta_i - z^-_i \bigr),        \\
w_i &=\bigl(a_1 \beta_i - z^-_i \bigr)\bigl(z^+_{i+1} -  a_1 \beta_{i+1} \bigr) - c_i b_i   \biggl(\frac{\gamma^*}{ \theta^*_1-\theta^*_2} \, \frac{\theta^*_i - \theta^*_1}{\theta^*_i+\theta^*_0}  \biggr)^2.
\end{align*}
\end{definition}
\noindent The next result indicates why the polynomials $\Phi_i(\lambda)$ are of interest.

\begin{lemma} \label{lem:PhiM} For $2 \leq i \leq D-1$,
\begin{align} \label{eq:Phiz2}
\Phi_i(z_2) =  \bigl(z_i - z^-_i \bigr)\bigl(z^+_{i+1} -  z_{i+1}\bigr) -  c_i b_i   \biggl(\frac{\gamma^*}{ \theta^*_1-\theta^*_2} \, \frac{\theta^*_i - \theta^*_1}{\theta^*_i+\theta^*_0}  \biggr)^2.
\end{align}
\end{lemma}
\begin{proof} By Lemma \ref{lem:kite} we have
\begin{align*}
z_i = z_2 \alpha_i + a_1 \beta_i, \qquad \qquad z_{i+1} = z_2 \alpha_{i+1} + a_1 \beta_{i+1}.
\end{align*}
Using these equations we eliminate $z_i$, $z_{i+1}$ from the right-hand side of \eqref{eq:Phiz2},  and evaluate the result using
Definition \ref{def:Phi}.
\end{proof}

\begin{proposition} \label{prop:assumeC} Assume that $\Gamma$ is reinforced. 
Then for $2 \leq i \leq D-1$,
\begin{align}
\Phi_i(z_2) \geq 0.
\label{eq:PhiZ}
\end{align}
\noindent Moreover, the following are equivalent:
\begin{enumerate}
\item[\rm (i)]  equality holds in \eqref{eq:PhiZ};
\item[\rm (ii)] for all $x, y \in X$ at distance $\partial(x,y)=i$, the vectors  $Ex^-_y$, $Ex^+_y$, $E{\hat x}$, $E{\hat y}$ are linearly dependent;
\item[\rm (iii)] there exists  $x, y \in X$ at distance $\partial(x,y)=i$ such that   $Ex^-_y$, $Ex^+_y$, $E{\hat x}$, $E{\hat y}$ are linearly dependent.
\end{enumerate}
\end{proposition}
\begin{proof} First assume that $\gamma^*=0$.
The inequality \eqref{eq:PhiZ} follows from  \eqref{eq:GZ} and Lemma  \ref{lem:PhiM}.
The equivalence of (i)--(iii) follows from the discussion below \eqref{eq:GZ}.
Next assume that  $\gamma^* \not=0$.
The inequality \eqref{eq:PhiZ} follows from \eqref{eq:split} and Lemma  \ref{lem:PhiM}.
The equivalence of (i)--(iii) follows from the discussion below \eqref{eq:split}.
\end{proof}

\begin{corollary}
\label{cor:NBZ}
 Assume that $\Gamma$ is reinforced. Then the following are equivalent:
 \begin{enumerate}
 \item[\rm (i)] for all $x,y \in X$ 
the vectors $Ex^-_y$, $Ex^+_y$, $E{\hat x}$, $E{\hat y}$ are linearly dependent;
 \item[\rm (ii)]  $\Phi_i(z_2)=0$ for $2 \leq i \leq D-1$.
 \end{enumerate}
 \end{corollary}
\begin{proof}  By Proposition  \ref{prop:assumeC}(i),(ii) and since  $Ex^-_y$, $Ex^+_y$, $E{\hat x}$, $E{\hat y}$ are linearly dependent for all $x,y \in X$ with $\partial(x,y) \in \lbrace 0,1,D\rbrace$.
\end{proof}

\noindent Next, we describe the Norton-balanced condition in terms of the polynomials $\Phi_i(\lambda)$, under the assumption that
$\Gamma$ is reinforced. We will treat separately the cases $\gamma^*=0$ and $\gamma^*\not=0$.
\begin{proposition}\label{cor:NBPhi}
Assume that $\gamma^*=0$ and $\Gamma$ is reinforced. Then the following are equivalent:
\begin{enumerate}
\item[\rm (i)]   the set $\lbrace E{\hat x} \vert x \in X\rbrace$
is Norton-balanced;
\item[\rm (ii)] 
$\Phi_i(z_2)=0$ for $2 \leq i \leq D-1$.
\end{enumerate}
\end{proposition}
\begin{proof} ${\rm (i)} \Rightarrow {\rm (ii)}$
By Lemma  \ref{lem:NBLD} and
Corollary \ref{cor:NBZ}. \\
${\rm (ii)} \Rightarrow {\rm (i)}$ By Corollary
\ref{cor:NBZ} and the comment below
\eqref{eq:onceAgain}.
\end{proof}

\begin{proposition}\label{cor:NBPhi2}
Assume that $\gamma^*\not=0$ and $\Gamma$ is reinforced. Then the following are
equivalent:
\begin{enumerate}
\item[\rm (i)]  the set $\lbrace E{\hat x} \vert x \in X\rbrace$
is Norton-balanced;
\item[\rm (ii)]  for $2 \leq i \leq D-1$ both
\begin{align} 
\Phi_i(z_2)=0, \qquad \qquad \lambda_i \not= \frac{\theta^*_i-\theta^*_{i+1}}{\theta^*_{i-1}-\theta^*_i},         \label{eq:2conditions}
\end{align}
 where $\lambda_i$ is from   \eqref{eq:lamiR}.
\end{enumerate}
\end{proposition}
\begin{proof}  ${\rm (i)} \Rightarrow {\rm (ii)}$ 
We have $\Phi_i(z_2)=0$ by Lemma  \ref{lem:NBLD} and
Corollary \ref{cor:NBZ}. To verify the inequality on the right in \eqref{eq:2conditions},
we assume that $\lambda_i = (\theta^*_i - \theta^*_{i+1})/(\theta^*_{i-1}-\theta^*_i)$ and get a contradiction.
Pick $x,y \in X$ at distance $\partial(x,y)=i$. Combining   \eqref{eq:main}, \eqref{eq:lamiR} we find
$E {\hat x} \star E{\hat y} \in {\rm Span} \lbrace E{\hat x}, E{\hat y}\rbrace$. By Lemma \ref{lem:clarify},
\begin{align*}
Ex^-_y, Ex^+_y \in {\rm Span} \lbrace E{\hat x}, E{\hat y}, E{\hat x} \star E{\hat y}\rbrace = {\rm Span} \lbrace E{\hat x}, E{\hat y}\rbrace.
\end{align*}
This contradicts Corollary \ref{cor:gamNZ} and Corollary \ref{cor:PgamNZ}. Therefore
$\lambda_i \not= (\theta^*_i - \theta^*_{i+1})/(\theta^*_{i-1}-\theta^*_i)$.
\\
\noindent   ${\rm (ii)} \Rightarrow {\rm (i)}$ 
By Lemma  \ref{lem:clarify}, Corollary
\ref{cor:NBZ}, and the comment below \eqref{lem:lambdaCALC}.
\end{proof}

\noindent Motivated by Propositions   \ref{cor:NBPhi} and \ref{cor:NBPhi2}, we next consider how the polynomial $\Phi_i(\lambda)$ depends on $i$ for $2 \leq i \leq D-1$.

\begin{lemma} \label{lem:uform} If $\beta \not=-2$  then $u_i \not=0$ for $2 \leq i \leq D-1$. If $\beta=-2$  then $u_i = 0 $ for $2 \leq i \leq D-1$.
\end{lemma}
\begin{proof} By Lemma \ref{lem:aaform} and Definition \ref{def:Phi}, along with the fact that  $\lbrace \theta^*_j \rbrace_{j=0}^D$  are mutually distinct.
\end{proof}

\begin{lemma} \label{lem:det3} The rank of the following matrix is at most 2:
\begin{align*}
\begin{pmatrix} u_2 & u_3 & u_4 & \cdots & u_{D-1} \\
                                        v_2 & v_3 & v_4 & \cdots & v_{D-1}  \\
                                        w_2 & w_3 & w_4 & \cdots & w_{D-1} 
                                        \end{pmatrix}.
\end{align*}
\end{lemma}
\begin{proof}  It suffices to show that for  $2 \leq h<i<j\leq D-1$, 
\begin{align} \label{eq:uvw}
{\rm det} \begin{pmatrix}  u_h & u_i & u_j  \\
        v_h & v_i & v_j \\
        w_h  & w_i  & w_j
        \end{pmatrix} = 0.
\end{align}
First assume that $\beta =-2$. Then  \eqref{eq:uvw} holds since the top row is zero by Lemma \ref{lem:uform}. Next assume that $\beta \not=-2$.
To verify \eqref{eq:uvw} in this case, we refer to the table 
 above Remark \ref{rem:LP}. We verify \eqref{eq:uvw}  for each subcase such that $\beta \not=- 2$.
For each of these subcases, the verification of \eqref{eq:uvw} is done by 
evaluating the matrix entries using Definition \ref{def:Phi} and the data in \cite[Section~20]{LSnotes}. 
\end{proof} 

\noindent For $2 \leq i \leq D-1$, by a {\it root} of $\Phi_i(\lambda)$ we mean a scalar $\xi \in \mathbb C$ such that $\Phi_i(\xi)=0$.
As we investigate these roots, we will treat separately the cases $\beta \not=-2$ and $\beta=-2$.

\begin{lemma}\label{lem:cr} Assume that $\beta \not=-2$. Then for $2 \leq i,j\leq D-1$ the following hold.
\begin{enumerate}
\item[\rm (i)] Assume that $\Phi_i(\lambda)$, $\Phi_j(\lambda)$ have no roots in common. Then
\begin{align*}
(u_i w_j - u_j w_i)^2 \not= (v_i w_j -v_j w_i)(u_i v_j-u_j v_i).
\end{align*}
\item[\rm (ii)] Assume that  $\Phi_i(\lambda)$, $\Phi_j(\lambda)$  have a root in common, and  $\Phi_i(\lambda)$, $\Phi_j(\lambda)$ are linearly independent. Then
\begin{align*}
(u_i w_j - u_j w_i)^2 = (v_i w_j -v_j w_i)(u_i v_j-u_j v_i)
\end{align*}
and  $u_i v_j - u_j v_i \not=0$. The common root is
\begin{align*}
     \frac{w_i u_j - w_j u_i}{u_i v_j - u_j v_i}.
     \end{align*}
     \item[\rm (iii)] Assume that $\Phi_i(\lambda)$, $\Phi_j(\lambda)$ are linearly dependent. Then both
     \begin{align*}
     u_i v_j - u_j v_i = 0, \qquad \qquad u_i w_j - u_j w_i = 0.
     \end{align*}
\end{enumerate}
\end{lemma}
\begin{proof} Write 
\begin{align*}
\Phi_i(\lambda) = u_i (\lambda - r)(\lambda-s), \qquad \qquad \Phi_j(\lambda) = u_j(\lambda-R)(\lambda-S).
\end{align*}
We have
\begin{align*}
v_i = -u_i (r+s), \qquad w_i = u_i r s, \qquad  v_j = -u_j (R+S), \qquad w_j = u_j RS.
\end{align*}
We obtain
\begin{align*}
(u_i w_j - u_j w_i)^2 - (v_i w_j -v_j w_i)(u_i v_j-u_j v_i) = u^2_i u^2_j (r-R)(r-S)(s-R)(s-S)
\end{align*}
and
\begin{align*}
u_i v_j - u_j v_i = u_i u_j (r+s-R-S), \qquad \qquad w_i u_j - w_j u_i  = u_i u_j (rs-RS).
\end{align*}
Using these comments we routinely obtain the result.
\end{proof}

\begin{lemma} \label{lem:cr2} Assume that $\beta \not=-2$. Then for $2 \leq h<i<j\leq D-1$ the following are equivalent:
\begin{enumerate}
\item[\rm (i)] any two of $\Phi_h(\lambda), \Phi_i(\lambda), \Phi_j(\lambda)$ have a root in common;
\item[\rm (ii)] there exists $\xi \in \mathbb C$ such that $\Phi_h (\xi)=\Phi_i(\xi) = \Phi_j(\xi)=0$.
\end{enumerate}
\end{lemma}
\begin{proof} ${\rm (i)} \Rightarrow {\rm (ii)}$ We assume that (ii) is false, and get a contradiction. There exist mutually distinct $r,s,t \in \mathbb C$ such that
\begin{align*}
\Phi_h (\lambda) = u_h (\lambda - r)(\lambda-s), \qquad
\Phi_i (\lambda) = u_i (\lambda - s)(\lambda-t), \qquad
\Phi_j (\lambda) = u_j (\lambda - t)(\lambda-r).
\end{align*}
Using these forms we obtain
\begin{align*}
{\rm det} \begin{pmatrix}  u_h & u_i & u_j  \\
        v_h & v_i & v_j \\
        w_h  & w_i  & w_j
        \end{pmatrix} = u_h u_i u_j (r-s)(s-t)(t-r) \not=0.
\end{align*}
This contradicts Lemma \ref{lem:det3}.\\
\noindent  ${\rm (ii)} \Rightarrow {\rm (i)}$ Clear.
\end{proof}

\begin{proposition} \label{prop:uwuv} Assume that $\beta \not=-2$. Then the following are equivalent:
\begin{enumerate}
\item[\rm (i)] there exists $\xi \in \mathbb C$ such that $\Phi_i (\xi)=0$ for $2 \leq i \leq D-1$;
\item[\rm (ii)]  for $2 \leq i, j\leq D-1$,
\begin{align*}
(u_i w_j - u_j w_i)^2 = (v_i w_j -v_j w_i)(u_i v_j-u_j v_i).
\end{align*}
\end{enumerate}
\end{proposition}
\begin{proof} ${\rm (i)} \Rightarrow {\rm (ii)}$ By  Lemma \ref{lem:cr}.   \\   
\noindent ${\rm (ii)} \Rightarrow {\rm (i)}$ By Lemmas \ref{lem:cr}, \ref{lem:cr2}.
\end{proof}

\noindent Next, we examine condition (ii) of Proposition \ref{prop:uwuv}. Under the assumption that $\beta \not=-2$,  we compute the scalars
\begin{align*}
(u_i w_j - u_j w_i)^2 - (v_i w_j -v_j w_i)(u_i v_j-u_j v_i) \qquad \quad (2 \leq i,j\leq D-1).
\end{align*}
Recall the subcases listed in the table above Remark \ref{rem:LP}.  For each subcase such that $\beta \not=-2$, we will do the above computation using the data in
\cite[Section~20]{LSnotes}.

\begin{proposition} \label{prop:qrac} Assume the given $Q$-polynomial structure has $q$-Racah type. Then for $2 \leq i,j\leq D-1$ the scalar
\begin{align*}
(u_i w_j - u_j w_i)^2 - (v_i w_j -v_j w_i)(u_i v_j-u_j v_i)
\end{align*}
is equal to
\begin{align*}
a^*_1 (r^2_1 - s)(r^2_2-s)(r^2_3-s)
\end{align*}
times
\begin{align*}
s+ s^* -q^{-1} r_1 -q^{-1} r_2 +r_3 + r_1 r_2 - q r_2 r_3 - q r_3 r_1
  \end{align*}
  times
\begin{align*}
s+s^* -q^{-1} r_2 - q^{-1} r_3 +r_1 + r_2 r_3 - q r_3 r_1 - q r_1 r_2
 \end{align*}
 times
 \begin{align*}
 s+s^* - q^{-1} r_3 -q^{-1} r_1 + r_2 + r_3 r_1 - q r_1 r_2 - q r_2 r_3
 \end{align*}
times 
\begin{align*}
\frac{ u^2_i u^2_j (\theta^*_i - \theta^*_j)^2 h^4 h^* }{(\theta^*_i + \theta^*_0)^2(\theta^*_j + \theta^*_0)^2}\,
 \frac{1-q^4 s}{(1-q^2 s)^3} \; \frac{(1-q^3 s^*)^4}{(1-q^4 s^*)^8} \; \frac{q^{10} (q-1)^7}{s} ,
\end{align*}
 where $r_3 = q^{-D-1}$. 
 \end{proposition}
 \begin{proof} Use the data in \cite[Example~20.1]{LSnotes}.
 \end{proof}
\begin{remark}\rm Referring to Proposition \ref{prop:qrac},
 \begin{align*}
 s+ s^*-q^{-1} r_1 -q^{-1} r_2 +r_3 + r_1 r_2 - q r_2 r_3 - q r_3 r_1 
 &  = \frac{ a^*_D (\theta_0 - \theta_1)(\theta_{D-1}-\theta_D)}{ h  h^* (q-1)^2 (\theta_0 - \theta_D)}.
 \end{align*}
 \end{remark}
 

\begin{proposition} \label{prop:qhahn} Assume the given $Q$-polynomial structure has $q$-Hahn type. Then for $2 \leq i,j\leq D-1$ the scalar
\begin{align*}
(u_i w_j - u_j w_i)^2 - (v_i w_j -v_j w_i)(u_i v_j-u_j v_i)
\end{align*}
is equal to
\begin{align*}
-a^*_1 r^2 r^2_3
( s^* -q^{-1} r +r_3  - q r r_3 )
(s^* - q^{-1} r_3 +r  - q r  r_3 )
 (s^* - q^{-1} r_3 -q^{-1} r + r r_3)
 \end{align*}
times 
\begin{align*}
\frac{ u^2_i u^2_j (\theta^*_i - \theta^*_j)^2 h^4 h^* }{(\theta^*_i + \theta^*_0)^2(\theta^*_j + \theta^*_0)^2} \,
  \frac{(1-q^3 s^*)^4 q^{10} (q-1)^7}{(1-q^4 s^*)^8},
\end{align*}
where $r_3 = q^{-D-1}$. 
 \end{proposition}
  \begin{proof} Use the data in \cite[Example~20.2]{LSnotes}.
 \end{proof}
\begin{remark}\rm With reference to Proposition \ref{prop:qhahn},
 \begin{align*}
 s^*-q^{-1} r +r_3  - q r r_3 
 &  = \frac{ a^*_D (\theta_0 - \theta_1)(\theta_{D-1}-\theta_D)}{ h  h^* (q-1)^2 (\theta_0 - \theta_D)}.
 \end{align*}
 \end{remark}


\begin{proposition} \label{prop:dqhahn} Assume the given $Q$-polynomial structure has dual $q$-Hahn type. Then for $2 \leq i,j\leq D-1$ the scalar
\begin{align*}
(u_i w_j - u_j w_i)^2 - (v_i w_j -v_j w_i)(u_i v_j-u_j v_i)
\end{align*}
is equal to
\begin{align*}
-a^*_1 (r^2 - s)(r^2_3-s)
(s -q^{-1} r  +r_3  - q r r_3 )
(s - q^{-1} r_3 +r - q r r_3)
 (s - q^{-1} r_3 -q^{-1} r   + r r_3)
 \end{align*}
times 
\begin{align*}
\frac{ u^2_i u^2_j (\theta^*_i - \theta^*_j)^2 h^4 h^* }{(\theta^*_i + \theta^*_0)^2(\theta^*_j + \theta^*_0)^2}\,
 \frac{(1-q^4 s)q^{10} (q-1)^7}{(1-q^2 s)^3},
\end{align*}
 where $r_3 = q^{-D-1}$. 
 \end{proposition}
  \begin{proof} Use the data in \cite[Example~20.3]{LSnotes}.
 \end{proof}
\begin{remark} \rm With reference to Proposition \ref{prop:dqhahn},
 \begin{align*}
s -q^{-1} r  +r_3  - q r r_3 
  = \frac{ a^*_D (\theta_0 - \theta_1)(\theta_{D-1}-\theta_D)}{ h  h^* (q-1)^2 (\theta_0 - \theta_D)}.
 \end{align*}
 \end{remark}

\begin{proposition} \label{prop:qkraw} Assume the given $Q$-polynomial structure has $q$-Krawtchouk type. Then for $2 \leq i,j\leq D-1$ the scalar
\begin{align*}
(u_i w_j - u_j w_i)^2 - (v_i w_j -v_j w_i)(u_i v_j-u_j v_i)
\end{align*}
is equal to $0$.
\end{proposition}
 \begin{proof} Use the data in \cite[Example~20.5]{LSnotes}.
 \end{proof}

\begin{proposition} \label{prop:aqkraw} Assume the given $Q$-polynomial structure has affine $q$-Krawtchouk type. Then for $2 \leq i,j\leq D-1$ the scalar
\begin{align*}
(u_i w_j - u_j w_i)^2 - (v_i w_j -v_j w_i)(u_i v_j-u_j v_i)
\end{align*}
is equal to
\begin{align*}
-a^*_1 r^2 r^2_3
( -q^{-1} r  +r_3  - q r r_3)
 (- q^{-1} r_3 +r - q r r_3)
 (- q^{-1} r_3 -q^{-1} r   + r r_3)
 \end{align*}
times 
\begin{align*}
\frac{ u^2_i u^2_j (\theta^*_i - \theta^*_j)^2 h^4 h^* }{(\theta^*_i + \theta^*_0)^2(\theta^*_j + \theta^*_0)^2}\,
q^{10} (q-1)^7,
\end{align*}
 where $r_3 = q^{-D-1}$. 
 \end{proposition}
  \begin{proof} Use the data in \cite[Example~20.6]{LSnotes}.
 \end{proof}
 \begin{remark}\rm Referring to Proposition 
 \ref{prop:aqkraw},
 \begin{align*}
 -q^{-1} r  +r_3  - q r r_3 
  = \frac{ a^*_D (\theta_0 - \theta_1)(\theta_{D-1}-\theta_D)}{ h  h^* (q-1)^2 (\theta_0 - \theta_D)}.
 \end{align*}
 \end{remark}

\begin{proposition} \label{prop:dqkraw} Assume the given $Q$-polynomial structure has dual $q$-Krawtchouk type. Then for $2 \leq i,j\leq D-1$ the scalar
\begin{align*}
(u_i w_j - u_j w_i)^2 - (v_i w_j -v_j w_i)(u_i v_j-u_j v_i)
\end{align*}
is equal to
\begin{align*}
a^*_1 s^2 (r^2_3-s)
(s +r_3)
(s - q^{-1} r_3 )
( s - q^{-1} r_3 )
 \end{align*}
times 
\begin{align*}
\frac{ u^2_i u^2_j (\theta^*_i - \theta^*_j)^2 h^4 h^* }{(\theta^*_i + \theta^*_0)^2(\theta^*_j + \theta^*_0)^2}\,
 \frac{1-q^4 s}{(1-q^2 s)^3} \; \frac{q^{10} (q-1)^7}{s} ,
\end{align*}
 where $r_3 = q^{-D-1}$. 
 \end{proposition}
  \begin{proof} Use the data in \cite[Example~20.7]{LSnotes}.
 \end{proof}
\begin{remark}\rm Referring to Proposition \ref{prop:dqkraw},
 \begin{align*}
 s +r_3
 &  = \frac{ a^*_D (\theta_0 - \theta_1)(\theta_{D-1}-\theta_D)}{ h  h^* (q-1)^2 (\theta_0 - \theta_D)}.
 \end{align*}
 \end{remark}
 
\begin{proposition} \label{prop:racah} Assume the given $Q$-polynomial structure has Racah type. Then for $2 \leq i,j\leq D-1$ the scalar
\begin{align*}
(u_i w_j - u_j w_i)^2 - (v_i w_j -v_j w_i)(u_i v_j-u_j v_i)
\end{align*}
is equal to
\begin{align*}
a^*_1 (2r_1-s)(2r_2-s)(2r_3-s) 
\end{align*}
times
\begin{align*}
(2 r_1 r_2 - 2 r_3 - 2 - s s^*)
(2 r_2 r_3 -2 r_1 -2 - s s^*)
(2 r_3 r_1 - 2 r_2 - 2 -  s s^*)
 \end{align*}
times 
\begin{align*}
\frac{ u^2_i u^2_j (\theta^*_i - \theta^*_j)^2 h^4 h^* }{(\theta^*_i + \theta^*_0)^2(\theta^*_j + \theta^*_0)^2}\,
 \frac{ s+4}{(s+2)^3} \; \frac{( s^*+3)^4}{(s^*+4)^8},
\end{align*}
where $r_3=-D-1$.
 \end{proposition}
  \begin{proof} Use the data in \cite[Example~20.8]{LSnotes}.
 \end{proof}
 \begin{remark}\rm Referring to Proposition \ref{prop:racah},
 \begin{align*}
 2r_1 r_2 - 2 r_3 -2 - s s^*
 &  = \frac{ a^*_D (\theta_0 - \theta_1)(\theta_{D-1}-\theta_D)}{ h  h^* (\theta_0 - \theta_D)}.
 \end{align*}
\end{remark}

\begin{proposition} \label{prop:hahn} Assume the given $Q$-polynomial structure has Hahn type. Then for $2 \leq i,j\leq D-1$ the scalar
\begin{align*}
(u_i w_j - u_j w_i)^2 - (v_i w_j -v_j w_i)(u_i v_j-u_j v_i)
\end{align*}
is equal to
\begin{align*}
- a^*_1 
(2 r-s^*)
(2 r_3-s^*)
 \end{align*}
times 
\begin{align*}
\frac{ u^2_i u^2_j (\theta^*_i - \theta^*_j)^2 s^4 h^* }{(\theta^*_i + \theta^*_0)^2(\theta^*_j + \theta^*_0)^2}\,
 \frac{(s^*+2) ( s^*+3)^4}{(s^*+4)^8},
\end{align*}
where $r_3=-D-1$.
 \end{proposition}
  \begin{proof} Use the data in \cite[Example~20.9]{LSnotes}.
 \end{proof}
  \begin{remark}\rm Referring to Proposition \ref{prop:hahn},
 \begin{align*}
 2r-s^*
 &  = \frac{ a^*_D (\theta_0 - \theta_1)(\theta_{D-1}-\theta_D)}{ s  h^* (\theta_0 - \theta_D)}.
 \end{align*}
 \end{remark}

\begin{proposition} \label{prop:dhahn} Assume the given $Q$-polynomial structure has dual Hahn type. Then for $2 \leq i,j\leq D-1$ the scalar
\begin{align*}
(u_i w_j - u_j w_i)^2 - (v_i w_j -v_j w_i)(u_i v_j-u_j v_i)
\end{align*}
is equal to zero.
 \end{proposition}
  \begin{proof} Use the data in \cite[Example~20.10]{LSnotes}.
 \end{proof}
 
\begin{proposition} \label{prop:kraw} Assume the given $Q$-polynomial structure has Krawtchouk  type. Then for $2 \leq i,j\leq D-1$ the scalar
\begin{align*}
(u_i w_j - u_j w_i)^2 - (v_i w_j -v_j w_i)(u_i v_j-u_j v_i)
\end{align*}
is equal to zero.
 \end{proposition}
  \begin{proof} Use the data in \cite[Example~20.11]{LSnotes}.
 \end{proof}

\noindent  We have been discussing the case $\beta \not=-2$. Next, we discuss the case $\beta=-2$. This case is 
called  Bannai/Ito type.
\medskip

\noindent Assume that $\beta=-2$. Pick $2 \leq i \leq D-1$ and consider the polynomial $\Phi_i(\lambda)$. We have $u_i=0$ by
Lemma \ref{lem:uform},
 so $\Phi_i(\lambda) = v_i \lambda + w_i$. We will show that $w_i=0$.
 
\begin{proposition} \label{prop:bm2} Assume that $\beta=-2$. Then  $w_i=0$ and $\Phi_i(0)=0$ for $2 \leq i \leq D-1$.
\end{proposition}
\begin{proof} We invoke the classification   \cite[Theorem~2]{type3}. There are three solutions for $\Gamma$: the Odd graph $O_{D+1}$; the Hamming graph $H(D,2)$ with $D$ even; and the folded cube ${\tilde H}(2D+1,2)$.
The graphs $O_{D+1}$ and ${\tilde H}(2D+1,2)$ are almost bipartite, and $H(D,2)$ is bipartite. For each solution $\Gamma$ the set $\lbrace E{\hat x} \vert x \in X\rbrace$
is Norton-balanced by Proposition \ref{prop:BAB}. Each solution $\Gamma$ is distance-transitive, and hence
reinforced by Lemma \ref{lem:WhenReinforce}. By these comments and  Propositions \ref{cor:NBPhi}, \ref{cor:NBPhi2}
we have $\Phi_i(z_2)=0$ for $2 \leq i \leq D-1$. Each solution $\Gamma$ has $a_1=0$, so $\Gamma$ is kite-free. Consequently $z_2=0$,
so  $\Phi_i(0)=0$ for $2 \leq i \leq D-1$. Observe that $w_i = \Phi_i(0)=0$ for $2 \leq i \leq D-1$.
\end{proof} 
 
 \noindent A detailed discussion of $O_{D+1}$, $H(D,2)$, ${\tilde H}(2D+1,2)$ can be found in Sections 18, 23, 25 respectively.

  \section{The DC condition} 

 We continue to discuss the $Q$-polynomial distance-regular graph  $\Gamma=(X,\mathcal R)$ with diameter $D\geq 3$.
 Let $E$ denote a $Q$-polynomial primitive idempotent of $\Gamma$. 
  
  \begin{definition} \label{def:dc} \rm We say that $E$ is a {\it dependency candidate} (or {\it $DC$}) whenever there exists  $\xi \in \mathbb C$ such that
  $\Phi_i(\xi)=0$ for $2 \leq i \leq D-1$.
  \end{definition}
  
  \begin{remark}\rm Referring to Definition \ref{def:dc}, $E$ being DC  is a condition on the intersection numbers of $\Gamma$.
  \end{remark}
  
\begin{lemma} Assume that $\Gamma$ is reinforced.  Assume that for all $x,y \in X$ the vectors
$Ex^-_y$, $Ex^+_y$, $E{\hat x}$, $E{\hat y}$ are linearly dependent. Then
$E$ is DC.
\end{lemma}
\begin{proof} By Corollary \ref{cor:NBZ}  and Definition \ref{def:dc}.
\end{proof}

\begin{lemma} \label{lem:RNBDC} Assume that $\Gamma$ is reinforced.  Assume that the set $\lbrace E{\hat x} \vert x \in X\rbrace$
is Norton-balanced. Then $E$ is DC.
\end{lemma}
\begin{proof} By Propositions \ref{cor:NBPhi}, \ref{cor:NBPhi2} and Definition \ref{def:dc}.
\end{proof}

\noindent We now give our main result about DC. In this result we refer to the data in \cite[Section~20]{LSnotes}.

\begin{theorem} \label{thm:main} For $D\geq 4$ the following {\rm (i), (ii)} hold.
\begin{enumerate}
\item[\rm (i)]  Assume that the type of $E$ is included in the table below. Then  $E$ is DC iff at least one of the listed scalars is zero.
\bigskip

\centerline{
\begin{tabular}[t]{c|c}
{\rm type of $E$}& {\rm $E$ is  DC iff at least one of these scalars is zero} 
 \\
 \hline
{\rm  $q$-Racah} & 
$a^*_1,   \quad r^2_1 - s, \quad r^2_2-s, \quad r^2_3-s$, \\
& $s+ s^* -q^{-1} r_1 -q^{-1} r_2 +r_3 + r_1 r_2 - q r_2 r_3 - q r_3 r_1$, \\
& $s+s^* -q^{-1} r_2 - q^{-1} r_3 +r_1 + r_2 r_3 - q r_3 r_1 - q r_1 r_2$, \\
& $ s+s^* - q^{-1} r_3 -q^{-1} r_1 + r_2 + r_3 r_1 - q r_1 r_2 - q r_2 r_3$
\\ 
{\rm  $q$-Hahn} &
$a^*_1, \quad
s^* -q^{-1} r +r_3  - q r r_3$, \\
& $  s^* - q^{-1} r_3 +r  - qr  r_3, \quad
 s^* - q^{-1} r_3 -q^{-1} r + r r_3$
 \\
{\rm  dual $q$-Hahn} & 
$a^*_1, \quad  r^2 - s, \quad r^2_3-s, \quad 
s -q^{-1} r  +r_3  - q r r_3$, \\
&$s - q^{-1} r_3 +r - q r r_3, \quad 
 s - q^{-1} r_3 -q^{-1} r   + r r_3$
\\
{\rm  affine $q$-Krawtchouk} & 
$a^*_1, \quad 
 -q^{-1} r  +r_3  - q r r_3$, \\
& $
 - q^{-1} r_3 +r - q r r_3, \quad 
  - q^{-1} r_3 -q^{-1} r   + r r_3$
\\
{\rm  dual $q$-Krawtchouk} & 
$a^*_1, \quad  r^2_3-s, \quad 
s +r_3, \quad 
 s - q^{-1} r_3$
\\
{\rm  Racah} & 
$a^*_1, \quad  2r_1-s, \quad 2r_2-s, \quad 2r_3-s$, \\
& $ 2 r_1 r_2 - 2 r_3 - 2 - s s^*, \quad 
2 r_2 r_3 -2 r_1 -2 - s s^*$, \\
& $2 r_3 r_1 - 2 r_2 - 2 -  s s^*$
\\
{\rm  Hahn} & 
$a^*_1, \quad  
2 r-s^*, \quad 
2 r_3-s^*$
   \end{tabular} }
   \bigskip
   
\item[\rm (ii)]  Assume that the type of $E$ is $q$-Krawtchouk  or dual Hahn  or Krawtchouk  or Bannai/Ito. Then $E$ is $DC$.
\end{enumerate}
\end{theorem}
\begin{proof} (i) For Propositions \ref{prop:qrac}, \ref{prop:qhahn}, \ref{prop:dqhahn}, \ref{prop:aqkraw}, \ref{prop:dqkraw}, \ref{prop:racah},
\ref{prop:hahn},  examine the factorization in the proposition statement.  For each factorization, consider which factors could be zero.
Some of the factors are nonzero because of the inequalities in \cite[Section~5]{LPpa}. The remaining factors are listed in the above table. 
\\
\noindent (ii) By Propositions \ref{prop:qkraw}, \ref{prop:dhahn}, \ref{prop:kraw}, \ref{prop:bm2}.
\end{proof}

\noindent The book \cite[Chapter~6.4]{bbit} gives a list of the known infinite families of $Q$-polynomial distance-regular graphs with unbounded diameter.
For each listed graph, every $Q$-polynomial structure is described. In Sections 17--29, we will 
examine these $Q$-polynomial structures. For each listed graph $\Gamma=(X,\mathcal R)$ and each $Q$-polynomial primitive idempotent $E$ of $\Gamma$, we will determine
if the set  $\lbrace E{\hat x} \vert x \in X\rbrace$ is Norton-balanced or not. We will also determine if $E$ is $DC$ or not. Considerable supporting data will be given, using the notation of \cite[Section~20]{LSnotes}.
We obtained this supporting data using Sections 11--16; the computations are routine and omitted.
 For the rest of the
paper, the integer $D$ is assumed to be at least 3.

\section{Example: the Johnson graph}
\begin{example}\label{ex:johnson1} \rm (See \cite[Chapter~6.4]{bbit}, \cite[Example~6.1(1)]{tSub3}.)
The {\it Johnson graph} $J(N,D)$ $(N\geq  2D)$ has vertex set $X$ consisting of the subsets of $\lbrace 1,2,\ldots, N\rbrace$ 
that have cardinality $D$.
Vertices $x,y \in X$ are adjacent whenever $\vert x\cap y\vert = D-1$. The graph $J(N,D)$ is distance-regular with diameter $D$ and intersection numbers
\begin{align*}
c_i  = i^2, \qquad \qquad b_i = (D-i)(N-D-i) \qquad \qquad (0 \leq i \leq D).
\end{align*}
The graph $J(2D,D)$ is an antipodal 2-cover.
\end{example}

\begin{example}\label{ex:johnson2}\rm 
The graph $J(N,D)$ 
has a $Q$-polynomial structure such that
\begin{align*}
\theta_i &= (D-i)(N-D-i)-i \qquad \qquad (0 \leq i \leq D), 
\\
\theta^*_i &= N-1-\frac{i N (N-1)}{D(N-D)}\qquad \qquad (0 \leq i \leq D). 
\end{align*}
 This $Q$-polynomial structure has dual Hahn type with 
\begin{align*}
r=D-N-1, \qquad s=-N-2, \qquad h=1, \qquad s^*= \frac{N(1-N)}{D(N-D)}.
\end{align*}
This structure is DC with $\gamma^*=0$.
\medskip

\noindent For $2 \leq i \leq D$,
\begin{align*}
&\alpha_i = i-1, \qquad \beta_i = 0, \\
&r_i = i(i-1), \qquad  s_i = i, \qquad  z^-_i = 2(i-1).
\end{align*}
\noindent For $1 \leq i \leq D-1$,
\begin{align*}
& R_i = \frac{(D-i)(N-D-i)(2D^2-2DN+iN+N)}{iN-2DN+2D^2}, \\
& S_i = \frac{N(D-i)(N-D-i)}{iN-2DN+ 2D^2}, \\
& z^+_{i+1} = \frac{N(N-2i)i}{2DN-2D^2-iN}.
\end{align*}
\noindent For $2 \leq i \leq D-1$,
\begin{align*}
&u_i = -i(i-1),
\qquad v_i = \frac{i(i-1)\bigl(2N(2i-N)+(N-2D)^2\bigr)}{iN-2DN+2D^2},\\
&w_i = \frac{2i(i-1)N(N-2i)}{iN-2DN+2D^2},\\
&\Phi_i(\lambda) = u_i (\lambda-\xi)(\lambda-\xi_i), \qquad \xi=2, \qquad \xi_i = \frac{N(2i-N)}{iN-2DN+2D^2}.
\end{align*}
If $N=2D$ then $\xi_i=2$.
\end{example}

\begin{lemma}  \label{lem:Johnson3} For $J(N,D)$ the kite function $\zeta_i$ is constant for $2 \leq i \leq D$. Moreover  $z_i=2(i-1)$ for $2 \leq i \leq D$. 
\end{lemma}
\begin{proof} By combinatorial counting.
\end{proof}

\begin{lemma}\rm 
We refer to Example \ref{ex:johnson2} and write $E=E_1$.
 Pick distinct  $x,y \in X$ and write $i=\partial(x,y)$. For $2 \leq i \leq D$,
\begin{align} \label{eq:john1}
E x^-_y = i(i-1) E {\hat x} + iE{\hat y}.
\end{align}
For $1 \leq i \leq D-1$ and $N=2D$,
\begin{align}\label{eq:john2} 
E x^+_y = (D-i)(D-i-1) E {\hat x} + (i-D)E{\hat y}.
\end{align}
\noindent In any case, the set $\lbrace E{\hat x} \vert x \in X\rbrace$ is Norton-balanced.
\end{lemma} 
\begin{proof}  First assume that $N>2D$. Then $\Gamma$ is not an antipodal 2-cover.
To get \eqref{eq:john1}, use Proposition \ref{prop:AT} and  $z^-_i = z_i$.
Next assume that $N=2D$. Then $\Gamma$ is an antipodal 2-cover.
To get \eqref{eq:john1} for $1 \leq i \leq D-1$, use Proposition \ref{prop:AT} and  $z^-_i = z_i$.
 To get \eqref{eq:john1} for $i=D$, use $E{\hat x} + E{\hat y}=0$ and $x^-_y=A{\hat x}$ and $\theta_1 = D(D-2)$.
To get \eqref{eq:john2}, use Proposition
 \ref{prop:PAT} and $z^+_{i+1} = z_{i+1}$. 
 Next assume that $N\geq 2D$.
 It follows from Lemma  \ref{lem:clarify2} and  \eqref{eq:john1}  that the set  $\lbrace E{\hat x} \vert x \in X\rbrace$ is Norton-balanced.
\end{proof}

\section{Example: the Odd graph}
\begin{example}\label{ex:odd} \rm (See \cite[Chapter~6.4]{bbit}, \cite[Example~6.1(2)]{tSub3}.)
The {\it Odd graph} $O_{D+1}$  has vertex set $X$ consisting of the $D$-element subsets of the set $\lbrace 1,2,\ldots, 2D+1\rbrace$.
Vertices $x,y \in X$ are adjacent whenever they are disjoint. The graph $O_{D+1}$ is distance-regular with diameter $D$ and intersection numbers
\begin{align*}
c_i & = \frac{2i+1-(-1)^i}{4} \qquad \qquad (1 \leq i \leq D), \\
b_i &= D+\frac{3-2i+(-1)^i}{4} \qquad \qquad (0 \leq i \leq D-1).
\end{align*}
The graph $O_{D+1}$ is almost bipartite.
\end{example}

\begin{example}\label{ex:odd2} \rm 
The graph $O_{D+1}$ 
has a $Q$-polynomial structure such that
\begin{align*}
\theta_i &= (-1)^i(D-i+1) \qquad \qquad (0 \leq i \leq D), 
\\
\theta^*_i &= \frac{(-1)^i(4D^2-4iD+4D-2i+1)-1}{2(D+1)}\qquad \quad (0 \leq i \leq D). 
\end{align*}
 This $Q$-polynomial structure has Bannai/Ito type with 
\begin{align*}
&r_1=-D-1, \qquad \qquad r_2=-2D-3, \qquad \qquad s = 2D+3,\\
&s^*=2D+2, \qquad \qquad h=-1/2, \qquad \qquad h^*= -\frac{2D+1}{2(D+1)}.
\end{align*}
This structure is DC with 
\begin{align*}
\gamma^*= -\frac{2}{D+1} \not=0.
\end{align*}

\noindent For $2 \leq i \leq D$,
\begin{align*}
&\alpha_i = \frac{D-1}{2} \, \frac{(-1)^i+1}{D-i+1}, \qquad \qquad \beta_i =   \frac{3-2i+(-1)^i}{4(D-i+1)}, \\
&r_i = \frac{1}{2}\,\frac{2i+1-(-1)^i}{(-1)^i(2D-2i+1)-2D-1} \\
& \qquad \times 
\frac{8iD^2 - 4 i^2 D-8D^2+12 i D-2i^2 -8D+4i-1+(-1)^i}{(-1)^i (4D^2-4iD+4D-2i+1)+4D^2+4D-1}, \\
& s_i =  D\frac{2i+1-(-1)^i}{(-1)^i(2D-2i+1)-2D-1} \\
& \qquad \times \frac{(-1)^i(4D-2i+3)+1}{(-1)^i (4D^2-4iD+4D-2i+1)+4D^2+4D-1},\\
 & z^-_i = -\frac{1}{D-1}\, \frac{1}{(-1)^i(2D-2i+1)-2D-1} \\
 &\qquad \times \frac{(-1)^i \bigl((2D-i)^2+4D-3i+1\bigr)+  (2D-i)  (4 iD - 2i^2 -2D+5i-2) +i-1}{ (-1)^i (4D^2-4iD+4D-2i+1)+4D^2+4D-1}.
\end{align*}
\noindent For $1 \leq i \leq D-1$,
\begin{align*}
& R_i = \frac{1}{2}\, \frac{4D-2i+3+(-1)^i}{(-1)^i(2D-2i+1)-2D-1} \\
& \qquad \times \frac{ 8iD^2- 4i^2D+4iD-2i^2+1-(-1)^i}{(-1)^i (4D^2-4iD+4D-2i+1)+4D^2+4D-1}, \\
& S_i =  D \frac{1-(-1)^i(2i+1)}{(-1)^i(2D-2i+1)-2D-1} \\
&\qquad \times \frac{4D-2i+3+(-1)^i}{(-1)^i (4D^2-4iD+4D-2i+1)+4D^2+4D-1},
 \\
& z^+_{i+1} =\frac{1}{D-1}\, \frac{1}{(-1)^i(2D-2i+1)-2D-1} \\
&\qquad \quad \times \frac{(-1)^i(2D-i^2-i+1)+4i^2D-2i^3+i^2-2D+i-1}{(-1)^i (4D^2-4iD+4D-2i+1)+4D^2+4D-1}.
\end{align*}
\noindent For $2 \leq i \leq D-1$,
\begin{align*}
&u_i = 0, \\
& v_i =- \frac{(-1)^i}{2(D-i)(D-i+1)}\\
&\qquad \times \frac{(-1)^i(2D^2-6iD+3i^2+3D-3i+1)+  (2iD-i^2-D+i-1)(2D-2i+1) }{(-1)^i (4D^2-4iD+4D-2i+1)+4D^2+4D-1},
\\
&w_i = 0, \qquad \Phi_i(\lambda) = v_i (\lambda-\xi), \qquad \xi=0.
\end{align*}
\end{example}

\begin{lemma}  \label{lem:odd} For $O_{D+1}$ the kite function $\zeta_i$ is constant for $2 \leq i \leq D$. Moreover  $z_i=0$ for $2 \leq i \leq D$. 
\end{lemma}
\begin{proof} The graph $O_{D+1}$ is almost bipartite, and hence kite-free.
\end{proof}

\begin{lemma}\rm 
We refer  to Example \ref{ex:odd2} and write $E=E_1$. 
The set $\lbrace E{\hat x} \vert x \in X\rbrace$ is Norton-balanced.  For $0 \leq i \leq D-1$ and $x,y \in X$ at distance $\partial(x,y)=i$,
\begin{align*} 
0 = E x^-_y + E x^+_y +D E{\hat x}.
\end{align*}
\end{lemma}
\begin{proof} The graph  $O_{D+1}$ is almost bipartite and $\theta_1 = -D$.
\end{proof}

\section{Example: the Grassmann graph}
\begin{example} \label{ex:Grassman1} \rm (See \cite[Chapter~6.4]{bbit}, \cite[Example~6.1(5)]{tSub3}.)
Let $GF(q)$ denote a finite field with cardinality $q$. Fix an integer $N\geq 2D$, and let $U$ denote a vector space over $GF(q)$ that has dimension $N$.
The Grassmann graph  $J_q(N,D)$  has vertex set $X$ consisting of the subspaces of $U$ that have dimension $D$.
Vertices $x,y \in X$ are adjacent  whenever $x \cap y$ has dimension $D-1$. The graph $J_q(N,D)$ is distance-regular
with diameter $D$ and intersection numbers 
\begin{align*}
c_i = \biggl(\frac{q^i-1}{q-1}\biggr)^2, \qquad \quad
b_i = q \frac{ q^D-q^i}{q-1} \frac{ q^{N-D}-q^i}{q-1} \qquad \quad (0 \leq i \leq D).
\end{align*}
\end{example}

\begin{example} \label{ex:Grassman2}\rm 
The graph $J_q(N,D)$ has a $Q$-polynomial structure such that
\begin{align*}
\theta_i &= q^{1-i} \frac{q^D-q^i}{q-1} \frac{q^{N-D}-q^i}{q-1} - \frac{q^i-1}{q-1} \qquad \quad (0 \leq i \leq D), \\
\theta^*_i &= \frac{q^N-q}{q-1} - q^{-i} \frac{q^N-q}{q^D-1} \frac{q^N-1}{q^{N-D}-1}\frac{q^i-1}{q-1} \qquad (0 \leq i \leq D).
\end{align*}
This $Q$-polynomial structure has dual $q$-Hahn type with
\begin{align*}
&r = q^{D-N-1}, \qquad \qquad s= q^{-N-2}, \qquad \qquad h = \frac{q^{N+1}}{(q-1)^2}, \\
& h^*= \frac{q(q^N-1)(q^{N-1}-1)}{(q-1)(q^D-1)(q^{N-D}-1)}.
\end{align*}
 $Q$-polynomial structure is DC if and only if $N=2D$ (provided that $D\geq 4$). Assume that $N=2D$. We have
\begin{align*}
\gamma^* = \frac{2(q-1)(q^{2D-1}-1)}{q^D-1} \not=0.
\end{align*}
\noindent For $2 \leq i \leq D$,
\begin{align*}
&\alpha_i = \frac{q^{i-1}-1}{q-1}, \qquad \qquad \beta_i = 0, \\
&r_i = \frac{q^i-1}{q-1}\, \frac{q^{i-1}-1}{q-1}\, \frac{q^i(q^D-2q-1)+q^{D+1}+q}{q^i(q^D-3)+q^D+1}, \\
& s_i = \frac{q^{i-1}(q^i-1)}{q-1} \, \frac{ q^{D+1}+q^D+1-q-2q^i}{q^i(q^D-3)+q^D+1}, \\
& z^-_i = 2q \frac{q^{i-1}-1}{q-1} \, \frac{q^{2i} -q^i(q+3)+q^{D+1}+q^D+1}{q^i(q^D-3)+q^D+1}.
\end{align*}
\noindent For $1 \leq i \leq D-1$,
\begin{align*}
& R_i = \biggl( \frac{q^D-q^i}{q-1}\biggr)^2 \frac{ q^i(q^D-2q-1)+q^D+1}{q^i (q^D-3)+q^D+1}, \\
& S_i =- \frac{2q^i}{q-1} \; \frac{(q^D-q^i)^2}{q^i(q^D-3)+q^D+1},    \\
& z^+_{i+1} = 2q \frac{q^i-1}{q-1} \, \frac{1+ 2q^i(q^D-1)-q^{2i}}{q^i(q^D-3)+q^D+1}.
\end{align*}
\noindent For $2 \leq i \leq D-1$,
\begin{align*}
&u_i =  - \frac{q^i-1}{q-1}\, \frac{q^{i-1}-1}{q-1}, \\
& v_i = 2q \frac{ q^i-1}{q-1}\, \frac{q^{i-1}-1}{q-1}\, \frac{q^i(2q^D-q-5)+q^{D+1}+q^D+2}{q^i(q^D-3)+q^D+1}       ,\\
&w_i = -4q^2 \frac{ q^i-1}{q-1}\, \frac{q^{i-1}-1}{q-1} \, \frac{q^i(q^D-q-2)+q^{D+1}+1}{q^i(q^D-3)+q^D+1},    \\
&\Phi_i(\lambda) = u_i (\lambda-\xi)(\lambda-\xi_i) \qquad \xi=2q, \qquad \xi_i = 2q \frac{q^i(q^D-q-2)+q^{D+1}+1}{q^i(q^D-3)+q^D+1}.
\end{align*}
\end{example}

\begin{lemma} \label{lem:grassman} For $J_q(N,D)$ the kite function $\zeta_i$ is constant for $2 \leq i \leq D$. Moreover
\begin{align*}
z_i = 2 q \frac{q^{i-1}-1}{q-1} \qquad \qquad (2 \leq i \leq D).
\end{align*}
\end{lemma}
\begin{proof} By combinatorial counting.
\end{proof}

\begin{lemma} We refer to Example
\ref{ex:Grassman2} with $N=2D$. Write $E=E_1$. 
Then the set $\lbrace E{\hat x} \vert x \in X\rbrace$ is Norton-balanced.
Pick distinct $x,y \in X$ and write $i=\partial(x,y)$. For $2 \leq i \leq D-1$,
\begin{align} \label{eq:step1}
0 &= E{x^-_y} +   \frac{ q^{i-1}-1}{q^D-q^i} E{x^+_y} - \frac{q^{i-1}-1}{q-1} \frac{q^D-q}{q-1}  E{\hat x} - q^{i-1} E{\hat y}.
\end{align}
For $i=D$,
\begin{align} \label{eq:step2}
0 &=E{x^-_y}  -q \biggl(\frac{q^{D-1}-1}{q-1}\biggr)^2  E{\hat x} - q^{D-1} E{\hat y}.
\end{align}
\end{lemma}
\begin{proof}  To get the first assertion, we use Proposition \ref{cor:NBPhi2}(ii).
Pick an integer $i$ $(2 \leq i \leq D-1)$. We verify the conditions in \eqref{eq:2conditions}. 
We have $z_2=\xi$, so $\Phi_i(z_2)=0$. We have
\begin{align*}
 \lambda_i = - \frac{q^{i-1}-1}{q^D-q^i},     \qquad \qquad \frac{\theta^*_{i}-\theta^*_{i+1}}{\theta^*_{i-1}-\theta^*_i} = q^{-1}.
\end{align*}
Therefore
\begin{align*}
\frac{\theta^*_{i}-\theta^*_{i+1}}{\theta^*_{i-1}-\theta^*_i} -\lambda_i= \frac{q^{D-1}-1}{q^D-q^i}\not=0.
\end{align*}
We have verified the conditions in  \eqref{eq:2conditions}, so the set $\lbrace E{\hat x} \vert x \in X\rbrace$ is Norton-balanced.
The linear dependence  \eqref{eq:step1} is obtained using  \eqref{eq:lamiR}, \eqref{lem:lambdaCALC}.
To obtain  \eqref{eq:step2}, use Proposition \ref{prop:AT}
and $z^-_D=z_D$.
\end{proof}

\section{Example: the dual polar graphs}
\begin{example} \label{ex:dp} \rm  (See \cite[Chapter~6.4]{bbit}, \cite[Example~6.1(6)]{tSub3}.)
Let $U$ denote a finite vector space with one of the following nondegenerate forms:
\begin{align*}
\begin{tabular}[t]{ccccc}
{\rm name }& ${\rm dim}(U)$ & {\rm field} & form &  $e$
 \\
 \hline
$ B_D(p^n) $ & $2D+1$ & $GF(p^n)$ & quadratic & $0$ \\ 
$ C_D(p^n) $ & $2D$ & $GF(p^n)$ & symplectic & $0$ \\ 
$ D_D(p^n) $ & $2D$ & $GF(p^n)$ & quadratic & $-1$ \\ 
&&& {\rm (Witt index $D$)} & \\
$ {}^2 D_{D+1}(p^n) $ & $2D+2$ & $GF(p^n)$ & quadratic & $1$ \\ 
&&& {\rm (Witt index $D$)} & \\
$ {}^2A_{2D}(p^n) $ & $2D+1$ & $GF(p^{2n})$ & Hermitean & $1/2$ \\ 
$ {}^2A_{2D-1}(p^n) $ & $2D$ & $GF(p^{2n})$ & Hermitean & $-1/2$
    \end{tabular}
\end{align*}
A subspace of $U$ is called {\it isotropic} whenever the form vanishes completely
on that subspace. In each of the above cases, the dimension of any maximal
isotropic subspace is $D$.
The corresponding dual polar graph $\Gamma$ has vertex set $X$ consisting of the maximal isotropic subspaces of $U$. Vertices $x,y \in X$ are adjacent whenever $x \cap y$ has dimension $D-1$.
The graph $\Gamma$ is distance-regular with diameter $D$ and intersection numbers 
\begin{align*}
c_i = \frac{q^i-1}{q-1}, \qquad \quad
b_i = q^{e+1} \frac{q^D-q^i}{q-1} \qquad \quad (0 \leq i \leq D),
\end{align*}
where  $q = p^n, p^n, p^n, p^n, p^{2n}, p^{2n}$. The dual polar graph $D_D(q)$ is bipartite.
\end{example}

\begin{example} \label{ex:dpg2} \rm 
The dual polar graph $\Gamma $ has a $Q$-polynomial structure such that
\begin{align*}
\theta_i &= q^{e+1} \frac{q^D-1}{q-1}-\frac{(q^i-1)(q^{D+e+1-i}+1)}{q-1} \qquad \quad (0 \leq i \leq D), \\
\theta^*_i &= \frac{q^{D+e}+q}{q^e+1} \, \frac{q^{-i}(q^{D+e}+1)-q^e-1}{q-1} \qquad (0 \leq i \leq D).
\end{align*}
This $Q$-polynomial structure has dual $q$-Krawtchouk type with
\begin{align*}
&s = -q^{-D-e-2},  \qquad \quad h = \frac{q^{D+e+1}}{q-1}, \qquad \quad h^*= \frac{(q^{D+e}+1)(q^{D+e}+q)}{(q-1)(q^e+1)}.
\end{align*}
This $Q$-polynomial structure is DC if and only if $\Gamma=D_D(q)$ (provided that $D\geq 4$). For $\Gamma=D_D(q)$  we have the following.
\begin{align*}
\gamma^* = (q-1)(q^{D-2}+1) \not=0.
\end{align*}
\noindent For $2 \leq i \leq D$,
\begin{align*}
&\alpha_i = \frac{q^{i-1}-1}{q-1}, \qquad \qquad \beta_i = 0, \\
&r_i = \frac{q^{i-1}-1}{q-1} \, \frac{q^i(q^D-q^2-q-1)+q^{D+1}+q^2}{q^i(q^D-q-2)+q^D+q}, \\
& s_i = \frac{(q+1)q^{i-1}(q^D-q^i)}{q^i(q^D-q-2)+q^D+q}, \\
& z^-_i = -\frac{(q+1)(q^i-q)(q^D-q^i)}{q^i(q^D-q-2)+q^D+q}.
\end{align*}
\noindent For $1 \leq i \leq D-1$,
\begin{align*}
& R_i = \frac{1}{q}\,\frac{q^D-q^i}{q-1} \, \frac{q^i(q^D-q^2-q-1)+q^D+q}{q^i(q^D-q-2)+q^D+q}, \\
& S_i = -\frac{(q+1)q^{i-1}(q^D-q^i)}{q^i(q^D-q-2)+q^D+q},  \\
& z^+_{i+1} =\frac{(q+1)(q^i-1)(q^i-q)}{q^i(q^D-q-2)+q^D+q}.
\end{align*}
\noindent For $2 \leq i \leq D-1$,
\begin{align*}
&u_i =  - \frac{q^i-1}{q-1}\, \frac{q^{i-1}-1}{q-1}, \\
& v_i =  \frac{(q+1)(q^i-1)(q^{i-1}-1)}{q-1} \, \frac{q^i(q+1)-q^{D+1}-q} {q^i(q^D-q-2)+q^D+q}    ,\\
&w_i = 0,   \\
&\Phi_i(\lambda) = u_i (\lambda-\xi)(\lambda-\xi_i) \qquad \xi=0, \qquad \xi_i = (q^2-1)\frac{q^i(q+1)-q^{D+1}-q}{q^i(q^D-q-2)+q^D+q}.
\end{align*}
\end{example}

\begin{lemma} \label{lem:dpz} For a dual polar graph $\Gamma$ the kite function $\zeta_i$ is constant for $2 \leq i \leq D$. Moreover
\begin{align*}
z_i = 0 \qquad \qquad (2 \leq i \leq D).
\end{align*}
\end{lemma}
\begin{proof} The graph $\Gamma$ is a regular near polygon \cite[Section~6.4]{bcn} and hence kite-free.
\end{proof}

\begin{lemma} \label{lem:dpz2} We refer to Example \ref{ex:dpg2} with $\Gamma=D_D(q)$. Write $E=E_1$.
Then the set $\lbrace E{\hat x} \vert x \in X\rbrace$ is Norton-balanced. 
For  $x,y \in X$ we have
\begin{align*}
0 &= E{x^-_y} +   E{x^+_y} - \frac{q^{D-1}-q}{q-1}   E{\hat x}.
\end{align*}
\end{lemma}
\begin{proof}  The graph $D_D(q)$ is bipartite and $\theta_1 = (q^{D-1}-q)/(q-1)$.
\end{proof}

\noindent The dual polar graph ${}^2A_{2D-1}(p^n)$ has a second $Q$-polynomial structure, which we now describe.

\begin{example} \label{ex:Aalt} \rm  (See \cite[Chapter~6.4]{bbit}, \cite[Example~6.1(7)]{tSub3}.)
The intersection numbers of $ {}^2A_{2D-1}(p^n) $ can be expressed as
\begin{align*}
c_i = \frac{q^{2i}-1}{q^2-1}, \qquad\qquad b_i = - q^{2i+1}\frac{q^{2D-2i}-1}{q^2-1}
\qquad \quad (0 \leq i \leq D),
\end{align*}
where $q=-p^n$.
The graph 
$ {}^2A_{2D-1}(p^n) $ 
 has a $Q$-polynomial structure such that
\begin{align*}
\theta_i &= \frac{(q^i-1)(q^{2D-i+1}-1)}{q^2-1} - q \frac{q^{2D}-1}{q^2-1} \qquad \quad (0 \leq i \leq D), \\
\theta^*_i &= - q^{-i} \frac{ q^{2D}-q}{q-1} \qquad \qquad (0 \leq i \leq D).
\end{align*}
This $Q$-polynomial structure is almost dual-bipartite. It has dual $q$-Hahn type with
\begin{align*}
&r=-q^{-D-1},  \qquad \quad s=q^{-2D-2}, \qquad \quad h =-\frac{q^{2D+1}}{q^2-1}, \qquad \quad h^*= - \frac{q^{2D}-q}{q-1}.
\end{align*}
This $Q$-polynomial structure is DC with $\gamma^*=0$.  
\noindent For $2 \leq i \leq D$,
\begin{align*}
&\alpha_i = \frac{q^{i-1}-1}{q-1}, \qquad \qquad \beta_i = 0, \\
&r_i = q\frac{q^{2i-2}-1}{q^2-1}, \qquad \qquad  s_i = q^{i-1}, \qquad \qquad z^-_i = 0.
\end{align*}
\noindent For $1 \leq i \leq D-1$,
\begin{align*}
& R_i = - \frac{q^{2D}-q^{2i}}{q^2-1}, \qquad \qquad 
 S_i = 0, \qquad \qquad 
 z^+_{i+1} =0.
\end{align*}
\noindent For $2 \leq i \leq D-1$,
\begin{align*}
&u_i =  - \frac{q^i-1}{q-1}\, \frac{q^{i-1}-1}{q-1}, \qquad \qquad 
 v_i = 0, \qquad \qquad 
w_i = 0,   \\
&\Phi_i(\lambda) = u_i (\lambda-\xi)(\lambda-\xi_i), \qquad \quad  \xi=0, \qquad \qquad \xi_i = 0.
\end{align*}
\end{example}

\begin{lemma} We refer  to Example  \ref{ex:Aalt} and write $E=E_1$.
The set $\lbrace E{\hat x} \vert x \in X\rbrace$ is Norton-balanced.
Pick distinct $x,y \in X$ and write $i=\partial(x,y)$. For $2 \leq i \leq D$,
\begin{align}
&Ex^-_y = q \frac{q^{2i-2}-1}{q^2-1} E{\hat x} + q^{i-1} E{\hat y}. \label{eq:rt1}
\end{align}
\noindent For $1 \leq i \leq D-1$,
\begin{align}
 Ex^+_y = - \frac{q^{2D}-q^{2i}}{q^2-1} E{\hat x}. \label{eq:rt2}
\end{align}
\end{lemma}
\begin{proof}  To get \eqref{eq:rt1}, use Proposition \ref{prop:AT} and  $z^-_i = z_i$.
To get \eqref{eq:rt2}, use Proposition
 \ref{prop:PAT} and $z^+_{i+1} = z_{i+1}$.
 It follows from  \eqref{eq:rt1}, \eqref{eq:rt2}  that the set  $\lbrace E{\hat x} \vert x \in X\rbrace$ is Norton-balanced.
\end{proof}

\section{Example: the halved bipartite dual polar graph}

\noindent  Recall that the dual polar graph $D_D(p^n)$ is bipartite.
\begin{example} \label{ex:hdp} \rm  (See \cite[Chapter~6.4]{bbit}, \cite[Example~6.1(8)]{tSub3}.)
The halved graph $\frac{1}{2} D_{2D}(p^n)$ is distance-regular, with diameter $D$ and intersection numbers
\begin{align*}
c_i = \frac{q^i-1}{q-1}  \,\frac{q^{i-\frac{1}{2}}-1}{q^\frac{1}{2}-1}, \qquad \qquad b_i = \frac{q^D-q^i}{q-1} \, \frac{q^D-q^{i+\frac{1}{2}}}{q^\frac{1}{2}-1}
\qquad \qquad (0 \leq i \leq D),
\end{align*}
where $q=p^{2n}$.
\end{example}

\begin{example}\label{ex:hdp2} \rm  The graph  $\frac{1}{2} D_{2D}(p^n)$ has a $Q$-polynomial structure such that
\begin{align*}
&\theta_i = q^\frac{1}{2} \frac{q^D-1}{q-1}\, \frac{q^{D-\frac{1}{2}}-1}{q^\frac{1}{2}-1}-\frac{q^i -1}{q-1} \, \frac{q^{2D-i}-1}{q^\frac{1}{2}-1} \qquad \qquad (0 \leq i \leq D),
 \\
&\theta^*_i = q^\frac{1}{2} \frac{q^D-1}{q-1}\, \frac{q^{2D-1}-q}{q^D-q} - \frac{q^{D-\frac{1}{2}}+1}{q^{i-1}}\, \frac{q^i-1}{q-1}\, \frac{q^{2D-1}-q}{q^D-q} \qquad \qquad (0 \leq i \leq D).
\end{align*}
This $Q$-polynomial structure has dual $q$-Hahn type with
\begin{align*}
&r=q^{-D-\frac{1}{2}},  \qquad \qquad s=q^{-2D-1}, \qquad \qquad  h = \frac{q^{2D}}{(q-1)(q^\frac{1}{2}-1)},\\
&h^*=  \frac{(q^{D-\frac{1}{2}}+1)(q^{2D}-q^2)}{(q-1)(q^D-q)}.
\end{align*}
This $Q$-polynomial structure is DC and
\begin{align*}
\gamma^* = \frac{(q-1)(q^\frac{1}{2}+1)(q^{2D}-q^2)}{ q^\frac{3}{2} (q^D-q)} \not=0.
\end{align*}
\noindent For $2 \leq i \leq D$,
\begin{align*}
&\alpha_i = \frac{q^{i-1}-1}{q-1}, \qquad \qquad \beta_i = 0, \\
&r_i =\frac{q^{i-1}-1}{q-1}\, \frac{q^{i-\frac{1}{2}}-1}{q^\frac{1}{2}-1} \, \frac{q^i (q^{D+\frac{1}{2}}-q^2-q^\frac{3}{2}-q^\frac{1}{2})+ q^{D+\frac{3}{2}}+q^2}{q^i(q^{D+\frac{1}{2}}-q-2q^\frac{1}{2})+q^{D+\frac{1}{2}}+q} , \\
& s_i = \frac{q^{i-1}(q^{i-\frac{1}{2}}-1)}{q^\frac{1}{2}-1} \, \frac{q^{D+\frac{3}{2}}+q^{D+\frac{1}{2}}-q^\frac{3}{2}+q-q^i(q+q^\frac{1}{2})} {q^i(q^{D+\frac{1}{2}}-q-2q^\frac{1}{2})+q^{D+\frac{1}{2}}+q},     \\
& z^-_i = (q+ q^\frac{1}{2}) \frac{q^{i-1}-1}{q^\frac{1}{2}-1} \; \frac{ q^{2i} q^\frac{1}{2} - q^i(q^\frac{3}{2} + q+2q^\frac{1}{2})+q^{D+\frac{3}{2}} + q^{D+\frac{1}{2}}+q} {q^i(q^{D+\frac{1}{2}}-q-2q^\frac{1}{2})+q^{D+\frac{1}{2}}+q}.
\end{align*}
\noindent For $1 \leq i \leq D-1$,
\begin{align*}
& R_i = q^{-\frac{1}{2}} \frac{q^D-q^i}{q-1}\, \frac{q^{D-\frac{1}{2}}-q^i}{q^\frac{1}{2}-1}\, \frac{q^i(q^{D+\frac{1}{2}}-q^2-q^\frac{3}{2}-q^\frac{1}{2})+q^{D+\frac{1}{2}}+q}{ q^i(q^{D+\frac{1}{2}}-q-2q^\frac{1}{2})+q^{D+\frac{1}{2}}+q},  \\
& S_i = - \frac{q^D-q^i}{q^\frac{1}{2}-1} \, \frac{(q^\frac{1}{2}+1)q^i(q^{D-\frac{1}{2}}-q^i)} { q^i(q^{D+\frac{1}{2}}-q-2q^\frac{1}{2})+q^{D+\frac{1}{2}}+q}, \\
& z^+_{i+1} = \frac{(q^\frac{1}{2}+1)(q^i-1)}{q^\frac{1}{2}(q^\frac{1}{2}-1)} \, \frac{q^2+ q^i(q^{D+\frac{3}{2}}+q^{D+1} -2q^\frac{3}{2}) -q^{2i} q^\frac{3}{2}} { q^i(q^{D+\frac{1}{2}}-q-2q^\frac{1}{2})+q^{D+\frac{1}{2}}+q}.
\end{align*}
\noindent For $2 \leq i \leq D-1$,
\begin{align*}
&u_i =  - \frac{q^i-1}{q-1}\,\frac{q^{i-1}-1}{q-1}, \\
& v_i =  \frac{q^i-1}{q-1}\, \frac{(q^{i-1}-1)(q^\frac{1}{2}+1)}{q^\frac{1}{2}(q^\frac{1}{2}-1)} \,\frac{q^i(q^{D+\frac{3}{2}}+q^{D+1}-q^\frac{5}{2}-q^2-4q^\frac{3}{2})+q^{D+\frac{5}{2}}+q^{D+\frac{3}{2}}+2q^2} { q^i(q^{D+\frac{1}{2}}-q-2q^\frac{1}{2})+q^{D+\frac{1}{2}}+q} ,\\
&w_i = - \frac{(q^\frac{1}{2}+1)^2(q^i-1)(q^{i-1}-1)}{(q^\frac{1}{2}-1)^2} \, \frac{q^i(q^{D+1}-q^\frac{5}{2}-2q^\frac{3}{2})+q^{D+\frac{5}{2}}+q^2} { q^i(q^{D+\frac{1}{2}}-q-2q^\frac{1}{2})+q^{D+\frac{1}{2}}+q},   \\
&\Phi_i(\lambda) = u_i (\lambda-\xi)(\lambda-\xi_i), \qquad \qquad  \xi=q^\frac{1}{2}(q^\frac{1}{2}+1)^2, \\
& \xi_i = \frac{(q^\frac{1}{2}+1)^2}{q^\frac{1}{2}}\,\frac{q^i(q^{D+1}-q^\frac{5}{2}-2q^\frac{3}{2})+q^{D+\frac{5}{2}}+q^2}{ q^i(q^{D+\frac{1}{2}}-q-2q^\frac{1}{2})+q^{D+\frac{1}{2}}+q}.
\end{align*}
\end{example}

\begin{lemma} For the graph   $\frac{1}{2} D_{2D}(p^n)$       the kite function $\zeta_i$ is constant for $2 \leq i \leq D$. Moreover
\begin{align*}
z_i = q^\frac{1}{2}\frac{(q-1)(q^{i-1}-1)}{(q^\frac{1}{2}-1)^2} \qquad \qquad (2 \leq i \leq D).
\end{align*}
\end{lemma}
\begin{proof} By combinatorial counting using \cite[Section~5]{miklavic}.
\end{proof}

\begin{lemma} We refer to Example \ref{ex:hdp2} and write $E=E_1$.
The set $\lbrace E{\hat x} \vert x \in X\rbrace$ is Norton-balanced.  
Pick distinct  $x,y \in X$ and write $i = \partial(x,y)$. For $2 \leq i \leq D-1$,
\begin{align}\label{eq:yu1}
0 &= E{x^-_y} +  \frac{q^{i-1}-1}{q^{D-\frac{1}{2}}-q^i} E{x^+_y} - \frac{q^{i-1}-1}{q-1}\,\frac{q^{D-\frac{1}{2}}-q}{q^\frac{1}{2}-1} E{\hat x} - q^{i-1} E{\hat y}.
\end{align}
\noindent For $i=D$,
\begin{align}\label{eq:yu2}
0 &= E{x^-_y}  - \frac{q^{D-1}-1}{q-1}\,\frac{q^{D-\frac{1}{2}}-q}{q^\frac{1}{2}-1} E{\hat x} - q^{D-1} E{\hat y}.
\end{align}
\end{lemma}
\begin{proof}  To get the first assertion, we use Proposition \ref{cor:NBPhi2}(ii).
Pick an integer $i$ $(2 \leq i \leq D-1)$. We verify the conditions in \eqref{eq:2conditions}. 
We have $z_2=\xi$, so $\Phi_i(z_2)=0$. We have
\begin{align*}
\lambda_i =   \frac{q^{i-1}-1}{q^i-q^{D-\frac{1}{2}}},     \qquad \qquad \frac{\theta^*_{i}-\theta^*_{i+1}}{\theta^*_{i-1}-\theta^*_i} = q^{-1}.
\end{align*}
Therefore
\begin{align*}
\frac{\theta^*_{i}-\theta^*_{i+1}}{\theta^*_{i-1}-\theta^*_i} -\lambda_i= \frac{q^{D-1}-q^{\frac{1}{2}}}{q^D-q^{i+\frac{1}{2}}}\not=0.
\end{align*}
We have verified the conditions in  \eqref{eq:2conditions}, so the set $\lbrace E{\hat x} \vert x \in X\rbrace$ is Norton-balanced.
The linear dependence  \eqref{eq:yu1} is obtained using  \eqref{eq:lamiR}, \eqref{lem:lambdaCALC}.
To obtain  \eqref{eq:yu2}, use Proposition \ref{prop:AT}
and $z^-_D=z_D$.
\end{proof}

\begin{example}\label{ex:hdb3} \rm (See\cite[Chapter~6.4]{bbit},  \cite[Example~6.1(9)]{tSub3}.)
The halved graph $\frac{1}{2} D_{2D+1}(p^n)$ is distance-regular, with diameter $D$ and intersection numbers
\begin{align*}
c_i = \frac{q^i-1}{q-1}  \,\frac{q^{i-\frac{1}{2}}-1}{q^\frac{1}{2}-1}, \qquad \qquad b_i = q^\frac{1}{2}\frac{q^D-q^i}{q-1} \, \frac{q^{D+\frac{1}{2}}-q^i}{q^\frac{1}{2}-1}
\qquad \qquad (0 \leq i \leq D),
\end{align*}
where $q=p^{2n}$.
\end{example}

\begin{example}\label{ex:hdb4}\rm  The graph  $\frac{1}{2} D_{2D+1}(p^n)$ has a $Q$-polynomial structure such that
\begin{align*}
&\theta_i = q^\frac{1}{2} \frac{q^D-1}{q-1}\, \frac{q^{D+\frac{1}{2}}-1}{q^\frac{1}{2}-1}-\frac{q^i -1}{q-1} \, \frac{q^{2D-i+1}-1}{q^\frac{1}{2}-1} \qquad \qquad (0 \leq i \leq D),
 \\
&\theta^*_i =  \frac{q^{D+\frac{1}{2}}-1}{q-1}\, \frac{q^{2D}-q}{q^D-q^\frac{1}{2}} - \frac{q^{D}+1}{q^{i}}\, \frac{q^i-1}{q-1}\, \frac{q^{2D}-q}{q^{D-\frac{1}{2}}-1} \qquad \qquad (0 \leq i \leq D).
\end{align*}
This $Q$-polynomial structure has dual $q$-Hahn type with
\begin{align*}
&r=q^{-D-\frac{3}{2}},  \qquad \qquad s=q^{-2D-2}, \qquad \qquad  h = \frac{q^{2D+1}}{(q-1)(q^\frac{1}{2}-1)},\\
&h^*=  \frac{(q^{D}+1)(q^{2D}-q)}{(q-1)(q^{D-\frac{1}{2}}-1)}.
\end{align*}
This $Q$-polynomial structure is DC and
\begin{align*}
\gamma^* = \frac{(q-1)(q^\frac{1}{2}+1)(q^{2D-1}-1)}{ q^D-q^\frac{1}{2} } \not=0.
\end{align*}
\noindent For $2 \leq i \leq D$,
\begin{align*}
&\alpha_i = \frac{q^{i-1}-1}{q-1}, \qquad \qquad \beta_i = 0, \\
&r_i =\frac{q^{i-1}-1}{q-1}\, \frac{q^{i-\frac{1}{2}}-1}{q^\frac{1}{2}-1} \, \frac{q^i (q^{D+\frac{1}{2}} -q^\frac{3}{2}-q-1)+ q^{D+\frac{3}{2}}+q^\frac{3}{2}}{q^i(q^{D+\frac{1}{2}}-q^\frac{1}{2}-2)+q^{D+\frac{1}{2}}+q^\frac{1}{2}} , \\
& s_i = \frac{q^{i-1}(q^{i-\frac{1}{2}}-1)}{q^\frac{1}{2}-1} \, \frac{q^{D+\frac{3}{2}}+q^{D+\frac{1}{2}}-q+q^\frac{1}{2}-q^i(q^\frac{1}{2}+1)} {q^i(q^{D+\frac{1}{2}}-q^\frac{1}{2}-2)+q^{D+\frac{1}{2}}+q^\frac{1}{2}},     \\
& z^-_i = (q+ q^\frac{1}{2}) \frac{q^{i-1}-1}{q^\frac{1}{2}-1} \; \frac{ q^{2i}  - q^i(q + q^\frac{1}{2}+2)+q^{D+\frac{3}{2}} + q^{D+\frac{1}{2}}+q^\frac{1}{2}} {q^i(q^{D+\frac{1}{2}}-q^\frac{1}{2}-2)+q^{D+\frac{1}{2}}+q^\frac{1}{2}}.
\end{align*}
\noindent For $1 \leq i \leq D-1$,
\begin{align*}
& R_i =  \frac{q^D-q^i}{q-1}\, \frac{q^{D}-q^{i-\frac{1}{2}}}{q^\frac{1}{2}-1}\, \frac{q^i(q^{D+\frac{1}{2}}-q^\frac{3}{2}-q-1)+q^{D+\frac{1}{2}}+q^\frac{1}{2}}{ q^i(q^{D+\frac{1}{2}}-q^\frac{1}{2}-2)+q^{D+\frac{1}{2}}+q^\frac{1}{2}},  \\
& S_i = - \frac{q^D-q^i}{q^\frac{1}{2}-1} \, \frac{(q^\frac{1}{2}+1)q^i(q^{D}-q^{i-\frac{1}{2}})} { q^i(q^{D+\frac{1}{2}}-q^\frac{1}{2}-2)+q^{D+\frac{1}{2}}+q^\frac{1}{2}}, \\
& z^+_{i+1} =q^\frac{1}{2}\frac{(q^\frac{1}{2}+1)(q^i-1)}{q^\frac{1}{2}-1} \, \frac{q^\frac{1}{2}+ q^i(q^{D+\frac{1}{2}}+q^{D} -2) -q^{2i} } { q^i(q^{D+\frac{1}{2}}-q^\frac{1}{2}-2)+q^{D+\frac{1}{2}}+q^\frac{1}{2}}.
\end{align*}
\noindent For $2 \leq i \leq D-1$,
\begin{align*}
&u_i =  - \frac{q^i-1}{q-1}\,\frac{q^{i-1}-1}{q-1}, \\
& v_i =  \frac{q^i-1}{q-1}\, \frac{(q^{i-1}-1)(q+q^\frac{1}{2})}{q^\frac{1}{2}-1} \,\frac{q^i(q^{D+\frac{1}{2}}+q^{D}-q-q^\frac{1}{2}-4)+q^{D+\frac{3}{2}}+q^{D+\frac{1}{2}}+2q^\frac{1}{2}} { q^i(q^{D+\frac{1}{2}}-q^\frac{1}{2}-2)+q^{D+\frac{1}{2}}+q^\frac{1}{2}} ,\\
&w_i = - \frac{q(q^\frac{1}{2}+1)^2(q^i-1)(q^{i-1}-1)}{(q^\frac{1}{2}-1)^2} \, \frac{q^i(q^{D}-q-2)+q^{D+\frac{3}{2}}+q^\frac{1}{2}} { q^i(q^{D+\frac{1}{2}}-q^\frac{1}{2}-2)+q^{D+\frac{1}{2}}+q^\frac{1}{2}},   \\
&\Phi_i(\lambda) = u_i (\lambda-\xi)(\lambda-\xi_i), \qquad \qquad  \xi=q^\frac{1}{2}(q^\frac{1}{2}+1)^2, \\
& \xi_i = q^\frac{1}{2}(q^\frac{1}{2}+1)^2\,\frac{q^i(q^{D}-q-2)+q^{D+\frac{3}{2}}+q^\frac{1}{2}}{ q^i(q^{D+\frac{1}{2}}-q^\frac{1}{2}-2)+q^{D+\frac{1}{2}}+q^\frac{1}{2}}.
\end{align*}
\end{example}

\begin{lemma} For the graph   $\frac{1}{2} D_{2D+1}(p^n)$       the kite function $\zeta_i$ is constant for $2 \leq i \leq D$. Moreover
\begin{align*}
z_i = q^\frac{1}{2}\frac{(q-1)(q^{i-1}-1)}{(q^\frac{1}{2}-1)^2} \qquad \qquad (2 \leq i \leq D).
\end{align*}
\end{lemma}
\begin{proof} By combinatorial counting using \cite[Section~5]{miklavic}.
\end{proof}

\begin{lemma} We refer to Example  \ref{ex:hdb4} and write $E=E_1$.
The set $\lbrace E{\hat x} \vert x \in X\rbrace$ is Norton-balanced.  
For $2 \leq i \leq D-1$ and  $x,y \in X$ at distance $ \partial(x,y)=i$,
\begin{align} \label{eq:cd1}
0 &= E{x^-_y} +  \frac{q^{i-1}-1}{q^{D}-q^i} E{x^+_y} - \frac{q^{i-1}-1}{q-1}\,\frac{q^{D}-q}{q^\frac{1}{2}-1} E{\hat x} - q^{i-1} E{\hat y}.
\end{align} 
\end{lemma}
\begin{proof}  To get the first assertion, we use Proposition \ref{cor:NBPhi2}(ii).
Pick an integer $i$ $(2 \leq i \leq D-1)$. We verify the conditions in \eqref{eq:2conditions}. 
We have $z_2=\xi$, so $\Phi_i(z_2)=0$. We have
\begin{align*}
\lambda_i =   \frac{q^{i-1}-1}{q^i-q^{D}} ,     \qquad \qquad \frac{\theta^*_{i}-\theta^*_{i+1}}{\theta^*_{i-1}-\theta^*_i} = q^{-1}.
\end{align*}
Therefore
\begin{align*}
\frac{\theta^*_{i}-\theta^*_{i+1}}{\theta^*_{i-1}-\theta^*_i} -\lambda_i= \frac{q^{D-1}-1}{q^D-q^i}\not=0.
\end{align*}
We have verified the conditions in  \eqref{eq:2conditions}, so the set $\lbrace E{\hat x} \vert x \in X\rbrace$ is Norton-balanced.
The linear dependence  \eqref{eq:cd1} is obtained using  \eqref{eq:lamiR}, \eqref{lem:lambdaCALC}.
\end{proof}

\section{Example: the Hemmeter graph}

\begin{example}\label{ex:hem} \rm (See \cite[Chapter~6.4]{bbit}, \cite[Example~6.1(10)--(12)]{tSub3}.)
Let  $GF(p^n)$ denote a finite field  with $p$ odd. The Hemmeter graph $Hem_D(p^n)$ is described in \cite[p.~383]{bbit};
it is  distance-regular 
with diameter $D$ and intersection numbers
\begin{align*}
c_i = \frac{q^i-1}{q-1}, \qquad \qquad b_i = \frac{q^D-q^i}{q-1} \qquad \qquad (0 \leq i \leq D),
\end{align*}
where $q=p^n$.
Note that $Hem_D(p^n)$ has the same intersection numbers as
$D_D(p^n)$. The graphs $Hem_D(p^n)$ and $D_D(p^n)$ 
 are not isomorphic.
The assertions about $D_D(p^n)$ in Example \ref{ex:dpg2} and Lemmas \ref{lem:dpz}, \ref{lem:dpz2}  hold for $Hem_D(p^n)$ as well. 
 Note that $Hem_D(p^n)$ is bipartite.
The assertions about $\frac{1}{2} D_{2D}(p^n)$ in Section 21 hold for  $\frac{1}{2} Hem_{2D}(p^n)$ as well.
The assertions about $\frac{1}{2} D_{2D+1}(p^n)$ in Section 21 hold for  $\frac{1}{2} Hem_{2D+1}(p^n)$ as well.
\end{example}

\section{Example: the Hamming graph}

\begin{example}\label{ex:hypercube} \rm (See \cite[Chapter~6.4]{bbit}, \cite[Example~6.1(13)]{tSub3}.)
For an integer $N\geq 2$, the {\it Hamming graph} $H(D,N)$  has vertex set $X$ consisting of the $D$-tuples of elements taken from the set $\lbrace 1,2,\ldots, N\rbrace$. 
Vertices $x,y \in X$ are adjacent  whenever $x,y$ differ in exactly one coordinate. The graph $H(D,N)$ is distance-regular with diameter $D$ and intersection numbers
\begin{align*}
c_i  = i, \qquad \qquad b_i = (N-1)(D-i) \qquad \qquad (0 \leq i \leq D). 
\end{align*}
The graph $H(D,2)$ is often called a {\it $D$-cube} or  {\it hypercube}. It is bipartite and an antipodal 2-cover.
\end{example}

\begin{example} \label{ex:hypercube2} \rm
The graph $H(D,N)$ 
has a $Q$-polynomial structure such that
\begin{align*}
\theta_i = \theta^*_i = D(N-1)-i N \qquad \qquad (0 \leq i \leq D). 
\end{align*}
This $Q$-polynomial structure has Krawtchouk type, with
\begin{align*}
s=-N, \qquad \qquad s^*=-N, \qquad \qquad r=N(N-1).
\end{align*}
This $Q$-polynomial structure is DC with $\gamma^*=0$.
\medskip

\noindent Until further notice, assume that $N\geq 3$. For $2 \leq i \leq D$,
\begin{align*}
&\alpha_i = i-1, \qquad \quad \beta_i=0, \\
&r_i = i-1, \qquad \quad s_i = 1, \qquad \quad
z^-_i = 0.
\end{align*}
For $1 \leq i \leq D-1$,
\begin{align*}
& R_i =\frac{(N-1)(D-i)(2DN-iN-N-2D)}{2DN-iN-2D},\\
&S_i = -\frac{N(N-1)(D-i)}{2DN-iN-2D},\\
& z^+_{i+1} = \frac{N(N-2)i}{2DN-iN-2D}.
\end{align*}
\noindent For $2 \leq i \leq D-1$,
\begin{align*}
&u_i = -i(i-1), \qquad \quad 
v_i = \frac{i(i-1)N(N-2)}{2DN-iN-2D}, \qquad \quad 
w_i = 0,\\
&\Phi_i(\lambda) = u_i(\lambda-\xi)(\lambda-\xi_i), \qquad \xi=0, \qquad \xi_i = \frac{N(N-2)}{2DN-iN-2D}.
\end{align*}
\noindent We have been assuming that $N\geq 3$. From now until the beginning of Lemma  \ref{lem:hamming3}, we assume that $N=2$. For $2 \leq i \leq D$,
\begin{align*}
&\alpha_i = i-1, \qquad \qquad \beta_i=0.
\end{align*}
For $2 \leq i \leq D-1$,
\begin{align*}
&r_i = i-1, \qquad \quad s_i = 1, \qquad \quad z^-_i=0.
\end{align*}
For $1 \leq i \leq D-1$,
\begin{align*}
& R_i =D-i-1, \qquad \quad 
S_i = -1, \qquad \quad 
 z^+_{i+1} = 0.
\end{align*}
\noindent For $2 \leq i \leq D-1$,
\begin{align*}
&u_i = -i(i-1), \qquad \quad 
v_i = 0, \qquad \quad 
w_i = 0,\\
&\Phi_i(\lambda) = u_i(\lambda-\xi)(\lambda-\xi_i), \qquad \xi=0, \qquad \xi_i = 0.
\end{align*}
\end{example}

\begin{lemma}  \label{lem:hamming3} For $H(D,N)$  the kite function $\zeta_i$ is constant for $2 \leq i \leq D$. Moreover  $z_i=0$ for $2 \leq i \leq D$. 
\end{lemma}
\begin{proof} The graph $H(D,N)$ is a regular near polygon, and hence kite-free.
\end{proof}

\begin{lemma} \label{lem:hypercube}
We refer to Example \ref{ex:hypercube2} and write $E=E_1$.
 Pick distinct  $x,y \in X$ and write $i=\partial(x,y)$. For $2 \leq i \leq D$,
\begin{align} 
E x^-_y = (i-1) E {\hat x} + E{\hat y}. \label{eq:rr1}
\end{align}
For $1 \leq i \leq D-1$ and $N=2$,
\begin{align}
Ex^+_y= (D-i-1) E{\hat x} - E{\hat y}. \label{eq:rr2}
\end{align}
In any case, the set $\lbrace E{\hat x} \vert x \in X\rbrace$ is Norton-balanced.
\end{lemma}
\begin{proof}  To get \eqref{eq:rr1}, use Proposition \ref{prop:AT} and  $z^-_i = z_i$.
To get \eqref{eq:rr2}, use Proposition
 \ref{prop:PAT} and $z^+_{i+1} = z_{i+1}$.
 It follows from Lemma  \ref{lem:clarify2} and  \eqref{eq:rr1}  that the set  $\lbrace E{\hat x} \vert x \in X\rbrace$ is Norton-balanced.
\end{proof}

\noindent For $D$ even, the hypercube  $H(D,2)$ 
has a second $Q$-polynomial structure that we now describe.
\begin{example} \label{ex:hypercubenew} \rm (See \cite[Chapter~6.4]{bbit}, \cite[Example~6.1(14)]{tSub3}.)
Assume that $D$ is even.
The hypercube $H(D,2)$ 
has a $Q$-polynomial structure such that
\begin{align*}
&\theta_i = \theta^*_i =  (-1)^i (D-2i) \qquad \qquad (0 \leq i \leq D).
\end{align*}
This $Q$-polynomial structure has Bannai/Ito type, with
\begin{align*}
r_1 = r_2 = - (D+1)/2, \qquad \qquad s = s^*= D+1, \qquad \qquad h = h^* = -1.
\end{align*}
This $Q$-polynomial structure is DC with $\gamma^*=0$.
\medskip

\noindent For $2 \leq i \leq D$,
\begin{align*}
&\alpha_i = \frac{(D-3)\bigl(1+(-1)^i\bigr)}{2(-1)^i(D-2i+1)}, \qquad \quad \beta_i= \frac{1- (-1)^i(2i-3)}{2(-1)^i(D-2i+1)}.
\end{align*}
\noindent For $2 \leq i \leq D-1$,
\begin{align*}
&r_i = 1-i, \qquad \qquad s_i = (-1)^{i-1}, \qquad \qquad z^-_i = 0.
\end{align*}
For $1 \leq i \leq D-1$,
\begin{align*}
& R_i =i+1-D, \qquad \qquad  S_i = (-1)^i, \qquad \qquad   z^+_{i+1} = 0.
\end{align*}
\noindent For $2 \leq i \leq D-1$,
\begin{align*}
&u_i = 0, \qquad \qquad 
v_i =0, \qquad \qquad 
w_i = 0, \qquad \qquad  \Phi_i(\lambda) =0.
\end{align*}
\end{example}

\begin{lemma} \label{lem:hypercubenew}
We refer  to Example \ref{ex:hypercubenew} and write $E=E_1$. 
The set $\lbrace E{\hat x} \vert x \in X\rbrace$ is Norton-balanced. Pick distinct $x,y \in X$ and write $i=\partial(x,y)$.
For $2 \leq i \leq D$,
\begin{align} 
E x^-_y = (1-i) E {\hat x} - (-1)^i E{\hat y}.
\label{eq:one1}
\end{align}
\noindent For $1 \leq i \leq D-1$,
\begin{align} \label{eq:one2}
Ex^+_y= (i+1-D) E{\hat x} +(-1)^i E{\hat y}.
\end{align}
\end{lemma}
\begin{proof}  To get \eqref{eq:one1}, use Proposition \ref{prop:AT} and  $z^-_i = z_i$.
To get \eqref{eq:one2}, use Proposition
 \ref{prop:PAT} and $z^+_{i+1} = z_{i+1}$.
 It follows from  \eqref{eq:one1}, \eqref{eq:one2} that the set  $\lbrace E{\hat x} \vert x \in X\rbrace$ is Norton-balanced.
\end{proof}

\section{Example: the halved  hypercube}

\noindent  Recall that the hypercube $H(D,2)$  is bipartite.
\begin{example}\label{ex:hc} \rm  (See  \cite[Chapter~6.4]{bbit}, \cite[Example~6.1(15)]{tSub3}.)
The halved graph $\frac{1}{2} H(2D,2)$ is distance-regular, with diameter $D$ and intersection numbers
\begin{align*}
c_i = i(2i-1), \qquad \qquad b_i = (D-i)(2D-1-2i)
\qquad \qquad (0 \leq i \leq D).
\end{align*}
The graph $\frac{1}{2} H(2D,2)$ is an antipodal 2-cover.
\end{example}

\begin{example}\label{ex:hc2} \rm  The graph  $\frac{1}{2} H(2D,2)$ has a $Q$-polynomial structure such that
\begin{align*}
&\theta_i = D(2D-1)-2i(2D-i), \qquad \qquad \theta^*_i = 2D-4i \qquad \qquad (0 \leq i \leq D).
 \end{align*}
This $Q$-polynomial structure has dual Hahn type with
\begin{align*}
&r=-D-1/2,  \qquad \qquad s=-2D-1, \qquad \qquad  s^*=-4, \qquad \qquad h=2.
\end{align*}
This $Q$-polynomial structure is DC with $\gamma^*=0$.
\noindent For $2 \leq i \leq D$,
\begin{align*}
&\alpha_i = i-1, \qquad \qquad \beta_i = 0. 
\end{align*}
\noindent For $2\leq i \leq D-1$,
\begin{align*}
&r_i =(2i-1)(i-1), \qquad \qquad
 s_i = 2i-1,\qquad \qquad       z^-_i = 4(i-1).
\end{align*}
\noindent For $1 \leq i \leq D-1$,
\begin{align*}
& R_i = (2D-2i-1)(D-i-1),  \qquad \quad 
S_i = 2i+1-2D, \qquad \quad 
 z^+_{i+1} = 4i.
\end{align*}
\noindent For $2 \leq i \leq D-1$,
\begin{align*}
&u_i =  -i(i-1), \qquad \qquad
 v_i = 8i(i-1), \qquad \qquad 
w_i = -16i(i-1),
\\
&\Phi_i(\lambda) = u_i (\lambda-\xi)(\lambda-\xi_i), \qquad \qquad  \xi=4, \qquad \qquad  \xi_i = 4.
\end{align*}
\end{example}

\begin{lemma} \label{lem:hc3} For the graph   $\frac{1}{2} H(2D,2)$       the kite function $\zeta_i$ is constant for $2 \leq i \leq D$. Moreover
\begin{align*}
z_i = 4(i-1)\qquad \qquad (2 \leq i \leq D).
\end{align*}
\end{lemma}
\begin{proof} By combinatorial counting.
\end{proof}

\begin{lemma} We refer to Example \ref{ex:hc2} and write $E=E_1$.
The set $\lbrace E{\hat x} \vert x \in X\rbrace$ is Norton-balanced.  
Pick distinct  $x,y \in X$ and write $i = \partial(x,y)$. For $2 \leq i \leq D$,
\begin{align}\label{eq:zyx1}
 &E{x^-_y}=  (2i-1)(i-1) E{\hat x} +(2i-1)E{\hat y}.
 \end{align}
 For $1 \leq i \leq D-1$,
 \begin{align}\label{eq:zyx2}
 & E{x^+_y} =(2D-2i-1)(D-i-1) E{\hat x} +(2i+1-2D)E{\hat y}.
\end{align}
\end{lemma}
\begin{proof}  To get \eqref{eq:zyx1}, use Proposition \ref{prop:AT} and  $z^-_i = z_i$.
To get \eqref{eq:zyx2}, use Proposition
 \ref{prop:PAT} and $z^+_{i+1} = z_{i+1}$.
 It follows from  \eqref{eq:zyx1}, \eqref{eq:zyx2} that the set  $\lbrace E{\hat x} \vert x \in X\rbrace$ is Norton-balanced.
\end{proof}

\begin{example}\label{ex:hc4} \rm  (See \cite[Chapter~6.4]{bbit},  \cite[Example~6.1(16)]{tSub3}.)
The halved graph $\frac{1}{2} H(2D+1,2)$ is distance-regular, with diameter $D$ and intersection numbers
\begin{align*}
c_i = i(2i-1), \qquad \qquad b_i = (D-i)(2D+1-2i)
\qquad \qquad (0 \leq i \leq D).
\end{align*}
\end{example}

\begin{example}\label{ex:hc5}\rm  The graph  $\frac{1}{2} H(2D+1,2)$ has a $Q$-polynomial structure such that
\begin{align*}
&\theta_i = D(2D+1)-2i(2D-i+1), \qquad \qquad \theta^*_i = 2D+1-4i \qquad \qquad (0 \leq i \leq D).
 \end{align*}
This $Q$-polynomial structure has dual Hahn type with
\begin{align*}
&r=-D-3/2,  \qquad \qquad s=-2D-2, \qquad \qquad  s^*=-4, \qquad \qquad h=2.
\end{align*}
This $Q$-polynomial structure is DC with $\gamma^*=0$.
\noindent For $2 \leq i \leq D$,
\begin{align*}
&\alpha_i = i-1, \qquad \qquad \beta_i = 0, \\
&r_i =(2i-1)(i-1), \qquad \qquad
 s_i = 2i-1,\qquad \qquad       z^-_i = 4(i-1).
\end{align*}
\noindent For $1 \leq i \leq D-1$,
\begin{align*}
& R_i = (2D-2i-1)(D-i),  \qquad \qquad 
S_i = 2(i-D), \qquad \qquad 
 z^+_{i+1} = 4i.
\end{align*}
\noindent For $2 \leq i \leq D-1$,
\begin{align*}
&u_i =  -i(i-1), \qquad \qquad
 v_i = 8i(i-1), \qquad \qquad 
w_i = -16i(i-1),
\\
&\Phi_i(\lambda) = u_i (\lambda-\xi)(\lambda-\xi_i), \qquad \qquad  \xi=4, \qquad \qquad  \xi_i = 4.
\end{align*}
\end{example}

\begin{lemma} \label{lem:hc6} For the graph   $\frac{1}{2} H(2D+1,2)$       the kite function $\zeta_i$ is constant for $2 \leq i \leq D$. Moreover
\begin{align*}
z_i = 4(i-1)\qquad \qquad (2 \leq i \leq D).
\end{align*}
\end{lemma}
\begin{proof} By combinatorial counting.
\end{proof}

\begin{lemma} We refer to Example \ref{ex:hc5} and write $E=E_1$.
The set $\lbrace E{\hat x} \vert x \in X\rbrace$ is Norton-balanced.  
Pick distinct  $x,y \in X$ and write $i = \partial(x,y)$.  For $2 \leq i \leq D$,
\begin{align} \label{eq:zy1}
 &E{x^-_y}=  (2i-1)(i-1) E{\hat x} +(2i-1)E{\hat y}.
 \end{align}
 \noindent For $1 \leq i \leq D-1$,
 \begin{align}
 & E{x^+_y} =(2D-2i-1)(D-i) E{\hat x} +2(i-D)E{\hat y}. \label{eq:zy2}
\end{align}
\end{lemma}
\begin{proof}  To get \eqref{eq:zy1}, use Proposition \ref{prop:AT} and  $z^-_i = z_i$.
To get \eqref{eq:zy2}, use Proposition
 \ref{prop:PAT} and $z^+_{i+1} = z_{i+1}$.
 It follows from  \eqref{eq:zy1}, \eqref{eq:zy2} that the set  $\lbrace E{\hat x} \vert x \in X\rbrace$ is Norton-balanced.
\end{proof}

\noindent We now give a second $Q$-polynomial structure for  $\frac{1}{2} H(2D+1,2)$.

\begin{example}\label{ex:hc7} \rm  (See \cite[Chapter~6.4]{bbit}, \cite[Example~6.1(18)]{tSub3}.)
 The graph  $\frac{1}{2} H(2D+1,2)$ has a $Q$-polynomial structure such that
\begin{align*}
&\theta_i = \theta^*_i = D(2D+1)-4i(2D-2i+1)  \qquad \qquad (0 \leq i \leq D).
 \end{align*}
This $Q$-polynomial structure has Racah type with
\begin{align*}
&r_1=-D/2-3/4,  \qquad \qquad r_2 = -D/2-5/4, \\
& s=s^*=-D-3/2, \qquad \quad h=h^*=8.
\end{align*}
This $Q$-polynomial structure is DC with $\gamma^*=16$.
\noindent For $2 \leq i \leq D$,
\begin{align*}
&\alpha_i = \frac{(i-1)(2D-5)(2D-2i+1)}{(2D-3)(2D-4i+3)}, \\
& \beta_i = - \frac{4(i-1)(i-2)}{(2D-3)(2D-4i+3)}, \\
&r_i = \frac{(i-1)(2i-1)\bigl(4i^2-2i(2D+3)+2D^2+D \bigr)}{4i^2-2i(2D+1)+2D^2+D}, \\
 &s_i = \frac{(2i-1)\bigl( 2D^2+D-2i(2D-1) \bigr)}{4i^2-2i(2D+1)+2D^2+D},\\ 
  &   z^-_i =  \frac{4(i-1)}{2D-5} \, \frac{ 4D^3-D-i(16D^2+4D-6)+8i^2(3D-1)-8i^3}{4i^2-2i(2D+1)+2D^2+D}.
\end{align*}
\noindent For $1 \leq i \leq D-1$,
\begin{align*}
& R_i =  \frac{(D-i)(2D-2i-1)\bigl(4i^2-2i(2D-1)+2D^2-3D-2 \bigr)}{4i^2-2i(2D+1)+2D^2+D}, \\
&S_i = - \frac{ 2(D-i)\bigl( 2D^2-D-1-2i(2D-1)\bigr)}{4i^2-2i(2D+1)+2D^2+D}, \\
& z^+_{i+1} = \frac{4i}{2D-5}\, \frac{ 8i^3 -8i^2 (D+2)+8i(2D+1) + 4D^3-12D^2-D-1}{4i^2 -2i(2D+1)+2D^2+D} .
\end{align*}
\noindent For $2 \leq i \leq D-1$,
\begin{align*}
&u_i =-\frac{i(i-1)(2D-5)^2 (2D-2i+1)(2D-2i-1)}{(2D-3)^2(2D-4i+3)(2D-4i-1)}, \\
& v_i = \frac{4i(i-1)(2D-2i+1)(2D-2i-1)}{(2D-3)^2(2D-4i+3)(2D-4i-1)}  \\
& \quad \times \frac{ 16D^4-64D^3+80D^2+3 -4i(8D^3-20D^2+14D+13)+8i^2(4D^2-12D+13)}{4i^2-2i(2D+1)+2D^2+D}, \\
&w_i = - \frac{16i(i-1)(2D-1)(2D-2i+1)(2D-2i-1)}{(2D-3)^2 (2D-4i-1)(2D-4i+3)} \\
& \quad \times \frac{4D^3-12D^2+19D-3 -2i(4D^2-1)+4i^2(2D-1)}{4i^2-2i(2D+1)+2D^2+D},
\\
&\Phi_i(\lambda) = u_i (\lambda-\xi)(\lambda-\xi_i), \qquad \qquad  \xi=4, \\
&  \xi_i =  \frac{4(2D-1)}{(2D-5)^2} \, \frac{4D^3-12 D^2+19D-3-2i(4D^2-1)+4i^2(2D-1)}{4i^2-2i(2D+1)+2D^2+D}.
\end{align*}
\end{example}

\begin{lemma}  We refer to Example \ref{ex:hc7} and write $E=E_1$. 
The set $\lbrace E{\hat x} \vert x \in X\rbrace$ is Norton-balanced.  
Pick distinct  $x,y \in X$ and write $i = \partial(x,y)$. For $2 \leq i \leq D-1$,
\begin{align} \label{eq:pr1}
 &0 = E{x^-_y}+ \frac{2(i-1)}{2D-2i-1}  E{x^+_y} - (i-1)(2D-5) E{\hat x} +\frac{2D-3}{2i-2D+1}E{\hat y}.
\end{align}
For $i=D$,
\begin{align} \label{eq:pr2}
0 = Ex^-_y -(D-1)(2D-5) E{\hat x} + (2D-3) E {\hat y}.
\end{align}
\end{lemma}
\begin{proof}  To get the first assertion, we use Proposition \ref{cor:NBPhi2}(ii).
Pick an integer $i$ $(2 \leq i \leq D-1)$. We verify the conditions in \eqref{eq:2conditions}. 
We have $z_2=\xi$, so $\Phi_i(z_2)=0$. We have
\begin{align*}
\lambda_i =  \frac{2(i-1)}{2i-2D+1} ,     \qquad \qquad \frac{\theta^*_{i}-\theta^*_{i+1}}{\theta^*_{i-1}-\theta^*_i} = \frac{2D-4i-1}{2D-4i+3}.
\end{align*}
Therefore
\begin{align*}
\frac{\theta^*_{i}-\theta^*_{i+1}}{\theta^*_{i-1}-\theta^*_i} -\lambda_i= \frac{(2D-5)(2D-4i+1)}{(2D-2i-1)(2D-4i+3)}\not=0.
\end{align*}
We have verified the conditions in  \eqref{eq:2conditions}, so the set $\lbrace E{\hat x} \vert x \in X\rbrace$ is Norton-balanced.
The linear dependence  \eqref{eq:pr1} is obtained using  \eqref{eq:lamiR}, \eqref{lem:lambdaCALC}.
To obtain  \eqref{eq:pr2}, use Proposition \ref{prop:AT}
and $z^-_D=z_D$.
\end{proof}

\section{Example: the folded  hypercube}

\noindent  Recall that the hypercube $H(D,2)$ is an antipodal 2-cover. Its antipodal quotient is called a folded cube.
\begin{example}\rm  (See \cite[Chapter~6.4]{bbit}, \cite[Example~6.1(20)]{tSub3}.)
The folded cube $\tilde H(2D,2)$ is distance-regular, with diameter $D$ and intersection numbers
\begin{align*}
&c_i = i  \qquad (1 \leq i \leq D-1), \qquad \quad c_D=2D, \\
&b_i = 2D-i \qquad \quad  (0 \leq i \leq D-1).
\end{align*}
\end{example}

\begin{example}\label{ex:fc2} \rm  The graph  $\tilde H(2D,2)$ has a $Q$-polynomial structure such that
\begin{align*}
&\theta_i =  2D-4i \qquad \qquad (0 \leq i \leq D),
 \\
&\theta^*_i = D(2D-1)-2i(2D-i) \qquad \qquad (0 \leq i \leq D).
\end{align*}
This $Q$-polynomial structure has Hahn type with
\begin{align*}
&r=-D-1/2,  \qquad \quad s=-4, \qquad \quad  s^* = -2D-1, \qquad \quad h^*=2.
\end{align*}
This $Q$-polynomial structure is DC with $\gamma^*=4$.
\noindent We have
\begin{align*}
&\alpha_i = \frac{(2D-3)(i-1)(2D-i)}{2(D-1)(2D-2i+1)} \qquad \qquad (2 \leq i \leq D), \\
&\beta_i = -\frac{(i-1)(i-2)}{2(D-1)(2D-2i+1)} \qquad \qquad (2 \leq i \leq D), \\
&r_i =\frac{(i-1)(i^2-i(2D+1)+2D^2-D)}{i^2-2iD+2D^2-D}\qquad (2 \leq i \leq D-1),\qquad r_D=2(D-2),\\
& s_i =  \frac{(2D-1)(D-i)}{i^2-2iD+2D^2-D}\qquad \qquad \qquad (2 \leq i \leq D),    \\
& z^-_i =  -\frac{2(i-1)(2D-i)(2D-i-1)}{(2D-3)(i^2-2iD+2D^2-D)} \qquad (2 \leq i \leq D-1),\qquad z^-_D=0.
\end{align*}
\noindent For $1 \leq i \leq D-1$,
\begin{align*}
& R_i =  \frac{(2D-i-1)(i^2-i(2D-1)+2D^2-3D)}{i^2 -2iD+2D^2-D},\\
& S_i =-\frac{(2D-1)(D-i)}{i^2 -2iD+2D^2-D}, \\
& z^+_{i+1} =\frac{2i(i-1)(2D-i-1)}{(2D-3)(i^2-2iD+2D^2-D)}.
\end{align*}
\noindent For $2 \leq i \leq D-1$,
\begin{align*}
&u_i = - \frac{i(i-1)(2D-3)^2(2D-i)(2D-i-1)}{4(D-1)^2 (2D-2i+1)(2D-2i-1)}, \\
& v_i = -\frac{2i(i-1)(2D-i)(2D-i-1)(2i^2-4iD+2D^2+D-1)}{(D-1)(2D-2i+1)(2D-2i-1)(i^2-2iD+2D^2-D)} ,\\
&w_i =0, \\
&\Phi_i(\lambda) = u_i (\lambda-\xi)(\lambda-\xi_i), \qquad \qquad  \xi=0, \\
& \xi_i =-\frac{8(D-1)(2i^2-4iD+2D^2+D-1)}{(2D-3)^2(i^2-2iD+2D^2-D)}.
\end{align*}
\end{example}

\begin{lemma} \label{lem:fc3} For the graph   $\tilde H(2D,2)$      the kite function $\zeta_i$ is constant for $2 \leq i \leq D$. Moreover
\begin{align*}
z_i =0 \qquad \qquad (2 \leq i \leq D).
\end{align*}
\end{lemma}
\begin{proof} The graph $\tilde H(2D,2) $ is bipartite, and hence kite-free.
\end{proof}

\begin{lemma} We refer to Example \ref{ex:fc2} and write $E=E_1$.
The set $\lbrace E{\hat x} \vert x \in X\rbrace$ is Norton-balanced.  
For  $x,y \in X$ we have
\begin{align*}
0 &= E{x^-_y} +   E{x^+_y} - 2(D-2) E{\hat x}.
\end{align*}
\end{lemma}
\begin{proof}   The graph $\tilde H(2D,2) $ is bipartite and $\theta_1=2(D-2)$.
\end{proof}

\begin{example}\label{ex:fc4} \rm (See \cite[Chapter~6.4]{bbit}, \cite[Example~6.1(19)]{tSub3}.)
The folded cube $\tilde H(2D+1,2)$ is distance-regular, with diameter $D$ and intersection numbers
\begin{align*}
&c_i = i  \qquad \qquad (1 \leq i \leq D), \\
&b_i = 2D+1-i \qquad \quad  (0 \leq i \leq D-1).
\end{align*}
\end{example}

\begin{example}\label{ex:fc5} \rm  The graph  $\tilde H(2D+1,2)$ has a $Q$-polynomial structure such that
\begin{align*}
&\theta_i =  2D+1-4i \qquad \qquad (0 \leq i \leq D),
 \\
&\theta^*_i = D(2D+1)-2i(2D-i+1) \qquad \qquad (0 \leq i \leq D).
\end{align*}
This $Q$-polynomial structure has Hahn type with
\begin{align*}
&r=-D-3/2,  \qquad \quad s=-4, \qquad \quad  s^* = -2D-2, \qquad \quad h^*=2.
\end{align*}
This $Q$-polynomial structure is DC with $\gamma^*=4$.
\noindent For $2 \leq i \leq D$,
\begin{align*}
&\alpha_i = \frac{(D-1)(i-1)(2D-i+1)}{(2D-1)(D-i+1)} , \\
&\beta_i = -\frac{(i-1)(i-2)}{2(2D-1)(D-i+1)}, \\
&r_i = \frac{(i-1)(i^2 -2i(D+1)+2D^2 +D)}{i^2-i(2D+1)+2D^2+D}, \\
& s_i =  \frac{D(2D-2i+1)}{i^2-i(2D+1)+2D^2 +D},   \\
& z^-_i = - \frac{(i-1)(2D-i)(2D-i+1)}{(D-1)(i^2-i(2D+1)+2D^2+D)}.
\end{align*}
\noindent For $1 \leq i \leq D-1$,
\begin{align*}
& R_i = \frac{(2D-i)(i^2-2iD+2D^2-D-1)}{i^2-i(2D+1)+2D^2+D},\\
& S_i = -\frac{D(2D-2i+1)}{i^2-i(2D+1)+2D^2+D},\\
& z^+_{i+1} =\frac{i(i-1)(2D-i)}{(D-1)(i^2-i(2D+1)+2D^2+D)}.
\end{align*}
\noindent For $2 \leq i \leq D-1$,
\begin{align*}
&u_i = -\frac{i(i-1)(D-1)^2(2D-i)(2D-i+1)}{(2D-1)^2(D-i)(D-i+1)}, \\
& v_i =  -\frac{i(i-1)(2D-i)(2D-i+1)(2i^2-2i(2D+1)+2D^2+3D)}{(D-i)(2D-1)(D-i+1)(i^2-i(2D+1)+2D^2+D)},\\
&w_i =0, \\
&\Phi_i(\lambda) = u_i (\lambda-\xi)(\lambda-\xi_i), \qquad \qquad  \xi=0, \\
& \xi_i =- \frac{(2D-1)(2i^2-2i(2D+1)+2D^2+3D)}{(D-1)^2 (i^2-i(2D+1)+2D^2+D)}.
\end{align*}
\end{example}

\begin{lemma} \label{lem:hcnew} For the graph   $\tilde H(2D+1,2)$      the kite function $\zeta_i$ is constant for $2 \leq i \leq D$. Moreover
\begin{align*}
z_i =0 \qquad \qquad (2 \leq i \leq D).
\end{align*}
\end{lemma}
\begin{proof} The graph $\tilde H(2D+1,2) $ is almost bipartite, and hence kite-free.
\end{proof}

\begin{lemma}  We refer to Example \ref{ex:fc5} and write $E=E_1$.
The set $\lbrace E{\hat x} \vert x \in X\rbrace$ is Norton-balanced.  
For $0 \leq i \leq D-1$ and $x,y \in X$ at distance $\partial(x,y)=i$,
\begin{align*}
0 &= E{x^-_y} +   E{x^+_y} +(3-2D) E{\hat x}.
\end{align*}
\end{lemma}
\begin{proof}   The graph $\tilde H(2D+1,2) $ is almost bipartite and $\theta_1 = 2D-3$.
\end{proof}

\begin{example} \label{ex:fhypercubenew} \rm (See \cite[Chapter~6.4]{bbit}, \cite[Example~6.1(17)]{tSub3}.)
The graph
 $\tilde H(2D+1,2)$ 
has a second $Q$-polynomial structure such that
\begin{align*}
&\theta_i = \theta^*_i =  (-1)^i (2D-2i+1) \qquad \qquad (0 \leq i \leq D).
\end{align*}
This $Q$-polynomial structure has Bannai/Ito type, with
\begin{align*}
r_1 = -D-1, \qquad \quad r_2 =-2D-2,
 \qquad \quad s = s^*= 2D+2, \qquad \quad h = h^* = -1.
\end{align*}
This $Q$-polynomial structure is DC with $\gamma^*=0$.
\medskip

\noindent For $2 \leq i \leq D$,
\begin{align*}
&\alpha_i = \frac{(D-1)\bigl( 1+(-1)^i\bigr)}{2(-1)^i(D-i+1)}, \qquad \quad \beta_i= \frac{1- (-1)^i(2i-3)}{4(-1)^i(D-i+1)}, \\
&r_i = 1-i, \qquad \qquad s_i = (-1)^{i-1}, \qquad \qquad z^-_i = 0.
\end{align*}
For $1 \leq i \leq D-1$,
\begin{align*}
& R_i =i-2D, \qquad \qquad  S_i = (-1)^i, \qquad \qquad   z^+_{i+1} = 0.
\end{align*}
\noindent For $2 \leq i \leq D-1$,
\begin{align*}
&u_i = 0, \qquad \qquad 
v_i =0, \qquad \qquad 
w_i = 0, \qquad \qquad  \Phi_i(\lambda) =0.
\end{align*}
\end{example}

\begin{lemma} \label{lem:fhypercubenew}
We refer to Example \ref{ex:fhypercubenew} and write $E=E_1$. 
The set $\lbrace E{\hat x} \vert x \in X\rbrace$ is Norton-balanced. Pick  distinct $x,y \in X$ and write $i=\partial(x,y)$.
For $2 \leq i \leq D$,
\begin{align}  \label{eq:pp1}
E x^-_y = (1-i) E {\hat x} - (-1)^i E{\hat y}.
\end{align}
\noindent For $1 \leq i \leq D-1$,
\begin{align} \label{eq:pp2}
Ex^+_y= (i-2D) E{\hat x} +(-1)^i E{\hat y}.
\end{align}
\end{lemma}
\begin{proof} To get \eqref{eq:pp1}, use Proposition \ref{prop:AT} and  $z^-_i = 0= z_i$.
To get \eqref{eq:pp2}, use Proposition
 \ref{prop:PAT} and $z^+_{i+1} = 0 = z_{i+1}$.
 It follows from  \eqref{eq:pp1}, \eqref{eq:pp2} that the set  $\lbrace E{\hat x} \vert x \in X\rbrace$ is Norton-balanced.
\end{proof}

\section{Example: the folded-half hypercube}

\begin{example}\rm (See \cite[Chapter~6.4]{bbit}, \cite[Example~6.1(21)]{tSub3}.)
The folded-half graph $\frac{1}{2}\tilde H(4D,2)$ is distance-regular, with diameter $D$ and intersection numbers
\begin{align*}
&c_i = i (2i-1)  \qquad (1 \leq i \leq D-1), \qquad \quad c_D=2D(2D-1), \\
&b_i = (2D-i)(4D-2i-1) \qquad \quad  (0 \leq i \leq D-1).
\end{align*}
\end{example}

\begin{example} \label{ex:fhc} \rm  The graph  $\frac{1}{2}\tilde H(4D,2)$ has a $Q$-polynomial structure such that
\begin{align*}
&\theta_i = \theta^*_i = 2D(4D-1)-8i(2D-i) \qquad \qquad (0 \leq i \leq D).
\end{align*}
This $Q$-polynomial structure has Racah type with
\begin{align*}
&r_1=-D-1/2,  \qquad   r_2 = -2D-1/2,    \qquad  s= s^* = -2D-1, \qquad  h=h^*=8.
\end{align*}
This $Q$-polynomial structure is DC with $\gamma^*=16$.
\noindent We have
\begin{align*}
&\alpha_i = \frac{(i-1)(2D-3)(2D-i)}{2(D-1)(2D-2i+1)} \qquad \qquad (2 \leq i \leq D), \\
&\beta_i = -\frac{(i-1)(i-2)}{2(D-1)(2D-2i+1)} \qquad \qquad (2 \leq i \leq D), \\
&r_i = \frac{(i-1)(2i-1)\bigl(2i^2-2i(2D+1)+4D^2-D\bigr)}{2i^2-4iD+4D^2-D}\qquad (2 \leq i \leq D-1),\\
&\qquad \qquad \qquad  \qquad \qquad \qquad r_D=2(D-1)(2D-3),\\
& s_i = \frac{(2i-1)(2i-4iD+4D^2-D)}{2i^2-4iD+4D^2-D}\qquad \quad  (2 \leq i \leq D-1), \qquad \quad  s_D=2,   \\
& z^-_i =\frac{4(i-1)\bigl(8D^3-6D^2+D-i(16D^2-6D-1)+i^2(12D-5)-2i^3\bigr)}{(2D-3)(2i^2-4iD+4D^2-D)} \\
   &\qquad \qquad \qquad \qquad  \qquad (2 \leq i \leq D-1),  
      \qquad  \qquad z^-_D=4(D-1).
\end{align*}
\noindent For $1 \leq i \leq D-1$,
\begin{align*}
& R_i =  \frac{(2D-i-1)(4D-2i-1)\bigl(2i^2-2i(2D-1)+4D^2-5D\bigr)}{2i^2-4iD+4D^2-D},\\
& S_i =- \frac{(4D-2i-1)(2i-4iD+4D^2-3D)}{2i^2-4iD+4D^2-D},\\
& z^+_{i+1} =\frac{4i\bigl( 2i^3-i^2(4D+3)+8iD+8D^3-18D^2+7D-1\bigr)}
{(2D-3)(2i^2-4iD+4D^2-D)}.
\end{align*}
\noindent For $2 \leq i \leq D-1$,
\begin{align*}
&u_i =-\frac{i(i-1)(2D-3)^2(2D-i)(2D-i-1)}{4(D-1)^2(2D-2i+1)(2D-2i-1)} , \\
& v_i = \frac{2i(i-1)(2D-i)(2D-i-1)}{(D-1)^2(2D-2i+1)(2D-2i-1)} \\
& \quad \times \frac{i^2(8D^2-16D+10)-i(16D^3-32D^2+20D)+16D^4-48D^3+50D^2-17D+2}{2i^2-4iD+4D^2-D},
\\
&w_i = - \frac{4i(i-1)(2D-1)(2D-i)(2D-i-1)}{(D-1)^2(2D-2i-1)(2D-2i+1)} \\
& \quad \times \frac{i^2(4D-2)-i(8D^2-4D)+8D^3-18D^2+17D-4}{2i^2-4iD+4D^2-D},
\\
&\Phi_i(\lambda) = u_i (\lambda-\xi)(\lambda-\xi_i), \qquad \qquad  \xi=4, \\
& \xi_i =\frac{4(2D-1)\bigl( i^2(4D-2)-i(8D^2-4D)+8D^3-18D^2+17D-4\bigr)}{(2D-3)^2(2i^2-4iD+4D^2-D)}.
\end{align*}
\end{example}

\begin{lemma} \label{ex:fhc2} For the graph   $\frac{1}{2}\tilde H(4D,2)$      the kite function $\zeta_i$ is constant for $2 \leq i \leq D$. Moreover
\begin{align*}
z_i =4(i-1) \qquad \qquad (2 \leq i \leq D).
\end{align*}
\end{lemma}
\begin{proof} By combinatorial counting.
\end{proof}

\begin{lemma} We refer to Example \ref{ex:fhc} and write $E=E_1$. 
The set $\lbrace E{\hat x} \vert x \in X\rbrace$ is Norton-balanced.  
Pick distinct   $x,y \in X$ and write $i = \partial(x,y)$. For $2 \leq i \leq D-1$,
\begin{align}
0 &= E{x^-_y} +   \frac{i-1}{2D-i-1} E{x^+_y} + 2(i-1)(3-2D) E{\hat x} +\frac{2(D-1)}{i+1-2D} E{\hat y}. \label{eq:fhc1}
\end{align}
\noindent For $i=D$,
\begin{align}
0 &= E{x^-_y} +   2(D-1)(3-2D) E{\hat x} -2 E{\hat y}. \label{eq:fhc2}
\end{align}
\end{lemma}
\begin{proof}  To get the first assertion, we use Proposition \ref{cor:NBPhi2}(ii).
Pick an integer $i$ $(2 \leq i \leq D-1)$. We verify the conditions in \eqref{eq:2conditions}. 
We have $z_2=\xi$, so $\Phi_i(z_2)=0$. We have
\begin{align*}
\lambda_i =  \frac{i-1}{i+1-2D} ,     \qquad \qquad \frac{\theta^*_{i}-\theta^*_{i+1}}{\theta^*_{i-1}-\theta^*_i} = \frac{2D-2i-1}{2D-2i+1}.
\end{align*}
Therefore
\begin{align*}
\frac{\theta^*_{i}-\theta^*_{i+1}}{\theta^*_{i-1}-\theta^*_i} -\lambda_i= \frac{2(2D-3)(D-i)}{(2D-i-1)(2D-2i+1)}\not=0.
\end{align*}
We have verified the conditions in  \eqref{eq:2conditions}, so the set $\lbrace E{\hat x} \vert x \in X\rbrace$ is Norton-balanced.
The linear dependence  \eqref{eq:fhc1} is obtained using  \eqref{eq:lamiR}, \eqref{lem:lambdaCALC}.
To obtain  \eqref{eq:fhc2}, use Proposition \ref{prop:AT}
and $z^-_D=z_D$.
\end{proof}

\begin{example}\label{ex:fh3} \rm  (See \cite[Chapter~6.4]{bbit}, \cite[Example~6.1(22)]{tSub3}.)
The folded-half graph $\frac{1}{2}\tilde H(4D+2,2)$ is distance-regular, with diameter $D$ and intersection numbers
\begin{align*}
&c_i = i (2i-1)  \qquad (1 \leq i \leq D), \\
&b_i = (2D-i+1)(4D-2i+1) \qquad \quad  (0 \leq i \leq D-1).
\end{align*}
\end{example}

\begin{example}\label{ex:fh4} \rm  The graph  $\frac{1}{2}\tilde H(4D+2,2)$ has a $Q$-polynomial structure such that
\begin{align*}
&\theta_i = \theta^*_i = (2D+1)(4D+1)-8i(2D-i+1) \qquad \qquad (0 \leq i \leq D).
\end{align*}
This $Q$-polynomial structure has Racah type with
\begin{align*}
&r_1=-D-3/2,  \qquad   r_2 = -2D-3/2,    \qquad  s= s^* = -2D-2, \qquad  h=h^*=8.
\end{align*}
This $Q$-polynomial structure is DC with $\gamma^*=16$.
\medskip

\noindent For $2 \leq i \leq D$,
\begin{align*}
&\alpha_i = \frac{(i-1)(D-1)(2D-i+1)}{(2D-1)(D-i+1)},  \\
&\beta_i = -\frac{(i-1)(i-2)}{2(2D-1)(D-i+1)}, \\
&r_i = \frac{(i-1)(2i-1)\bigl(4i^2-8i(D+1)+8D^2+6D+1\bigr)}{4i^2-4i(2D+1)+8D^2+6D+1},\\
& s_i = \frac{(2i-1)(8D^2+6D+1-8iD)}{4i^2-4i(2D+1)+8D^2+6D+1},  \\
& z^-_i =\frac{4(i-1)\bigl(8D^3+6D^2+D-i(16D^2+10D)+i^2(12D+1)-2i^3\bigr)}{(D-1)\bigl(4i^2-4i(2D+1)+8D^2+6D+1\bigr)}.
\end{align*}
\noindent For $1 \leq i \leq D-1$,
\begin{align*}
& R_i = \frac{(2D-i)(4D-2i+1)(4i^2-8iD+8D^2-2D-3)}{4i^2-4i(2D+1)+8D^2+6D+1},      \\
& S_i = - \frac{(4D-2i+1)(8D^2+2D-1-8iD)}{  4i^2-4i(2D+1)+8D^2+6D+1},           \\
& z^+_{i+1} =\frac{4i}{D-1}\, \frac{2i^3-i^2(4D+5)+4i(2D+1)+8D^3-6D^2-5D-1}{4i^2-4i(2D+1)+8D^2+6D+1}.
\end{align*}
\noindent For $2 \leq i \leq D-1$,
\begin{align*}
&u_i =-\frac{i(i-1)(D-1)^2(2D-i)(2D-i+1)}{(2D-1)^2(D-i)(D-i+1)}, \\
& v_i = \frac{4i(i-1)(2D-i)(2D-i+1)}{(2D-1)^2(D-i)(D-i+1)} \\
& \quad \times \frac{16D^4-16D^3+2D^2+5D+1-4i(4D^3-2D^2+1)+4i^2(2D^2-2D+1)}{4i^2-4i(2D+1)+8D^2+6D+1},
\\
&w_i = -\frac{16Di(i-1)(2D-i)(2D-i+1)}{(2D-1)^2(D-i)(D-i+1)} \\
&\quad \times  \frac{4i^2D-4iD(2D+1)+8D^3-6D^2+5D+1}{4i^2-4i(2D+1)+8D^2+6D+1},
\\
&\Phi_i(\lambda) = u_i (\lambda-\xi)(\lambda-\xi_i), \qquad \qquad  \xi=4, \\
& \xi_i =\frac{4D}{(D-1)^2}\, \frac{4i^2 D-4iD(2D+1)+8D^3-6D^2+5D+1}{4i^2-4i(2D+1)+8D^2+6D+1}.
\end{align*}
\end{example}

\begin{lemma} For the graph   $\frac{1}{2}\tilde H(4D+2,2)$      the kite function $\zeta_i$ is constant for $2 \leq i \leq D$. Moreover
\begin{align*}
z_i =4(i-1) \qquad \qquad (2 \leq i \leq D).
\end{align*}
\end{lemma}
\begin{proof} By combinatorial counting.
\end{proof}

\begin{lemma} We refer to Example \ref{ex:fh4} and write $E=E_1$. 
The set $\lbrace E{\hat x} \vert x \in X\rbrace$ is Norton-balanced.  
For $2 \leq i \leq D-1$ and   $x,y \in X$ at distance $i = \partial(x,y)$,
\begin{align} \label{eq:LinDep}
0 &= E{x^-_y} +   \frac{i-1}{2D-i} E{x^+_y} -4(D-1)(i-1) E{\hat x} +\frac{2D-1}{i-2D} E{\hat y}.
\end{align}
\end{lemma}
\begin{proof}  To get the first assertion, we use Proposition \ref{cor:NBPhi2}(ii).
Pick an integer $i$ $(2 \leq i \leq D-1)$. We verify the conditions in \eqref{eq:2conditions}. 
We have $z_2=\xi$, so $\Phi_i(z_2)=0$. We have
\begin{align*}
\lambda_i = \frac{i-1}{i-2D}, \qquad \qquad \frac{\theta^*_{i}-\theta^*_{i+1}}{\theta^*_{i-1}-\theta^*_i} = \frac{D-i}{D-i+1}.
\end{align*}
Therefore
\begin{align*}
\frac{\theta^*_{i}-\theta^*_{i+1}}{\theta^*_{i-1}-\theta^*_i} -\lambda_i= \frac{(D-1)(2D-2i+1)}{(2D-i)(D-i+1)} \not=0.
\end{align*}
We have verified the conditions in  \eqref{eq:2conditions}, so the set $\lbrace E{\hat x} \vert x \in X\rbrace$ is Norton-balanced.
The linear dependence \eqref{eq:LinDep} is obtained using  \eqref{eq:lamiR}, \eqref{lem:lambdaCALC}.
\end{proof}

\section{Example: the Hermitean forms graph}
\begin{example} \label{ex:hermite} \rm (See \cite[Chapter~6.4]{bbit}, \cite[Note~6.2]{tSub3}.)
let $GF(p^n)$ denote a finite field.
 The Hermitean forms graph $Her_D(p^n)$ is distance-regular
with diameter $D$ and intersection numbers 
\begin{align*}
c_i = q^{i-1} \frac{q^i-1}{q-1}, \qquad \qquad
b_i = - \frac{q^{2 D}-q^{2i}}{q-1} \qquad \qquad (0 \leq i \leq D),
\end{align*}
where $q=-p^n$.
\end{example}

\begin{example}\label{ex:her} \rm 
The graph $Her_D(p^n)$ has a $Q$-polynomial structure such that
\begin{align*}
\theta_i = \theta^*_i = - \frac{q^{2D-i}-1}{q-1} \qquad \qquad (0 \leq i \leq D). 
\end{align*}
This $Q$-polynomial structure has affine $q$-Krawtchouk type, with
\begin{align*}
&r = -q^{-D-1}, \qquad \qquad h = -\frac{q^{2D}}{q-1}, \qquad \qquad  h^* = -\frac{q^{2D}}{q-1}.
\end{align*}
This $Q$-polynomial structure is DC iff $a^*_1=0$ iff $a_1=0$ iff $p^n=2$ iff $q=-2$ (provided that $D\geq 4$). Assume that $q=-2$. We have
\begin{align*}
\gamma^* = q^{-1}-1 \not=0.
\end{align*}
\noindent For $2 \leq i \leq D$,
\begin{align*}
&\alpha_i = \frac{q^{i-1}-1}{q-1}, \qquad \qquad \beta_i = 0, \\
&r_i = q^{i-1} \frac{q^{i-1}-1}{q-1}\, \frac{q^{2D+i}+q^{2D+1}+q^i}{q^{2D+i}+q^{2D}-2q^i}, \\
& s_i = q^{2i-2} \frac{q^{2D+1}+q^{2D}-q^i-q}{q^{2D+i}+q^{2D}-2q^i}, \\
& z^-_i = (q^{i-1}-1) \frac{q^{2D+1}+q^{2D}+q^{2i}}{q^{2D+i}+q^{2D}-2q^i}.
\end{align*}
\noindent For $1 \leq i \leq D-1$,
\begin{align*}
& R_i = - \frac{1}{q}\, \frac{q^{2D}-q^{2i}}{q-1}\, \frac{q^{2D+i}+q^{2D}+q^i}{q^{2D+i} +q^{2D}-2q^i}, \\
& S_i = q^{i-1} \frac{ q^{2D}-q^{2i}}{q^{2D+i}+q^{2D}-2q^i},   \\
& z^+_{i+1} = - \frac{q^i(q^i-1)(q^i-q)}{q^{2D+i}+q^{2D}-2q^i}.
\end{align*}
\noindent For $2 \leq i \leq D-1$,
\begin{align*}
&u_i =  - \frac{q^i-1}{q-1}\, \frac{q^{i-1}-1}{q-1}, \\
& v_i = \frac{ (q^i-1)(q^{i-1}-1)}{q-1}\, \frac{q^{2D+1}+q^{2D}-2q^i}{q^{2D+i}+q^{2D}-2q^i}  \qquad \qquad   
w_i = 0,   \\
&\Phi_i(\lambda) = u_i (\lambda-\xi)(\lambda-\xi_i), \qquad \xi=0, \qquad \xi_i = (q-1)\frac{q^{2D+1}+q^{2D}-2q^i}{q^{2D+i}+q^{2D}-2q^i}.
\end{align*}
\end{example}

\begin{lemma} For $Her_D(p^n)$ the kite function $\zeta_i$ is constant for $2 \leq i \leq D$. Moreover
\begin{align*}
z_i = 0 \qquad \qquad (2 \leq i \leq D).
\end{align*}
\end{lemma}
\begin{proof} The graph $Her_D(p^n)$ is kite-free by \cite[Theorem~2.12]{kiteFree}.  
\end{proof}

\begin{lemma} We refer to Example \ref{ex:her} with $q=-2$. Write $E=E_1$. 
The set $\lbrace E{\hat x} \vert x \in X\rbrace$ is not Norton-balanced. However  the following linear dependencies hold.
Pick distinct $x,y \in X$ and write $i=\partial(x,y)$. For $1 \leq i \leq D-1$,
\begin{align} \label{eq:Hermstep1}
0 &= E{x^-_y} -q^{-1} E{x^+_y} +\frac{1}{q^2}\, \frac{q^i-q^{2D}}{q-1}  E{\hat x} +q^{i-2} E{\hat y}.
\end{align}
For $i=D$,
\begin{align}\label{eq:Hermstep2}
0 = E{x^-_y} - q^{D-2} \frac{q^D-1}{q-1} E{\hat x} + q^{D-2} E{\hat y}.
\end{align}
\end{lemma}
\begin{proof} The first assertion follows from Lemma \ref{lem:q111}(i).
To obtain \eqref{eq:Hermstep1}, use  \eqref{eq:lamiR}, \eqref{lem:lambdaCALC} and
\begin{align*}
\lambda_i = q^{-1} \qquad \qquad (2 \leq i \leq D-1).
\end{align*}
To obtain \eqref{eq:Hermstep2}, use Proposition \ref{prop:AT} and $z^-_D = 0 = z_D$.  Alternatively, \eqref{eq:Hermstep1} and  \eqref{eq:Hermstep2}  follow from Lemma \ref{lem:q111}(ii).
\end{proof}

\section{Example: the Doob graphs}
In this section, we will discuss the Doob graphs and their relationship to the Hamming graphs. In this discussion
the Shrikhande graph makes an appearance. The Shrikhande graph is distance-regular with diameter $2$; it has the same intersection numbers as the Hamming graph $H(2,4)$. The Shrikhande graph is not isomorphic to $H(2,4)$,
 because the Shrikhande graph has a $2$-kite and $H(2,4)$ does not. See \cite[Example~2.10]{bbit} for more information about the Shrikhande graph.

\begin{example} \rm (See \cite[Chapter~6.4]{bbit}, \cite[p.~262]{bcn}.) \rm 
 By a {\it Doob graph}, we mean  a Cartesian product of graphs, with each factor  isomorphic to the Shrikhande graph or the complete graph $K_4$. We require that in the Cartesian product,
 at least one factor is isomorphic to the Shrikhande graph. 
 Let $\Gamma$ denote a Doob graph with diameter $D$.
 We have $D=2n+m$, where $n$ (resp. $m$) is the number of factors isomorphic to the Shrikhande graph (resp. $K_4$).
The graph $\Gamma$  is distance-regular and has the same intersection numbers
as $H(D,4)$. However $\Gamma$ is not isomorphic to $H(D,4)$. Both $H(D,4)$ and
$\Gamma$ have a $Q$-polynomial structure such that
\begin{align*}
\theta_i = \theta^*_i = 3D-4i \qquad \qquad (0 \leq i \leq D).
\end{align*}
Every assertion about $H(D,4)$ in Example \ref{ex:hypercube2} holds for $\Gamma$. In particular, the $Q$-polynomial structure for $\Gamma$ is DC and $\gamma^*=0$. Moreover
\begin{align}
z^-_i = 0, \qquad \qquad z^+_{i+1}= \frac{4i}{3D-2i} \qquad \qquad (2 \leq i \leq D-1). \label{eq:ziplus}
\end{align}
\end{example}

\begin{lemma} Assume that $\Gamma=(X,\mathcal R)$ is a  Doob graph, and write $E=E_1$. Then the set $\lbrace E{\hat x} | x \in X\rbrace$ is not Norton-balanced.
\end{lemma}
\begin{proof} We assume that the set 
$\lbrace E{\hat x} | x \in X\rbrace$ 
is
 Norton-balanced, and get a contradiction. 
 There exists a subset $S \subseteq X$ such that (i) the subgraph of $\Gamma$ induced on $S$ is isomorphic to the Shrikhande graph;
 (ii)  
 for all $x,y\in S$ and all $z \in X\backslash S$, $\partial(x,z)+\partial(y,z) \geq \partial(x,y) +2$.
  Since the Shrikhande graph has a $2$-kite, there exist $x,y \in S$ at distance $\partial(x,y)=2$ such that $\Gamma(x) \cap \Gamma(y)$
 contains an edge. By Definition \ref{def:ziBar}, $\zeta_2(x,y,\ast)=1$. By Definition \ref{def:Pave} and the construction,
 $\zeta_3(\ast,y,x) = 0$. We have $\zeta_2(x,y,\ast) \not=z^-_2$,  so $Ex^-_y \not=r_2 E{\hat x} + s_2 E{\hat y}$ by Proposition \ref{prop:AT}.
 We have $\zeta_3(\ast, y,x) \not=z^+_3$, so  $Ex^+_y \not= R_2 E{\hat x} + S_2 E{\hat y}$ by Proposition \ref{prop:PAT}. 
 The vectors $Ex^-_y, Ex^+_y, E{\hat x}, E{\hat y}$ are linearly dependent by Lemma \ref{lem:NBLD}.
 This contradicts Lemma \ref{cor:situation}, and the result follows.
 \end{proof}

\section{Further examples}

Below we list some $Q$-polynomial distance-regular graphs $\Gamma=(X,\mathcal R)$ with diameter $D\geq 4$. We describe
the $Q$-polynomial structure of $\Gamma$, using the data in \cite[Chapter~6.4]{bbit} and the notation of \cite[Section~20]{LSnotes}.  In each case, 
the $Q$-polynomial structure is not DC because the condition in Theorem \ref{thm:main}(i) is violated.
  In each case, $\gamma^*\not=0$. In each case, the set $\lbrace E {\hat x} | x \in X\rbrace$ is not Norton-balanced,
where $E$ is the $Q$-polynomial primitive idempotent of $\Gamma$ attached to the given $Q$-polynomial structure. For the cases (i)--(iv) this Norton-balanced assertion follows from Lemma \ref{lem:RNBDC}, because $\Gamma$ is distance-transitive
and therefore reinforced. For case (vi) the assertion follows from Corollary \ref{cor:WhenZ5} and the fact that $\Gamma$ has a non-regular $\mu$-graph \cite{twist, twist2}. For case (v) the assertion is proved in Lemma  \ref{lem:finalRes} below.
Let $GF(q)$ denote a finite field.
\begin{enumerate}
\item[\rm (i)]  The folded graph $\tilde J(4D,2D)$ has Racah type with
\begin{align*}
r_1 = -D- 1/2, \qquad r_2 = -2D-1, \qquad s=-2D-3/2, \qquad s^*=-2D-1.
\end{align*}
\item[\rm (ii)] The folded graph $\tilde J(4D+2,2D+1)$ has Racah type with
\begin{align*}
r_1 = -D- 3/2, \qquad r_2 = -2D-2, \qquad  s=-2D-5/2, \qquad s^*=-2D-2.
\end{align*}
\item[\rm (iii)]  The bilinear forms graph $ H_q(D, N)$ $(N \geq  D)$  has affine $q$-Krawtchouk type with
\begin{align*}
 r = q^{-N-1}.
\end{align*}

\item[\rm (iv)]  The alternating forms graph $Alt_q(N) $ $(D = \lfloor N/2 \rfloor)$ has affine $q$-Krawtchouk type with
\begin{align*}
 r = q^{-D-\frac{1}{2}}\; \mbox{\rm (if $N$ is even)}, \qquad  \quad r = q^{-D-\frac{3}{2}} \; \mbox{\rm (if $N$ is odd)}.
\end{align*}
\item[\rm (v)]  The quadratic forms graph
$Quad_q(N)$ ($D = \lfloor (N+1)/2\rfloor $) has affine $q$-Krawtchouk type with
\begin{align*}
  r = q^{-D-\frac{3}{2}}\; \mbox{\rm (if $N$ is even)}, \qquad \quad  r = q^{-D-\frac{1}{2}} \; \mbox{\rm (if $N$ is odd)}.
\end{align*}
\item[\rm (vi)]  The twisted Grassmann graph ${}^2J_q(2D+1,D)$ has dual $q$-Hahn type with
\begin{align*}
r= q^{-D-2}, \qquad \qquad s = q^{-2D-3}.
\end{align*}
This graph has the same intersection numbers as $J_q(2D+1,D)$.
\end{enumerate}

\noindent For the rest of this section, our goal is to show that the graph from item (v)  is not Norton-balanced. To reach the goal, we will derive some preliminary results that apply
to more general $Q$-polynomial distance-regular graphs. We will be discussing Lemma  \ref{lem:WhenZ3}, which involves two vertices $x,y$ and a scalar $\lambda_i$.  Going forward, the
scalar $\lambda_i$  will be denoted by $\lambda_i(x,y)$ in order to emphasize that it might depend on $x,y$ as well as $i$.

\begin{assumption}\label{assume} \rm Let $\Gamma=(X, \mathcal R)$ denote a  $Q$-polynomial distance-regular graph with diameter $D\geq 4$. Let $E$ denote a $Q$-polynomial primitive idempotent of $\Gamma$ with $\gamma^*\not=0$.
Assume that the set $\lbrace E{\hat x} \vert x \in X\rbrace$ is Norton-balanced.
\end{assumption}

\begin{lemma}\label{lem:SS1} With reference to Assumption \ref{assume},
pick an integer $i$ $(3 \leq i \leq D-1)$ and  $x,y,z \in X$ such that
\begin{align*}
\partial(x,y)=i, \qquad \quad \partial (x,z) = 1, \qquad \quad \partial(y,z)=i-1.
\end{align*}
Then
\begin{align*}
\zeta_i(x,y,z) &= \zeta_i(x,y,\ast) = z^-_i +  \lambda_i(x,y) \, \frac{\gamma^* b_i}{\theta^*_i + \theta^*_0} \, \frac{\theta^*_i - \theta^*_1}{\theta^*_1 - \theta^*_2} \\
&=  \zeta_i(\ast,y,z)  = z^+_i - \frac{1}{ \lambda_{i-1} (z,y)}\, \frac{\gamma^* c_{i-1}}{ \theta^*_{i-1}+\theta^*_0} \, \frac{\theta^*_{i-1} - \theta^*_1}{\theta^*_1-\theta^*_2}.
\end{align*}
\end{lemma}
\begin{proof} To get the first two equalities, apply Lemmas \ref{lem:NBLD}, \ref{lem:WhenZ3}, \ref{lem:WhenZ4}(i) to $x,y$.
 To get the last two equalities, apply Lemmas \ref{lem:NBLD}, \ref{lem:WhenZ3} and \ref{lem:WhenZ4}(ii)  to $z,y$. 
\end{proof}

\begin{definition} \label{def:RL} \rm With reference to Assumption \ref{assume},
pick $y \in X$ and an integer $n\geq 0$.
A path $\lbrace x_i \rbrace_{i=0}^n$  in $\Gamma$ is called {\it raising/lowering with respect to $y$} whenever the following {\rm (i), (ii)} hold:
\begin{enumerate}
\item[\rm (i)] $2 \leq \partial(x_i, y) \leq D-1$ for $0 \leq i \leq n$;
\item[\rm (ii)]
$\partial(x_{i-1}, y) \not= \partial(x_i,y)$ for $1 \leq i \leq n$.
\end{enumerate}
\end{definition}

\begin{lemma} \label{lem:raiseLower} With reference to Assumption \ref{assume}, pick $y \in X$. Pick an integer $i$ $(2 \leq i \leq D-1)$ and $x,x' \in \Gamma_i(y)$. Assume that $x,x'$ are connected
by a path that is raising/lowering with respect to $y$. Then $\lambda_i(x,y)= \lambda_i(x',y)$.
\end{lemma}
\begin{proof} Routine using Lemma \ref{lem:SS1} and Definition \ref{def:RL}.
\end{proof}

\begin{lemma} \label{lem:3R}  With reference to Assumption \ref{assume}, pick $y \in X$. Pick an integer $i$ $(2 \leq i \leq D-1)$ and adjacent $x,x' \in \Gamma_i(y)$.
Then
\begin{align*}
&\vert \Gamma(x) \cap \Gamma(x') \cap \Gamma_{i-1}(y) \vert + \frac{ c_i \theta^*_2 - r_i \theta^*_1 - s_i \theta^*_i}{\theta^*_1 - \theta^*_2} \\
&= \lambda_i(x,y) \biggl(
\vert \Gamma(x) \cap \Gamma(x') \cap \Gamma_{i+1}(y) \vert + \frac{b_i \theta^*_2 -R_i \theta^*_1 - S_i \theta^*_i}{\theta^*_1 - \theta^*_2} \biggr) \\
&= \lambda_i(x',y) \biggl(
\vert \Gamma(x) \cap \Gamma(x') \cap \Gamma_{i+1}(y) \vert + \frac{b_i \theta^*_2 -R_i \theta^*_1 - S_i \theta^*_i}{\theta^*_1 - \theta^*_2} \biggr).
\end{align*}
\end{lemma}
\begin{proof} To get the first equality, take the inner product of $E{\hat x'}$ with each side of  \eqref{eq:lambdaEQ}, and evaluate the result using Lemma \ref{lem:geometry}(i).
To get the second equality, apply the first equality with $x,x'$ interchanged.
\end{proof}

\begin{lemma} \label{lem:4P} With reference to Assumption \ref{assume}, pick $y \in X$. Pick an integer $i$ $(2 \leq i \leq D-1)$ and adjacent $x,x' \in \Gamma_i(y)$ such that
$\lambda_i(x,y) \not=\lambda_i(x',y)$. Then the following {\rm (i)--(iv)} hold:
\begin{enumerate}
\item[\rm (i)] $\displaystyle \vert \Gamma(x) \cap \Gamma(x') \cap \Gamma_{i-1}(y) \vert + \frac{ c_i \theta^*_2 - r_i \theta^*_1 - s_i \theta^*_i}{\theta^*_1 - \theta^*_2}=0$;
\item[\rm (ii)] $\displaystyle \vert \Gamma(x) \cap \Gamma(x') \cap \Gamma_{i+1}(y) \vert + \frac{b_i \theta^*_2 -R_i \theta^*_1 - S_i \theta^*_i}{\theta^*_1 - \theta^*_2}=0$;
\item[\rm (iii)] $\Gamma(x) \cap \Gamma(x') \cap \Gamma_{i-1}(y) = \emptyset$ if  $3 \leq i \leq D-1$;
\item[\rm (iv)] $ \Gamma(x) \cap \Gamma(x') \cap \Gamma_{i+1}(y) = \emptyset$ if  $2 \leq i \leq D-2$.
\end{enumerate}
\end{lemma}
\begin{proof} (i), (ii) By Lemma  \ref{lem:3R}. \\
\noindent (iii) Assume that $3 \leq i \leq D-1$ and $\Gamma(x) \cap \Gamma(x') \cap \Gamma_{i-1}(y) \not= \emptyset$. There exists $z \in \Gamma(x) \cap \Gamma(x') \cap \Gamma_{i-1}(y)$.
The sequence $x, z, x'$ is a path in $\Gamma$ that is lowering/raising with respect to $y$, forcing $\lambda_i(x,y) =\lambda_i(x',y)$ by Lemma  \ref{lem:raiseLower}. This is a contradiction. \\
\noindent (iv) Assume that $2 \leq i \leq D-2$ and $\Gamma(x) \cap \Gamma(x') \cap \Gamma_{i+1}(y) \not= \emptyset$. There exists $z \in \Gamma(x) \cap \Gamma(x') \cap \Gamma_{i+1}(y)$.
The sequence $x, z, x'$ is a path in $\Gamma$ that is lowering/raising with respect to $y$, forcing $\lambda_i(x,y) =\lambda_i(x',y)$ by Lemma  \ref{lem:raiseLower}. This is a contradiction.
\end{proof}

\noindent We return our attention to  the graph $\Gamma$ from item {\rm (v)} above.
\begin{lemma} \label{lem:finalRes} Assume that $\Gamma$ is from item {\rm (v)}, with $D\geq 4$. Then the set $\lbrace E{\hat x} \vert x \in X\rbrace$ is not Norton-balanced.
\end{lemma}
\begin{proof} We assume that the set  $\lbrace E{\hat x} \vert x \in X\rbrace$ is  Norton-balanced, and get a contradiction. To obtain the contradiction, we show that the kite function $\zeta_i$ is constant for $2 \leq i \leq D$.
Until further notice, fix $y \in X$. Our first step is to show that $\lambda_2(x,y)$ is independent of $x$ for all $x \in \Gamma_2(y)$. To this end, we define a set of vertices $\Delta=\Delta(y)$ by
  $\Delta= \cup_{i=0}^2 \Gamma_i(y)$. We consider the subgraph of $\Gamma$ induced on $\Delta$. In the subgraph $\Delta$, each vertex is connected to $y$
by a path of length at most 2.
Therefore the subgraph $\Delta$ is connected. One checks that  the subgraph $\Delta$ has diameter 4.
Let $\partial_\Delta$ denote the distance function for the subgraph $\Delta$.
Suppose that there exists a pair of vertices $x,x' \in \Gamma_2(y)$ such that
$\lambda_2(x,y) \not=\lambda_2(x',y)$. Of all such pairs of vertices, choose a pair $x,x'$ such that $\partial_\Delta(x,x') $ is minimal. By construction  $1 \leq \partial_\Delta(x,x') \leq 4$.
We now examine the cases. \\
\noindent {\bf Case $\partial_\Delta(x,x')=1$}. We have $\partial(x,x')=1$. Setting $i=2$ in Lemma \ref{lem:4P}(ii),(iv) we obtain $\Gamma(x) \cap \Gamma(x') \cap \Gamma_3(y) = \emptyset$ and
\begin{align*}
b_2 \theta^*_2 -R_2\theta^*_1 - S_2\theta^*_2=0.
\end{align*}
Using the data in \cite[Example~20.6]{LSnotes}, we obtain for $N$ even:
\begin{align*}
b_2 \theta^*_2 -R_2\theta^*_1 - S_2\theta^*_2 =  -\frac{ (q^{D + \frac{1}{2}} + q^D - 1) ( q^D-q^2) (q^{D+\frac{1}{2}}  - q^2)   q^{2D+\frac{1}{2}} }{q^{2D + 2+\frac{1}{2}} +q^{2D+\frac{1}{2}} - 2 q^{D + 2+ \frac{1}{2}} - 2q^{D+ 2} + 2q^2}
\not=0,
\end{align*}
and for $N$ odd:
\begin{align*}
b_2 \theta^*_2 -R_2\theta^*_1 - S_2\theta^*_2 =  -\frac{ (q^{D - \frac{1}{2}} + q^D - 1) ( q^D-q^2) (q^{D-\frac{1}{2}}  - q^2)   q^{2D-\frac{1}{2}} }{q^{2D + 2-\frac{1}{2}} +q^{2D-\frac{1}{2}} - 2 q^{D + 2- \frac{1}{2}} - 2q^{D+ 2} + 2q^2}
\not=0.
\end{align*}
This is a contradiction. \\
\noindent {\bf Case $\partial_\Delta(x,x')=2$}. We have $\partial(x,x')=2$. By the triangle inequality,
\begin{align*}
\Gamma(x) \cap \Gamma(x') = \cup_{i=1}^3 \Bigl(\Gamma(x) \cap \Gamma(x') \cap \Gamma_i(y)\Bigr).
\end{align*}
The set $\Gamma(x) \cap \Gamma(x') \cap \Gamma_{3}(y)$ must be empty,  because if it contains a vertex $z$  then
the sequence $x, z, x'$ is a path in $\Gamma$ that is raising/lowering with respect to $y$, contradicting Lemma  \ref{lem:raiseLower}.  The set $\Gamma(x) \cap \Gamma(x') \cap \Gamma_{2}(y)$ must be empty,
because if it contains a vertex $z$ then $x,z,x'$ is a path in the subgraph $ \Delta$, forcing  $\lambda_2(x,y)=\lambda_2(z,y)= \lambda_2(x',y)$ by the minimality of $\partial_\Delta(x,x')$.
By the above comments,
\begin{align*}
\Gamma(x) \cap \Gamma(x') =  \Gamma(x) \cap \Gamma(x') \cap \Gamma(y).
\end{align*}
Note that
\begin{align*}
\vert \Gamma(x) \cap \Gamma(x') \cap \Gamma(y) \vert =\vert \Gamma(x) \cap \Gamma(x') \vert = c_2.
\end{align*}
We have
\begin{align*}
\Gamma(x) \cap \Gamma(y) = \Gamma(x') \cap \Gamma(y) =  \Gamma(x) \cap \Gamma(x') \cap \Gamma(y),
\end{align*}
because the first two sets have cardinality $c_2$ and contain the third set.
By Lemma \ref{lem:subraphVal}, the scalar $\zeta_2(x,y,\ast)$ is the average valency of the induced subgraph $\Gamma(x) \cap \Gamma(y)$. Similarly,
 $\zeta_2(x',y,\ast)$ is the average valency of the induced subgraph $\Gamma(x') \cap \Gamma(y)$. These subgraphs coincide, so
$\zeta_2(x,y,\ast) = \zeta_2(x',y,\ast)$. By this and Lemma \ref{lem:WhenZ3} (with $i=2$), we obtain $\lambda_2(x,y) = \lambda_2(x',y)$. This is a contradiction. \\
\noindent {\bf Case $\partial_\Delta(x,x')=3$}. In the subgraph $\Delta$, the vertices $x,x'$ are connected by a path of length 3. Denote such a path by $x,z,z',x'$.
The vertices $x,z'$ are not adjacent; otherwise $x,z',x'$ is a path in $\Delta$, contradicting $\partial_\Delta(x,x')=3$. Similarly, the vertices $x',z$ are not adjacent.
The vertex $z$ is contained in $\Delta$ and adjacent to $x$, so  $z \in \Gamma(y) \cup \Gamma_2(y)$.
If $z \in \Gamma_2(y)$ then $\lambda_2(x,y) = \lambda_2(z,y) = \lambda_2(x',y)$ by the minimality of $\partial_\Delta(x,x')$. This is a contradiction, so
$z \in \Gamma(y)$. We have $z \in \Gamma(x) \cap \Gamma(y)$.
Similarly, $z' \in \Gamma(x')\cap \Gamma(y)$. These comments apply to every choice of $z,z'$. For each choice of $z$ there are $c_2$
choices for $z'$, and these are all contained in $\Gamma(x')\cap \Gamma(y)$. For each choice of $z'$ there are $c_2$ choices for $z$, and these are all contained in $\Gamma(x)\cap \Gamma(y)$.
We have $\vert \Gamma(x) \cap \Gamma(y)\vert = c_2$ and $\vert \Gamma(x')\cap \Gamma(y)\vert = c_2$.
By these comments, every vertex in $\Gamma(x) \cap \Gamma(y)$ is adjacent to every vertex in $\Gamma(x')\cap \Gamma(y)$. Pick $z \in \Gamma(x) \cap \Gamma(y)$ and
$w' \in \Gamma(x') \cap \Gamma_3(y)$. By construction $\partial(z,w') \in \lbrace 2,3\rbrace$. Suppose for the moment that $\partial(z,w')=2$. Then there exists $v \in \Gamma(z) \cap \Gamma(w')$.
By construction $v \in \Gamma_2(y)$. We have $\partial_\Delta(x,v) \leq 2$ since $x,z,v$ is a path in $\Delta$.
Therefore, $\lambda_2(x,y) =\lambda_2(v,y)$ by the minimality of $\partial_\Delta(x,x')$.
We have $\lambda_2(v,y) = \lambda_2(x',y)$ by  Lemma \ref{lem:raiseLower} and since $v,w',x'$ is a path in $\Gamma$ that is lowering/raising with respect to $y$.
By these comments,  $\lambda_2(x,y) =\lambda_2(v,y) = \lambda_2(x',y)$ for a contradiction. We have shown that $\partial(z,w') \not=2$, so $\partial(z,w')=3$.
It follows that in the graph $\Gamma$, every vertex in $\Gamma(x)\cap \Gamma(y)$
is at distance 3 from every vertex in $\Gamma(x') \cap \Gamma_3(y)$. Similarly,  in the graph $\Gamma$ every vertex in $\Gamma(x')\cap \Gamma(y)$
is at distance 3 from every vertex in $\Gamma(x) \cap \Gamma_3(y)$.
 Pick $z' \in \Gamma(x') \cap \Gamma(y)$. 
  Take the inner product of $E{\hat z'}$ with each side of  \eqref{eq:lambdaEQ}, and evaluate the result using Lemma \ref{lem:geometry}(i);
  this yields
 \begin{align*}
 c_2 \theta^*_1 - r_2 \theta^*_2 - s_2 \theta^*_1 = \lambda_2(x,y) \bigl( b_2 \theta^*_3 - R_2 \theta^*_2 - S_2 \theta^*_1\bigr).
 \end{align*}
 Interchanging the roles of $x,x'$ we obtain
  \begin{align*}
 c_2 \theta^*_1 - r_2 \theta^*_2 - s_2 \theta^*_1 = \lambda_2(x',y) \bigl( b_2 \theta^*_3 - R_2 \theta^*_2 - S_2 \theta^*_1\bigr).
 \end{align*}
 By these comments and $\lambda_2(x,y) \not=\lambda_2(x',y)$, we obtain
 \begin{align*}
 c_2 \theta^*_1 - r_2 \theta^*_2 - s_2 \theta^*_1=0, \qquad \qquad  b_2 \theta^*_3 - R_2 \theta^*_2 - S_2 \theta^*_1=0.
 \end{align*}
 \noindent Using the data in \cite[Example 20.6]{LSnotes}, we obtain
 \begin{align*}
  b_2 \theta^*_3 - R_2 \theta^*_2 - S_2 \theta^*_1 =  b_2 \theta^*_2 - R_2 \theta^*_1 - S_2 \theta^*_2 \not=0.
 \end{align*}
 This
 is a contradiction. \\
\noindent {\bf Case $\partial_\Delta(x,x')=4$}. Pick $z \in \Gamma(x) \cap \Gamma(y)$ and $z' \in \Gamma(x')\cap \Gamma(y)$.
Note that $\partial_\Delta(z,z')=2=\partial(z,z')$. Since $c_2 >1$, there exists $u \in \Gamma(z) \cap \Gamma(z')$ with $u \not=y$. 
By construction $u \in \Gamma(y) \cup \Gamma_2(y)$, so $u \in \Delta$. The sequence $x,z,u$ is a path in $\Delta$.
The vertices $x,u$ are not adjacent; otherwise $x,u,z',x'$ is a path in $\Delta$ of length 3. 
By these comments, $\partial_\Delta(x,u)=2$.
Similarly, $\partial_\Delta(x',u)=2$.
Suppose for the moment that $u \in \Gamma_2(y)$.
Then $\lambda_2(x,y) = \lambda_2(u,y) = \lambda_2(x',y)$ by the minimality of $\partial_\Delta(x,x')$. This is a contradiction, so $u \in \Gamma(y)$.
Note that $\vert  \Gamma(u) \cap \Gamma_2(y) \vert = b_1$, so  $\Gamma(u) \cap \Gamma_2(y)\not=\emptyset$.
Pick  $v \in \Gamma(u) \cap \Gamma_2(y)$. In the graph $\Delta$, the sequence $x,z,u,v$ is path, so $\partial_\Delta(x,v) \leq 3$.
Also in the graph $\Delta$, the sequence $x',z',u,v$ is path, so $\partial_\Delta(x',v) \leq 3$. Now
$\lambda_2(x,y) = \lambda_2(v,y) = \lambda_2(x',y)$ by the minimality of $\partial_\Delta(x,x')$. This is a contradiction.
\\
\noindent {\bf Conclusion}. We have shown that there does not exist a pair of vertices $x, x' \in \Gamma_2(y)$ such that $\lambda_2(x,y) \not=\lambda_2(x',y)$.
Consequently 
 $\lambda_2(x,y)$ is independent of $x$ for all $x \in \Gamma_2(y)$; we call this common value the $\lambda_2$-value of $y$.
 Until now, the vertex $y$ has been fixed. Next, we let $y$ vary. Pick any $x,y \in X$ at distance $\partial(x,y)=2$.
 We have $\zeta_2(x,y, \ast)=\zeta_2(y,x,\ast)$ because each side is equal to the average valency of the subgraph induced on  $\Gamma(x)\cap \Gamma(y)$.
 By this and Lemma   \ref{lem:WhenZ3} (with $i=2$) we obtain
 $\lambda_2(x,y) = \lambda_2(y,x)$.
 Therefore $x,y$ have the same $\lambda_2$-value.
 By this and since $\Gamma$ is not bipartite, we find that every vertex in $X$ has the same $\lambda_2$-value.
 This means that
 $\lambda_2(x,y)$ is independent of $x,y$ for all
$x,y \in X$ with $\partial(x,y)=2$.
By this and Lemma \ref{lem:SS1}, for $2 \leq i \leq D-1$
the scalar $\lambda_i(x,y)$ is independent of $x,y$ for all $x,y \in X$ with $\partial(x,y)=i$. By this and Lemmas \ref{lem:WhenZ3}, \ref{lem:WhenZ4}, the kite function $\zeta_i$ is constant for
$2 \leq i \leq D$. 
On one hand,  $\Gamma$ is reinforced by Lemma \ref{lem:2WhenReinforce}, so $E$ is DC by Lemma \ref{lem:RNBDC}.
On the other hand,  $E$ is not DC because the condition in Theorem \ref{thm:main}(i) is violated.
This is a contradiction, so  the set  $\lbrace E{\hat x} \vert x \in X\rbrace$ is not Norton-balanced.
\end{proof}


 \section{When $\Gamma$ affords a spin model}

We are done discussing the known infinite families of $Q$-polynomial distance-regular graphs with unbounded diameter. There is one more family 
of $Q$-polynomial distance-regular graphs that we would like to discuss; members of this family afford a spin model \cite{CW, C:thin,curtNom, CNhom, Jones, nomSpinModel, nomTerSM}. Very few examples are known;
see \cite[Section~9]{curtNom} or \cite[Section~15]{nomSpinModel}.
\medskip

\noindent
  Throughout this section, the following assumptions and notation are in effect.
  Let $\Gamma=(X,\mathcal R)$ denote a distance-regular graph with diameter $D \geq 3$.
  Assume that $\Gamma$ affords a spin model in the sense of \cite[Definition~11.1]{nomSpinModel}.
  By  \cite[Lemma~11.4]{nomSpinModel} there exists an ordering $\lbrace E_i \rbrace_{i=0}^D$ of
  the primitive idempotents of $\Gamma$ that is formally self-dual in the sense of  \cite[Definition~10.1]{nomSpinModel}. By \cite[Lemma~10.2]{nomSpinModel} the ordering
  $\lbrace E_i \rbrace_{i=0}^D$ is
  $Q$-polynomial; to avoid trivialities we assume that this ordering has $q$-Racah type \cite[Example~20.1]{LSnotes}. Write $E=E_1$. 
  \medskip
  
  \noindent We discuss some cases. 
   Throughout  this paragraph, assume that $\Gamma$ is bipartite or almost bipartite.
  The given $Q$-polynomial structure is formally self-dual, so $E$ is dual-bipartite or almost dual-bipartite. The set $\lbrace E{\hat x} \vert x \in X\rbrace$ is strongly balanced by Lemma \ref{lem:DBip}.
  The set $\lbrace E{\hat x} \vert x \in X\rbrace$ is Norton-balanced by Definitions \ref{def:STR}, \ref{def:NB}.  The kite function $\zeta_i$ is constant for $2 \leq i \leq D$, because $\Gamma$ has
  no kites. The graph $\Gamma$ is reinforced by Lemma \ref{lem:2WhenReinforce}, so  $E$ is DC in view of Lemma \ref{lem:RNBDC}.
  \medskip

 \noindent For the rest of this section, assume that $\Gamma$ is not bipartite and not almost bipartite. 
 By \cite[Remark~6.10]{CW} or  \cite[Remark~7.4]{nomTerSM}, \cite[Appendix~18]{nomTerSM} the parameters $q, r_1, r_2, s, s^*$ from \cite[Example~20.1]{LSnotes} satisfy
   \begin{align*}
   r_1 = - q^{-1} \eta, \qquad \qquad r_2 = - \eta^3 q^{D-2},  \qquad \qquad s=s^* = r^2_1 
   \end{align*}
   for an appropriate $\eta \in \mathbb C$. By  \cite[Lemma~12.2]{nomSpinModel} and \cite[Remark~7.4]{nomTerSM}, we have
   \begin{align*}
 &  q^i \not=1 \qquad (1 \leq i \leq D), \qquad \qquad q^i \eta^2 \not=1 \qquad (0 \leq i \leq 2D-2),\\
   & q^i \eta^3 \not=-1 \qquad (D-1 \leq i \leq 2D-2).
   \end{align*}
   The intersection numbers of $\Gamma$ are given in \cite[Theorems~6.6, 6.8]{CW} and \cite[Corollary~6.9]{CW}. By \cite[Corollary~6.9]{CW}
   and since $\Gamma$ is not bipartite, $q^{D-1} \eta^2 \not=-1$. By \cite[Corollary~6.9]{CW}
   and since $\Gamma$ is not almost bipartite, $q^D \eta \not=1$. 
  The primitive idempotent $E$ is $DC$ by Theorem \ref{thm:main}(i) and $s=r^2_1$.
  By \eqref{eq:gam} and \cite[Example~20.1]{LSnotes}, 
  \begin{align*}
  \gamma^* = \frac{(q-1)(q \eta^2-1)(q^{D-1} \eta^2+1)(q^D\eta-1)}{q \eta (q^D \eta^2-1)(q^{D-1}\eta+1)} \not=0.
  \end{align*}
  By  \cite[Remark~15.6]{nomTerSM} the kite function $\zeta_i$ is constant for $2 \leq i \leq D$. Therefore $\Gamma$ is reinforced.
  By  \cite[Remark~7.4]{nomTerSM} and \cite[Lemma~15.7]{nomTerSM},
   \begin{align}
   z_i = - \frac{(q^i-q)(q \eta^2-1)(q^{D-1} \eta^2+1)(q^D\eta-1)}{(q^i\eta-q)(q^D\eta^2-1)(q^{D-1} \eta+1)(q\eta-1)}
   \qquad \quad (2 \leq i \leq D). \label{eq:ziFORM}
   \end{align}
  
 \begin{lemma}  With the above notation, the set $\lbrace E{\hat x} \vert x \in X\rbrace$ is Norton-balanced.
Pick distinct $x,y \in X$ and write $i=\partial(x,y)$. There is a linear dependence with the following terms and coefficients:
  \begin{align*}
 0 = \qquad 
 \hbox{
\begin{tabular}[t]{c|c}
{\rm term}& {\rm coefficient} 
 \\
 \hline
 $E x^-_y$ & $1$\\
 $E x^+_y$ & $\frac{\eta (q^i-q)(q^i \eta-1)}{(q^i \eta -q)(q^i \eta^2-1)}$ \\
$E {\hat x} $ & $-\frac{q^{D-1}(\eta-1)(q\eta+1)(q\eta^2-1)(q^i-q) }{(q-1)(q^i\eta-q)(q^{D-1} \eta+1)(q^D\eta^2-1)}$\\
$E {\hat y}$ & $-\frac{q^i(\eta-1)(q\eta^2-1)}{(q^i\eta^2-1)(q^i\eta-q)}$
   \end{tabular}}
   \end{align*}
    \end{lemma}
\begin{proof} We first verify the linear dependence in the above table.
For $2 \leq i \leq D-1$ we compute the polynomial $\Phi_i(\lambda)$ using Definition \ref{def:Phi} along with
Lemma \ref{lem:kite}, Definitions \ref{def:ziMinus}, \ref{def:ziPlus} and the data in \cite[Example~20.1]{LSnotes}.
We check using \eqref{eq:ziFORM} that $\Phi_i(z_2)=0$. By this and Corollary \ref{cor:NBZ}, the vectors $Ex^-_y$, $Ex^+_y$, $E{\hat x}$, $E{\hat y}$ are linearly dependent.
The coefficients in this linear dependence are found using  \eqref{eq:lamiR}, \eqref{lem:lambdaCALC}.
This yields the linear dependence
in the above table, for $2 \leq i \leq D-1$. Next, assume that $i=D$.
Using Definition \ref{def:ziMinus} and \eqref{eq:ziFORM}, we obtain
$z^-_D=z_D$. This and Proposition \ref{prop:AT} yield the linear dependence in the above table for $i=D$. For $i=1$ the linear dependence in
the above table holds vacuously. We have shown that the linear dependence in the above table holds in every case.
Next, we verify that the set $\lbrace E{\hat x} \vert x \in X\rbrace$ is Norton-balanced. We will use Proposition \ref{cor:NBPhi2}(ii). We mentioned
earlier that $\gamma^*\not=0$ and $\Gamma$ is reinforced.
Pick an integer $i$ $(2 \leq i \leq D-1)$. We verify the conditions in \eqref{eq:2conditions}. The condition $\Phi_i(z_2)=0$ is already verified.
Referring to the above table, in our calculation of the $Ex^+_y$ coefficient we found that the scalar $\lambda_i$ from  \eqref{eq:lamiR}, \eqref{lem:lambdaCALC} is given by
\begin{align*}
\lambda_i = -\frac{\eta (q^i-q)(q^i \eta-1)}{(q^i \eta -q)(q^i \eta^2-1)}.
\end{align*}
By the data in \cite[Example~20.1]{LSnotes},
\begin{align*}
\frac{\theta^*_{i}-\theta^*_{i+1}}{\theta^*_{i-1}-\theta^*_i} = \frac{q^{2i} \eta^2-1}{q(q^{2i-2} \eta^2-1)}.
\end{align*}
We have
\begin{align*}
\frac{\theta^*_{i}-\theta^*_{i+1}}{\theta^*_{i-1}-\theta^*_i} -\lambda_i= \frac{(q \eta+1)(q^i \eta-1)(q^{2i-1} \eta^2-1)}{q(q^i \eta^2-1)(q^{2i-2} \eta^2-1)} \not=0.
\end{align*}
We have verified the conditions in  \eqref{eq:2conditions}, so the set $\lbrace E{\hat x} \vert x \in X\rbrace$ is Norton-balanced.
\end{proof}

\section{Directions for future research}


In this section, we give some suggestions for future research.

\begin{problem}\rm Classify up to isomorphism the $Q$-polynomial  distance-regular graphs  with diameter $D\geq 3$ that are Norton-balanced.
\end{problem}

\begin{problem}\rm Classify up to isomorphism  the distance-regular graphs with diameter $D\geq 4$ that have a $Q$-polynomial primitive idempotent that is DC.
\end{problem}
\begin{problem}\rm Classify up to isomorphism the distance-regular graphs with diameter $D\geq 3$ that have a $Q$-polynomial structure such that $\gamma^*=0$.
\end{problem}

\begin{conjecture} \rm Let $\Gamma=(X,\mathcal R)$ denote a $Q$-polynomial distance-regular graph with diameter $D\geq 3$. Let $E$ denote
a $Q$-polynomial primitive idempotent of $\Gamma$. Assume that the set $\lbrace E{\hat x} \vert x \in X\rbrace$ is Norton-balanced. Then
the kite function $\zeta_i$ is constant for $2 \leq i \leq D$.
\end{conjecture}

\section{Acknowledgement} 
The authors  thank  Jae-ho Lee for reading the manuscript carefully and offering helpful comments.
Many of the calculations in this paper were carried out using MAPLE.


\bigskip

\noindent Kazumasa Nomura \hfil\break
\noindent Tokyo Medical and Dental University \hfil\break
\noindent Kohnodai Ichikawa 272-0827 Japan \hfil\break
\noindent email: {\tt knomura@pop11.odn.ne.jp} \hfil\break

\noindent Paul Terwilliger \hfil\break
\noindent Department of Mathematics \hfil\break
\noindent University of Wisconsin \hfil\break
\noindent 480 Lincoln Drive \hfil\break
\noindent Madison, WI 53706-1388 USA \hfil\break
\noindent email: {\tt terwilli@math.wisc.edu }\hfil\break

\section{Statements and Declarations}

\noindent {\bf Funding}: The author declares that no funds, grants, or other support were received during the preparation of this manuscript.
\medskip

\noindent  {\bf Competing interests}:  The author  has no relevant financial or non-financial interests to disclose.
\medskip

\noindent {\bf Data availability}: All data generated or analyzed during this study are included in this published article.

\end{document}